\documentclass[11pt,reqno]{amsart}

\usepackage{fourier}
\usepackage{fullpage}
\usepackage{bbm}
\usepackage{enumerate}

\usepackage{amsmath,amsfonts,amssymb,mathrsfs}

\usepackage{graphicx,extpfeil}
\usepackage{mathtools,xcolor}

\usepackage{indentfirst, latexsym, enumerate}
\usepackage{float}
\usepackage{epsfig,subfigure}
\usepackage{csquotes}
\usepackage{url}

\usepackage{colortbl}
\usepackage{caption}
\usepackage{bm}
\usepackage{tikz}
\usetikzlibrary{matrix,shapes,arrows,positioning,chains}
\usepackage{esint}
\newcounter{name}

\title[Oscillator death in the Winfree model]{On oscillator death in the Winfree model}

\author[Ryoo]{Seung-Yeon Ryoo}
\address[Seung-Yeon Ryoo]{Hausdorff Center for Mathematics,
	\newline University of Bonn
	Endenicher Allee 60
	53115 Bonn, Germany}
\email{sryoo@uni-bonn.de}

\newtheorem{theorem}{Theorem}
\newtheorem{lemma}[theorem]{Lemma}
\newtheorem{corollary}[theorem]{Corollary}
\newtheorem{proposition}[theorem]{Proposition}

\newtheorem{remark}[theorem]{Remark}
\newtheorem{conjecture}[theorem]{Conjecture}
\newtheorem{question}[theorem]{Question}
\newtheorem{definition}[theorem]{Definition}

\newcommand{\bbr}{\mathbb R}

\newcommand{\bbt} {\mathbb T}

\newcommand{\bbn} {\mathbb N}

\begin{document}
\tikzstyle{block} = [rectangle, draw, 
    text width=15em, text centered, rounded corners, minimum height=3em]
\tikzstyle{line} = [draw, -latex']

\date{\today}


\begin{abstract}
We show that for the standard sinusoidal Winfree model, a coupling strength exceeding twice the maximal magnitude of the intrinsic frequencies guarantees the convergence of the system for Lebesgue almost every initial data towards the asymptotically stable equilibrium exhibited by Ha, Park, and the author (2015); in other words, on the phase torus, this equilibrium has basin of attraction of full Lebesgue measure. This is proven by first showing, via a new bootstrapping argument for the order parameter, that the pathwise critical coupling strength is upper bounded by a function of the order parameter, and then showing by a volumetric argument that for Lebesgue almost every data the order parameter cannot stay below and be bounded away from 1 for all time. Using concentration of measure and the aforementioned volumetric argument, we show that, except possibly on a set of very small measure, the solution enters a trapping regime in finite time. At sufficiently strong coupling, complete oscillator death therefore occurs for Lebesgue almost every initial configuration, in qualitative agreement with the observations of Ariaratnam and Strogatz (2001). These results are robust under many other choices of interaction functions often considered for the Winfree model. We demonstrate that the asymptotic dynamics described in this paper are sharp by analyzing the equilibria of the Winfree model, and we bound the total number of equilibria using a polynomial description.
\end{abstract}
\maketitle

\tableofcontents

\section{Introduction}\label{sec:summary}

\emph{Synchronization} refers to collective phenomena of many-body systems wherein some aspect, such as phase, position, or frequency, of the constituent particles or individuals coalesce. It is observed across a wide range of physical, biological, and engineering systems: after Huygens's seminal observation in the mid-seventeenth century on the synchronization of two pendulum clocks hanging on a common bar \cite{oliveira2015huygens}, synchronous phenomena have been often reported in scientific literature, examples including flocking in birds and fish \cite{flierl1999individuals,vicsek1995novel}, synchronous flashing of fireflies \cite{buck1966biology}, collective dynamics of cardiac pacemaker cells \cite{peskin1975mathematical}, and frequency synchronization in power networks \cite{chiang2011direct,kundur2007power,sauer2017power}.

However, the systematic study of synchronization based on mathematical models was initiated relatively recently. To name a few among the many mathematical models proposed to capture such behavior, there are the Winfree \cite{winfree1967biological} and Kuramoto \cite{kuramoto1975international} models of coupled oscillators, and the Vicsek and Cucker--Smale models for flocking \cite{cucker2007emergent,vicsek1995novel}. In this paper, we focus on the Winfree model and draw analogies with the Kuramoto model.

The novel feature of the Winfree and Kuramoto models is that they exhibit phase transition phenomena from disordered states (incoherent states) to partially ordered states (partially locked states) and then to fully ordered states (completely locked states), as the \emph{coupling strength}, a variable representing the amount of interaction between the individuals, exceeds certain critical thresholds \cite{ariaratnam2001phase,crawford1994amplitude,kuramoto2002coexistence}. Due to this phase transition feature, Winfree and Kuramoto's mathematical approach to synchronous phenomena, and more generally the theory of weakly coupled oscillators, has received attention from the control theory, neuroscience, and statistical physics research communities \cite{acebron2005kuramoto,bernoff2011primer,dorfler2011critical,ermentrout2019recent,ha2016collective,hoppensteadt2012weakly,rodrigues2016kuramoto,strogatz2000kuramoto}.

We now formally define the Winfree model.
The \textit{Winfree model} is given as the initial value problem
\begin{equation}\label{GenWinfree}
\begin{cases}
\dot{\theta}_i(t)=\omega_i+\frac \kappa N \sum_{j=1}^N  I(\theta_j(t))S(\theta_i(t)) ,\quad t>0,\\
\theta_i(0)=\theta_i^0,
\end{cases}
\quad i\in [N]\coloneqq \{1,\cdots,N\},
\end{equation}
for the $N\ge 1$ real variables $\theta_1,\cdots,\theta_N$, where $\omega_1,\cdots,\omega_N\in\mathbb{R}$ (the \textit{intrinsic frequencies}), $\kappa\in \mathbb{R}$ (the \textit{coupling strength}), and $\theta_1^0,\cdots,\theta_N^0\in \mathbb{R}$ (the \textit{initial data}) are fixed parameters, and $I:\mathbb{R}\to\mathbb{R}$ (the \textit{influence function}) and $S:\mathbb{R}\to\mathbb{R}$ (the \textit{sensitivity function})  are $2\pi$-periodic Lipschitz functions. Uniqueness and global existence of solutions to \eqref{GenWinfree} follow from the standard Cauchy-Lipschitz theory.

The purpose of this paper is to rigorously establish criteria for boundedness of solutions to \eqref{GenWinfree}, which we will refer to as the \textit{oscillator death} phenomenon.

In the prototypical case 
\begin{equation}\label{standard}
S(\theta)=-\sin\theta,\quad I(\theta)=1+\cos\theta,    
\end{equation}
the initial value problem \eqref{GenWinfree} becomes
\begin{equation}\label{Winfree}
\begin{cases}
\dot{\theta}_i(t)=\omega_i-\frac \kappa N \sum_{j=1}^N (1+\cos\theta_j(t))\sin \theta_i(t),\quad t>0,\\
\theta_i(0)=\theta_i^0,
\end{cases}
\quad i\in [N].
\end{equation}
We define the \emph{order parameter}
\begin{equation}\label{order_parameter}
R\coloneqq \frac 1N \sum_{j=1}^N (1+\cos\theta_j),
\end{equation}
with which one can rewrite the ordinary differential equation of \eqref{Winfree} in the mean-field form
\begin{equation}\label{Winfree_orderparam}
\dot{\theta}_i=\omega_i-\kappa R\sin \theta_i.
\end{equation}

Before stating the results, we first organize the notation.
We will use capital Greek letters to refer to the tuples of the corresponding lower Greek letters:
\[
\Omega\coloneqq (\omega_1,\cdots,\omega_N),\quad \Theta\coloneqq (\theta_1,\cdots,\theta_N),
\]
We will use $\|\cdot\|_\infty$ to denote the $\ell^\infty$-norm:
\[
\|\Omega\|_\infty\coloneqq \max_{i\in [N]} |\omega_i|,\quad \|\Theta\|_\infty\coloneqq \max_{i\in [N]} |\theta_i|.
\]
\begin{definition}
	Let $\Theta=(\theta_1,\cdots,\theta_N)$ be a solution to the Winfree model \eqref{GenWinfree}.
	We say that the ensemble $\Theta$ exhibits \emph{(complete) oscillator death} if
	\[
	\sup_{t\ge 0}\|\Theta(t)\|_\infty<\infty.
	\]
\end{definition}
\begin{remark}\label{rem:loja-cor}
	By Proposition \ref{Loja} and its following remarks, in the prototypical model \eqref{Winfree}, oscillator death for $\Theta$ implies the stronger condition
	\[
	\exists \lim_{t\to\infty}\theta_i(t)\mbox{ and } \lim_{t\to\infty}\dot{\theta}_i(t)=0,\quad \forall i\in [N].
	\]
\end{remark}

\begin{remark}
	Our notion of oscillator death is stronger than that usually used in the literature, which typically defines oscillator death to be the phenomenon where the \emph{rotation numbers} 
	\[
	\rho_i\coloneqq\lim_{t\to\infty} \frac{\theta_i(t)}{t},\quad i\in [N],
	\]
	defined whenever the limit exists, equal zero for all $i\in [N]$.
\end{remark}

A stable equilibrium for system \eqref{Winfree}, and hence oscillator death from restricted initial data, was established in \cite{ha2015emergence} as follows.
\begin{theorem}[{\cite[special case of Theorem 2.2]{ha2015emergence}}]\label{standardpreciseCOD}
	For $\alpha\in \left(\frac{\pi}{3},\pi\right)$, let $\alpha^\infty\in \left(0,\frac{\pi}{3}\right)$ be the unique solution to $\sin\alpha^\infty(1+\cos\alpha^\infty)=\sin\alpha(1+\cos\alpha)$.
	
	\begin{enumerate}[(a)]
		\item (Existence and uniqueness of equilibrium) Then system \eqref{Winfree} with parameters $\{\omega_i\}_{i=1}^N$ and $\kappa$ satisfying
		\[
		\kappa>\frac{\|\Omega\|_\infty}{\sin\alpha(1+\cos\alpha)}
		\]
		has a unique equilibrium $\Theta^\infty$ in $[-\alpha,\alpha]^N$.
		Furthermore, $\Theta^\infty\in (-\alpha^\infty,\alpha^\infty)^N$.
		\setcounter{name}{\value{enumi}}
	\end{enumerate}
	Moreover, let $\Theta(t)=\{\theta_i(t)\}_{i=1}^N$ be the solution to \eqref{Winfree} with initial data $\{\theta_i^0\}_{i=1}^N$ such that $\theta_i^0\in [-\alpha,\alpha]$, $i\in [N]$. Then $\Theta(t)\to \Theta^\infty$ exponentially as $t\to\infty$. More precisely,
	\begin{enumerate}[(a)]
		\setcounter{enumi}{\value{name}}
		\item (Finite-time entrance into stable region) there exists a time $T\le \frac{\pi}{\kappa\sin\alpha(1+\cos\alpha)-\|\Omega\|_\infty}$ such that $\Theta(t)\in (-\alpha^\infty,\alpha^\infty)^N$ for $t\ge T$, and we can take $T=0$ if $\Theta^0\in (-\alpha^\infty,\alpha^\infty)^N$; also,
		\item (Exponential convergence to equilibrium) we have that\footnote{The same estimate holds in every $\ell_p^N$ norm for $1\le p\le\infty$  \cite[Theorem~3.1~and~Theorem~4.1]{ha2026relaxation}.}
		\[
		\|\Theta(t)-\Theta^\infty\|_{\ell_1^N}\le \|\Theta(T)-\Theta^\infty\|_{\ell_1^N}\exp\left[-\kappa (2\cos\alpha^\infty-1)(\cos\alpha^\infty+1)(t-T)\right],\quad t\ge T.
		\]
	\end{enumerate}
\end{theorem}
In particular, if $\kappa>\frac{4}{3\sqrt{3}}\|\Omega\|_\infty$, we may find $\alpha\in (\frac{\pi}3,\pi)$ such that $\kappa>\frac{\|\Omega\|_\infty}{\sin\alpha(1+\cos\alpha)}$, and Theorem \ref{standardpreciseCOD} tells us that system \eqref{Winfree} has a stable equilibrium in $(-\frac{\pi}{3},\frac{\pi}{3})^N$. Taking $\alpha=\frac\pi 2$, if $\kappa>\|\Omega\|_\infty$ then \eqref{Winfree} has an asymptotically stable equilibrium in $(-\frac\pi 4,\frac\pi 4)^N$, which is the unique equilibrium in $[-\frac\pi 2,\frac\pi 2]^N$. This equilibrium has order parameter\footnote{Indeed, trivially $R\ge 1$ since it lies in $[-\frac\pi 2,\frac\pi 2]^N$. We note that, at the equilibrium, by \eqref{Winfree_orderparam} we have $\sin\theta_i=\frac{\omega_i}{\kappa R}$, so that $\cos\theta_i = \sqrt{1-\frac{\omega_i^2}{\kappa^2 R^2}}\ge  1-\frac{\|\Omega\|_\infty^2}{\kappa^2}$ and hence, by \eqref{order_parameter}, $R\ge 2-\frac{\|\Omega\|_\infty^2}{\kappa^2}$.}
\[
R\ge 2-\frac{\|\Omega\|^2_\infty}{\kappa^2}.
\]

The main result of this paper stated in the special case \eqref{standard} is the following theorem. It shows that the basin of attraction of the equilibrium of Theorem \ref{standardpreciseCOD} has full Lebesgue measure once $\kappa>2\|\Omega\|_\infty$. For more general interaction functions, see subsection \ref{subsec:interaction} for analogous results.

\begin{theorem}\label{themainthm}
Fix parameters $\Omega \in \mathbb{R}^N$ and $\kappa>2\|\Omega\|_\infty$. Denote by $\Theta^\infty\in (-\frac{\pi}{4},\frac{\pi}{4})^N$ the unique equilibrium in $[-\frac{\pi}{2},\frac{\pi}{2}]^N$ guaranteed by Theorem \ref{standardpreciseCOD}.
For Lebesgue almost every initial data $\Theta^0$, the solution $\Theta(t)$ to \eqref{Winfree} converges to $\Theta^\infty$ modulo $2\pi$: there is a vector $\vec{k}\in \mathbb{Z}^N$ such that
\[
\lim_{t\to\infty}\Theta(t)=\Theta^\infty+2\pi\vec{k},\quad \lim_{t\to\infty}\dot\Theta(t)=0.
\]
The convergence is exponential with rate $\kappa/\sqrt{2}$, namely,
\[
\|\Theta(t)-\Theta^\infty-2\pi\vec{k}\|_{\ell_p^N}\le \|\Theta(T)-\Theta^\infty-2\pi\vec{k}\|_{\ell_p^N}e^{-\kappa (t-T)/\sqrt{2}},\quad t\ge T
\]
for a large enough time $T$.
\end{theorem}

\begin{corollary}\label{liminfcor}
Fix parameters $\Omega \in \mathbb{R}^N$ and $\kappa>2\|\Omega\|_\infty$. For each initial data $\Theta^0\in \mathbb{R}^N$ let $\Theta(t)$ be the solution to \eqref{Winfree} with order parameter $R(t)$. Then
\[
\underset{\Theta^0\in \mathbb{R}^N}{\operatorname{ess~inf}}\liminf_{t\to\infty}R(t)\ge 2-\frac{\|\Omega\|_\infty^2}{\kappa^2},
\]
and consequently
\[
\lim_{\kappa\to\infty}\underset{\Theta^0\in \mathbb{R}^N}{\operatorname{ess~inf}}\liminf_{t\to\infty}R(t)=2.
\]
\end{corollary}

\begin{remark}\label{rem:exchange-order}
By Remark \ref{rem:all-negative-state} below, Corollary \ref{liminfcor} is false if we take $\inf$ instead of $\operatorname{ess~inf}$. In fact,
\[
\lim_{\kappa\to\infty}\inf_{\Theta^0\in \mathbb{R}^N}\limsup_{t\to\infty}R(t)=0.
\]
On the other hand, Theorem \ref{standardpreciseCOD} tells us that the variant  where we interchange the order of quantifiers is also affirmative:
\[
\underset{\Theta^0\in \mathbb{R}^N}{\operatorname{ess~inf}}\liminf_{\kappa\to\infty}\liminf_{t\to\infty}R(t)= 2.
\]
\end{remark}

For general interaction functions $I$ and $S$, we will show that as long as $I$ and $S$ mimic \eqref{standard} in a certain sense, then with a sufficiently large $\kappa$ we have oscillator death, i.e., boundedness of solutions, for Lebesgue almost every data. See subsection \ref{subsec:interaction} for the statements and further discussion.

Winfree proposed the model \eqref{GenWinfree} in \cite{winfree1967biological} to model the behavior of biological oscillators as linearly perturbative periodic solutions to \eqref{GenWinfree}, when $\kappa$ is very small (i.e. the system is ``weakly coupled'') and the $\omega_i$ are concentrated around a single positive value (i.e. the system has ``small bandwidth''). Ariaratnam and Strogatz \cite{ariaratnam2001phase} observed that in the mathematically tractable special case \eqref{standard}, which is representative of the behavior of some biological oscillators, if one strengthens the parameters $\kappa$ and widens the distribution of the $\omega_i$'s, the system exhibits a variety of asymptotic behavior. In particular, if $\kappa$ is sufficiently large compared to the distribution of the $\omega_i$'s, then their numerical simulations suggest that $\lim_{t\to\infty}\theta_i(t)/t=0$ for all $i\in [N]$, regardless of the initial value.

Here is the exact setup of \cite{ariaratnam2001phase}. They take $N=800$, a sample of $\theta_i^0\in [-\pi,\pi]$, a (nonrandom) sample of $\omega_i$ from the uniform distribution on $[1-\gamma,1+\gamma]$, and some fixed $\kappa\ge 0$, and then simulate \eqref{Winfree} with these parameters up to time $t=500$ (sometimes longer times were required). Then they observed that the qualitative profile of the ``rotation numbers''
\[
\rho_i\coloneqq\lim_{t\to\infty}\theta_i(t)/t,\quad i\in [N],
\]
does not depend on the initial conditions $\theta_i^0$, and they drew a phase diagram depicting the configuration of the rotation numbers for different choices of the parameters $\kappa\in [0,1]$ and $\gamma\in [0,1]$. The large $\kappa$ regime of this phase diagram corresponds to the ``death'' regime, where $\rho_i=0$ for all $i\in [N]$.

Note that our notion of oscillator death implies $\lim_{t\to\infty}\theta_i(t)/t=0$ for all $i\in [N]$. Theorem \ref{themainthm} implies the oscillator death phenomenon for the model \eqref{Winfree} when $\kappa>2\|\Omega\|_\infty$ for Lebesgue a.e.~data, in accordance with the numerical observation of the death regime by Ariaratnam and Strogatz.

To be more precise, in the following theorem, we can compute an upper bound on the measure of the exceptional set for which a phase range below $2\pi$ from time $t=0$ is not guaranteed. This is a quantitative version of Theorem \ref{themainthm}.

\begin{theorem}\label{sincosmaincor}
Fix parameters $\Omega\in \mathbb{R}^N$, $\kappa>0$, and $\varepsilon\in (0,1]$ such that
\[
\kappa\ge (2+\varepsilon)\|\Omega\|_\infty.
\]
For each initial data $\Theta^0$, denote by $\Theta(t)$ the solution to \eqref{Winfree}.
Then, denoting by $m$ the normalized Lebesgue measure on $[-\pi,\pi]^N$ such that $m([-\pi,\pi]^N)=1$, one has\footnote{By Theorem \ref{themainthm} we may harmlessly insert the condition ``$\exists \lim_{t\to\infty}\theta_i(t)$ and $\lim_{t\to\infty}\dot\theta_i(t)$ for all $i\in [N]$'' into the events below.}
\[
m\left\{\Theta^0\in[-\pi,\pi]^N:
\begin{array}{l}
	R(t)\ge1/\sqrt2-3\varepsilon/20\quad\text{for all }t\ge0,\\
	\sup_{t\ge0}\theta_i(t)-\inf_{t\ge0}\theta_i(t)<2\pi\quad\text{for all }i\in[N]
\end{array}\right\}
\ge1-e^{-\varepsilon^2N/25}.
\]
If, in addition,
$
\kappa>\left(\frac{\pi e}{2}\right)^{3/2}\|\Omega\|_\infty\approx 8.8231 \|\Omega\|_\infty,
$
then
\[
    m\left\{\{\theta_i^0\}_{i=1}^N\in [-\pi,\pi]^N: \sup_{t\ge 0} \theta_i(t)-\inf_{t\ge 0}\theta_i(t)<2\pi ~\mbox{for all } i\in [N]\right\}\ge 1-\left(\sqrt{\frac{\pi e}{2}}\frac{\|\Omega\|_\infty^{1/3}}{\kappa^{1/3}}\right)^N.
\]
\end{theorem}

Furthermore, for any time $T>0$, we can compute an upper bound on the measure of the exceptional set for which a phase range below $2\pi$ on $[T,\infty)$ is not guaranteed, strengthening Theorem \ref{sincosmaincor}.

\begin{theorem}\label{sincosmaincor-time}
Fix parameters\footnote{The reason we don't consider the case $N=1$ is because oscillator death is trivial in this case; see Proposition \ref{verytrivial} with $N=1$.} $N\ge 2$, $\Omega \in \mathbb{R}^N$, $\kappa>0$, and $\varepsilon\in (0,1]$ such that
\[
\kappa\ge (2+\varepsilon)\|\Omega\|_\infty.
\]
For each initial data $\{\theta_i^0\}_{i=1}^N$, denote by $\{\theta_i(t)\}_{i=1}^N$ the solution to \eqref{Winfree}.
Then, denoting by $m$ the normalized Lebesgue measure on $[-\pi,\pi]^N$ such that $m([-\pi,\pi]^N)=1$, one has, for each time $T>0$,
\begin{align*}
&m\left\{\Theta^0\in[-\pi,\pi]^N:
\begin{array}{l}
	R(t)\ge1/\sqrt2-3\varepsilon/20\quad\text{for all }t\ge T,\\
	\sup_{t\ge T}\theta_i(t)-\inf_{t\ge T}\theta_i(t)<2\pi\quad\text{for all }i\in[N]
\end{array}\right\}\\
&\ge 
\begin{cases}
1-\frac 1{N/2+1} \left(\frac{\sqrt{\pi e \varepsilon}}{2\sqrt{5}}\right)^N\left(1-\exp\left(-\frac{\kappa N \varepsilon(5-\varepsilon)T}{25}\right)\right)-\exp\left(-\frac{N\varepsilon^2}{25}-\frac{\kappa N \varepsilon(5-\varepsilon)T}{25}\right),& \mathrm{if~}T\le T_0,\\
1-\left(\frac{20}{\pi e\varepsilon}+\frac{4\kappa(5-\varepsilon)}{5\pi e}(T-T_0)\right)^{-\frac{N}{2}},&\mathrm{if~}T>T_0,
\end{cases}
\end{align*}
where $T_0 = \frac{25}{\kappa N\varepsilon(5-\varepsilon)}\log\left(\left(1+\frac 2N\right)\left(\frac{2\sqrt{5}}{\sqrt{\pi \varepsilon}e^{1/2+\varepsilon^2/25}}\right)^N-\frac 2N\right)$.

Furthermore, if
$
\kappa>8\|\Omega\|_\infty,
$
then\footnote{If $\|\Omega\|_\infty=0$ so that a fraction in the expression below is undefined, then the stated probability is 1.}
\begin{align*}
    &m\left\{\{\theta_i^0\}_{i=1}^N\in [-\pi,\pi]^N: \sup_{t\ge T} \theta_i(t)-\inf_{t\ge T}\theta_i(t)<2\pi~\mbox{for all } i\in [N]\right\}\\
&\ge 
\begin{cases}
1-\left(e^{1/2}+\frac{3\kappa T}{\pi e}\right)^{-\frac{N}{2}}&\mathrm{if~}8\|\Omega\|_\infty<\kappa\le 16\sqrt{2}\|\Omega\|_\infty,\\
1-\left(\frac{4}{\pi e}\left(\frac{\kappa}{2\sqrt{2}\|\Omega\|_\infty}\right)^{2/3}+\frac{3\kappa T}{\pi e}\right)^{-\frac{N}{2}}&\mathrm{if~}\kappa>16\sqrt{2}\|\Omega\|_\infty.
\end{cases}
\end{align*}
\end{theorem}

Theorem \ref{themainthm} gives complete oscillator death for Lebesgue almost every initial configuration. Theorems \ref{sincosmaincor} and \ref{sincosmaincor-time} quantify how likely the system is to have entered the trapping regime by a given time. Namely, Theorem \ref{sincosmaincor} tells us that if $\kappa>(2+\varepsilon)\|\Omega\|_\infty$, then if we independently sample $\theta_i^0$ uniformly from $[-\pi,\pi]$ and simulate \eqref{Winfree}, we will then observe that the oscillators stay within $2\pi$ of its initial value for all time with probability at least $1-\exp\left(-\frac{\varepsilon^2 N}{25}\right)$, which tends rapidly to $1$ as $N\to\infty$. For example, if we have $N=800$ and $\kappa>3\|\Omega\|_\infty $, then this probability is at least $1-\exp\left(-800/25\right)\approx 1-1.266\times 10^{-14}$. Theorem \ref{sincosmaincor-time} goes one step further and tells us that the solution will have entered the stated trapping regime by time $T$ with probability at least $1-\left(\frac{20}{\pi e\varepsilon}+\frac{4\kappa(5-\varepsilon)}{5\pi e}(T-T_0)\right)^{-\frac{N}{2}}$ if $T$ is large enough. For example, if $N=800$, $\|\Omega\|_\infty=2$, $\kappa=6$ (so that $\varepsilon=1$), and $T=500$, then $T_0 = \frac{25}{4\kappa N}\log\left(\left(1+\frac 2N\right)\left(\frac{2\sqrt{5}}{\sqrt{\pi }e^{1/2+1/25}}\right)^N-\frac 2N\right)=0.40156699\cdots$ and the probability of entering the stated trapping regime by time $T=500$ is at least
\[
1-\left(\frac{20}{\pi e}+\frac{16\kappa}{5\pi e}(T-T_0)\right)^{-\frac{N}{2}}\approx 1-2.7990\times 10^{-1221}.
\]

Although the statements of Theorems \ref{themainthm} and \ref{sincosmaincor} treat $2 \|\Omega\|_\infty$ as if it were a critical threshold, this is an artifact\footnote{Indeed, incorporating the analysis of \cite{ha2020asymptotic} on the time derivative of the order parameter can slightly lower the level of required $\kappa$ from $2\|\Omega\|_\infty$. Further improvements can come depending on the distribution of the $\omega_i$'s. However, we won't write out the details since we cannot reach the bona fide critical coupling strength even with these methods.} of our proof, as the phase diagram of \cite{ariaratnam2001phase} indicates that oscillator death occurs even for smaller values of $\kappa$. To be precise, we make the following definitions.

\begin{definition}\label{def:crit}
Fix a frequency vector $\Omega=\{\omega_1,\cdots,\omega_N\}\in \mathbb{R}^N$.
\begin{enumerate}[(a)]
\item The \emph{critical coupling strength} is
\[
\kappa_{\mathrm{c}}(\Omega)\coloneqq \inf\{\kappa_*>0:\mathrm{for~}\kappa>\kappa_*,\mathrm{~system~}\eqref{Winfree}\mathrm{~admits~an~equilibrium}\}.
\]
\item For an initial phase vector $\Theta^0=\{\theta^0_1,\cdots,\theta^0_N\}\in \mathbb{R}^N$, the \emph{pathwise critical coupling strength} is
\[
\kappa_{\mathrm{pc}}(\Theta^0,\Omega)\coloneqq \inf\{\kappa_*>0:\mathrm{for~}\kappa>\kappa_*,\mathrm{~the~solution~}\{\theta_i(t)\}_{i=1}^N\mathrm{~to~}\eqref{Winfree}\mathrm{~converges}\}.
\]
\end{enumerate}
\end{definition}

We have a fortiori $\kappa_{\mathrm{c}}(\Omega)\le \kappa_{\mathrm{pc}}(\Theta^0,\Omega)$, while by Proposition \ref{prop:crit_comp} we have 
\begin{equation}\label{eq:crit-range}
\frac{2N}{4N-1}\|\Omega\|_\infty\le \kappa_{\mathrm{c}}(\Omega)\le \frac{4}{3\sqrt{3}}\|\Omega\|_\infty.
\end{equation}

Theorem \ref{themainthm} also gives the following almost-everywhere statement in the coupling parameter:
\begin{theorem}\label{thm:pc2}
	Fix $\Omega\in\mathbb R^N$. For Lebesgue almost every initial datum $\Theta^0\in\mathbb R^N$, the solution to \eqref{Winfree} converges (modulo $2\pi$) to the unique equilibrium in $(-\frac{\pi}{4},\frac{\pi}{4})^N$ of Theorem \ref{standardpreciseCOD} for Lebesgue almost every coupling strength $\kappa>2\|\Omega\|_\infty$.
	
	Moreover, for any prescribed countable set $\Lambda\subset(2\|\Omega\|_\infty,\infty)$, convergence holds simultaneously for every $\kappa\in\Lambda$ outside a single null set of initial data.
\end{theorem}
\begin{proof}
	The solution depends continuously on $(\Theta^0,\kappa)$ at every finite time. The set of pairs for which the trajectory enters the trapping set $(-\frac{\pi}{4},\frac{\pi}{4})^N$ modulo $2\pi$ (and hence converges to the unique equilibrium) is Borel, namely consider
	\[
	\bigcup_{n\in \mathbb{N}} \left\{(\Theta^0,\kappa):\Theta(n)\in (-\frac{\pi}{4},\frac{\pi}{4})^N \mod 2\pi\right\}.
	\]
	On each slice given by a fixed $\kappa$, this is a full-measure set; we have the first conclusion by Fubini's theorem.
	
	The second conclusion follows by taking a countable union of exceptional null sets.
\end{proof}
\noindent Theorem \ref{thm:pc2} does not imply
$\kappa_{\mathrm{pc}}(\Theta^0,\Omega)\le2\|\Omega\|_\infty$ for almost every $\Theta^0$: that stronger assertion requires convergence for every coupling above the threshold, and would need an additional uniform-in-$\kappa$ argument.

\begin{remark}\label{rem:pc-sup}
In general, we have
\[
\sup_{\Theta^0,\Omega\in \mathbb{R}^N,\|\Omega\|_\infty\le M}\kappa_{\mathrm{pc}}(\Theta^0,\Omega)=\frac{4N}{3\sqrt{3}}M.
\]
Indeed, the upper bound follows from Proposition \ref{verytrivial} and from $\max_{\theta\in [-\pi,\pi]}(1+\cos\theta)\sin\theta=\frac{3\sqrt{3}}{4}$. For the lower bound, let $\Omega=(M,0,\cdots,0)$ and $\Theta^0=(\pi,\cdots,\pi)$. Then $\theta_i(t)=\pi$ for $i=2,\cdots,N$ and $t\ge 0$, so $R=\frac{1+\cos\theta_1}{N}$ and $\dot\theta_1=M-\frac{\kappa}{N}(1+\cos\theta_1)\sin\theta_1$. Again from $\max_{\theta\in [-\pi,\pi]}(1+\cos\theta)\sin\theta=\frac{3\sqrt{3}}{4}$, $\theta_1(t)\to\infty$ as $t\to\infty$ if $\kappa<\frac{4N}{3\sqrt{3}}M$. 

When $N\ge 3$, the stated universal bound is $>2M$, but these configurations form a null set and do not violate Theorem \ref{themainthm}.
\end{remark}

Theorem \ref{themainthm} and \eqref{eq:crit-range} show that the sufficient coupling threshold for convergence from almost every initial datum is within a universal factor of the threshold for existence of an equilibrium, although the pathwise threshold has the stronger quantifier ``for every larger coupling,'' and is not supplied by Theorem \ref{thm:pc2}. Nevertheless, the phase diagram of \cite{ariaratnam2001phase} motivates the following stronger conjecture.

\begin{conjecture}
Let the frequency vector $\Omega\in \mathbb{R}^N$ be fixed. Then $\kappa_{\mathrm{pc}}(\Theta^0,\Omega)=\kappa_c(\Omega)$ for a.e.~$\Theta^0\in \mathbb{R}^N$.
\end{conjecture}

It seems likely that the results obtained here for the ODE level carry on to the infinite-$N$ limits. See Ha, Kang, and Moon \cite{ha2021uniform} for a description of the continuum limit of \eqref{Winfree}. For related relaxation results on the sphere $\mathbb{S}^{d-1}$ and the unitary group $\mathrm{U}(d)$, see \cite{ha2026relaxation}.

The results stated above extend to a larger class of interaction functions $I$ and $S$ for \eqref{GenWinfree}; see subsection \ref{subsec:interaction} for the statements. Also, many of the results of this paper can be extended to the manifold setting, particularly the Euclidean sphere $\mathbb{S}^{d-1}$ and the unitary group $\mathrm{U}(d)$. We will deal with this in a sequel to this paper.


\bigskip
\noindent{\bf Roadmap.} 
The rest of this paper is organized as follows. In Section \ref{sec:gallery}, we list our new technical results. In subsection \ref{subsec:orderparam} we will state a key intermediate Corollary \ref{bestsincosmainthm} in proving Theorems \ref{themainthm} and \ref{sincosmaincor}. For more general interaction functions $I$ and $S$, we will not be able to conclude the convergence of system \eqref{GenWinfree} unless under very restrictive conditions, and to make this point clear, we will discuss the {\L}ojasiewicz gradient theorem, which is a powerful tool for discussing convergence of \eqref{Winfree}, in subsection \ref{subsec:loja}. In subsection \ref{subsec:interaction}, we will see how to extend the statements of Theorems \ref{themainthm} and \ref{sincosmaincor} and Corollary \ref{bestsincosmainthm} to more general interaction functions $I$ and $S$. After that, in subsection \ref{subsec:partialdeath} we will formalize the notion of partial oscillator death and provide criteria for it to occur, and in subsection \ref{subsec:equilibria} we will discuss the equilibria of \eqref{Winfree_orderparam} to show the sharpness of the asymptotic dynamics of Corollary \ref{bestsincosmainthm}. 

In Section \ref{sec:history}, we briefly review the relevant literature on the Winfree model, starting from Winfree's original paper \cite{winfree1967biological} and describing the subsequent developments. We will also compare the Winfree model to the Kuramoto model, which was inspired by the Winfree model and which exhibits analogous but different synchronization behavior. We will suggest some lines of possible future research along the way.

Section \ref{sec:bootstrap} is devoted to proving Corollary \ref{bestsincosmainthm} and its variant Theorem \ref{mainthm} via a bootstrapping argument which also gives the sufficient conditions of subsection \ref{subsec:partialdeath} that guarantee partial oscillator death.
 
 In Section \ref{sec:large_deviations}, we calculate the divergence of the vector field associated to system \eqref{GenWinfree} and use an argument on the rate of change of volume to prove Theorem \ref{themainthm}. By applying some large deviations theory and capitalizing on the positivity of the divergence, we will prove Theorems \ref{sincosmaincor} and \ref{sincosmaincor-time}, plus the many other new theorems mentioned in Section \ref{sec:gallery}.

 In Section \ref{sec:equilibria}, we prove the statements of subsection \ref{subsec:equilibria}, namely we obtain some existence results for the equilibria of \eqref{Winfree} and demonstrate a polynomial description of the equilibria of \eqref{Winfree} that gives us a uniform bound on the total number of equilibria. We also compute the critical coupling strength $\kappa_{\mathrm{c}}$.
 
 Finally, in Appendix \ref{app:ineq} we prove some elementary calculus inequalities that are needed in this paper.


\section{Gallery of new results}\label{sec:gallery}
In this section, we outline our new results regarding the models \eqref{GenWinfree} and \eqref{Winfree}.
\subsection{The pathwise critical coupling strength is bounded by a function of the order parameter}\label{subsec:orderparam}
The first step in the proof of Theorem \ref{themainthm} is the following intermediate result on the convergence of system \eqref{Winfree}.
\begin{theorem}\label{generalbestsincosmainthm}
Let $\Theta=\Theta(t)$ be the solution to \eqref{Winfree} with parameters $\{\omega_i\}_{i=1}^N$, $\kappa$, and $\{\theta_i^0\}_{i=1}^N$.  Suppose ${R_0}>0$, fix $\mu\in (0,\min\{R_0,1\})$, and suppose
\begin{equation}\label{large_kappa_general}
\kappa> \frac{\|\Omega\|_\infty}{({R_0}-\mu) \sqrt{\mu(2-\mu)}}.
\end{equation}
Then
\begin{enumerate}[(a)]
    \item $R(t)\ge R_0-\mu$ for all $t\ge 0$,
    \item For all $i\in [N]$, the limit $\lim_{t\to\infty}\theta_i(t)$ exists, and
\[
\sup_{t\ge 0} \theta_i(t)-\inf_{t\ge 0}\theta_i(t)<2\pi,\quad \lim_{t\to\infty}\dot{\theta}_i(t)=0.
\]
    \item for any $i\in [N]$ and $t_0\ge 0$ with
$\cos\theta_i(t_0)\ge -1+\mu$,
we have $\cos \theta_i(t)\ge -1+\mu$ for $t \ge t_0$ and $\cos \theta_i(t)\ge 1-\mu$ for  $t\ge t_0+\frac{\pi}{\kappa (R_0-\mu) \sqrt{\mu(2-\mu)}-\|\Omega\|_\infty}$,
\item for any $i\in [N]$, either $\sup_{t\ge 0}\cos\theta_i(t)<-1+\mu$, in which case $\lim_{t\to\infty}\theta_i(t)=\pi-\sin^{-1}\frac{\omega_i}{\kappa \lim_{t\to\infty}R(t)}\mod 2\pi$, or $\sup_{t\ge 0}\cos\theta_i(t)\ge -1+\mu$, in which case $\lim_{t\to\infty}\theta_i(t)=\sin^{-1}\frac{\omega_i}{\kappa \lim_{t\to\infty}R(t)}\mod 2\pi$.
\item For fixed $i,j\in [N]$, suppose there is a time $t_0\ge 0$ such that $\cos \theta_i(t_0)\ge 1-\mu$ and $\cos \theta_j(t_0)\ge 1-\mu$. If $\omega_i>\omega_j$, then after modulo $2\pi$ translations,
\begin{align*}
\theta_i(t)-\theta_j(t)\le \frac{\omega_i-\omega_j}{\kappa(R_0-\mu)(1-\mu)}+\pi\exp\left(-\kappa(R_0-\mu)(1-\mu)\left(t-t_0\right)\right)\quad 
\mathrm{for~}t\ge t_0,
\end{align*}
and
\begin{align*}
\theta_i(t)-\theta_j(t)\ge \frac{\omega_i-\omega_j}{2\kappa}-\frac{\omega_i-\omega_j}{2\kappa}\exp\left(-2\kappa\left(t-t_0-\frac{\pi}{\omega_i-\omega_j}\right)\right)\quad
\mathrm{for~}t\ge t_0+\frac{\pi}{\omega_i-\omega_j}.
\end{align*}
If $\omega_i=\omega_j$, then after modulo $2\pi$ translations, $\lim_{t\to\infty}\left(\theta_i(t)-\theta_j(t)\right)=0$ exponentially:
\begin{align*}
|\theta_i(t)-\theta_j(t)|\le \pi \exp\left(-\kappa(R_0-\mu)(1-\mu)\left(t-t_0\right)\right)\quad \mathrm{ for~}t\ge t_0.
\end{align*}
\end{enumerate}
\end{theorem}

Applying special values of $\mu$ gives the following Corollary.

\begin{corollary}\label{bestsincosmainthm}
Let $\Theta=\Theta(t)$ be the solution to \eqref{Winfree} with parameters $\{\omega_i\}_{i=1}^N$, $\kappa$, and $\{\theta_i^0\}_{i=1}^N$.  Suppose ${R_0}>0$ and
\begin{equation}\label{large_kappa_1}
\kappa>K_c\|\Omega\|_\infty,
\end{equation}
where
\begin{equation}\label{large_kappa_2}
K_c\coloneqq 
\begin{cases}
\frac{2}{{R_0}^{3/2}},&\mbox{if }{R_0}\le 1,\\
2(2-R_0)+\frac{4}{3\sqrt{3}}(R_0-1),&\mbox{if }1<{R_0}\le 2.
\end{cases}
\end{equation}
Then
\begin{enumerate}[(a)]
    \item $R(t)\ge \frac{3{R_0}-3+\sqrt{{R_0}^2-2{R_0}+9}}{4}\ge \frac 23 {R_0}$ for all $t\ge 0$,
    \item For all $i\in [N]$, the limit $\lim_{t\to\infty}\theta_i(t)$ exists, and
\[
\sup_{t\ge 0} \theta_i(t)-\inf_{t\ge 0}\theta_i(t)<2\pi,\quad \lim_{t\to\infty}\dot{\theta}_i(t)=0.
\]
    \item for any $i\in [N]$ and $t_0\ge 0$ with
$\cos\theta_i(t_0)\ge \frac{-1+{R_0}-\sqrt{{R_0}^2-2{R_0}+9}}{4}$,
we have $\cos \theta_i(t)\ge\frac{-1+{R_0}-\sqrt{{R_0}^2-2{R_0}+9}}{4}$ for $t \ge t_0$ and $\cos \theta_i(t)\ge \frac{1-{R_0}+\sqrt{{R_0}^2-2{R_0}+9}}{4}$ for  $t\ge t_0+\frac{\pi K_c}{\kappa-K_c\|\Omega\|_\infty}$.
\item for any $i\in [N]$, either $\sup_{t\ge 0}\cos\theta_i(t)<\frac{-1+{R_0}-\sqrt{{R_0}^2-2{R_0}+9}}{4}$, in which case
\[
\lim_{t\to\infty}\theta_i(t)=\pi-\sin^{-1}\frac{\omega_i}{\kappa \lim_{t\to\infty}R(t)}\mod 2\pi,
\]
or $\sup_{t\ge 0}\cos\theta_i(t)\ge \frac{-1+{R_0}-\sqrt{{R_0}^2-2{R_0}+9}}{4}$, in which case
\[
\lim_{t\to\infty}\theta_i(t)=\sin^{-1}\frac{\omega_i}{\kappa \lim_{t\to\infty}R(t)}\mod 2\pi.
\]
    \item For fixed $i,j\in [N]$, suppose there is a time $t_0\ge 0$ such that $\cos \theta_i(t_0)\ge 1-\frac{R_0}{3}$ and $\cos \theta_j(t_0)\ge 1-\frac{R_0}{3}$. If $\omega_i>\omega_j$, then after modulo $2\pi$ translations,
\begin{align*}
\theta_i(t)-\theta_j(t)\le \frac{9(\omega_i-\omega_j)}{2\kappa R_0(3-R_0)}+\pi\exp\left(-\frac{2\kappa R_0(3-R_0)}{9}\left(t-t_0\right)\right)\quad
\mathrm{for~}t\ge t_0,
\end{align*}
and
\begin{align*}
\theta_i(t)-\theta_j(t)\ge \frac{\omega_i-\omega_j}{2\kappa}-\frac{\omega_i-\omega_j}{2\kappa}\exp\left(-2\kappa\left(t-t_0-\frac{\pi}{\omega_i-\omega_j}\right)\right)\quad 
\mathrm{for~}t\ge t_0+\frac{\pi}{\omega_i-\omega_j}.
\end{align*}
If $\omega_i=\omega_j$, then after modulo $2\pi$ translations, $\lim_{t\to\infty}\left(\theta_i(t)-\theta_j(t)\right)=0$ exponentially:
\begin{align*}
|\theta_i(t)-\theta_j(t)|\le \pi \exp\left(-\frac{2\kappa R_0(3-R_0)}{9}\left(t-t_0\right)\right)\mathrm{~for~}t\ge t_0.
\end{align*}
\setcounter{name}{\value{enumi}}
\end{enumerate}
For the following refinements assume $\|\Omega\|_\infty>0$. If
$
\kappa>\frac{2\sqrt{2}\|\Omega\|_\infty}{R_0^{3/2}},
$
then
\begin{enumerate}[(a)]
\setcounter{enumi}{\value{name}}
    \item $R(t)\ge R_0-\frac{4\|\Omega\|_\infty^2}{\kappa^2R_0^2}$ for all $t\ge 0$,
    \item for any $i\in [N]$ and $t_0\ge 0$ with
    $\cos\theta_i(t_0)\ge -1+\frac{4\|\Omega\|_\infty^2}{\kappa^2R_0^2}$, we have $\cos \theta_i(t)\ge-1+\frac{4\|\Omega\|_\infty^2}{\kappa^2R_0^2}$ for $t \ge t_0$ and $\cos \theta_i(t)\ge 1-\frac{4\|\Omega\|_\infty^2}{\kappa^2R_0^2}$ for  $t\ge t_0+\frac{\pi}{\|\Omega\|_\infty\left(1-\frac{8\|\Omega\|_\infty^2}{\kappa^2R_0^3}\right)}$.
    \item for any $i\in [N]$, either $\sup_{t\ge 0}\cos\theta_i(t)<-1+\frac{4\|\Omega\|_\infty^2}{\kappa^2R_0^2}$, in which case $\lim_{t\to\infty}\theta_i(t)=\pi-\sin^{-1}\frac{\omega_i}{\kappa \lim_{t\to\infty}R(t)}\mod 2\pi$, or $\sup_{t\ge 0}\cos\theta_i(t)\ge -1+\frac{4\|\Omega\|_\infty^2}{\kappa^2R_0^2}$, in which case $\lim_{t\to\infty}\theta_i(t)=\sin^{-1}\frac{\omega_i}{\kappa \lim_{t\to\infty}R(t)}\mod 2\pi$.
    \item For fixed $i,j\in [N]$, suppose there is a time $t_0\ge 0$ such that $\cos \theta_i(t_0)\ge 1-\frac{4\|\Omega\|_\infty^2}{\kappa^2R_0^2}$ and $\cos \theta_j(t_0)\ge 1-\frac{4\|\Omega\|_\infty^2}{\kappa^2R_0^2}$. If $\omega_i>\omega_j$, then after modulo $2\pi$ translations,
\begin{align*}
\theta_i(t)-\theta_j(t)\le \frac{2(\omega_i-\omega_j)}{\kappa R_0\left(1-\frac{4\|\Omega\|_\infty^2}{\kappa^2R_0^2}\right)}+\pi\exp\left(-\frac{\kappa R_0}{2}\left(1-\frac{4\|\Omega\|_\infty^2}{\kappa^2R_0^2}\right)\left(t-t_0\right)\right)\quad
\mathrm{for~}t\ge t_0,
\end{align*}
and
\begin{align*}
\theta_i(t)-\theta_j(t)\ge \frac{\omega_i-\omega_j}{2\kappa}-\frac{\omega_i-\omega_j}{2\kappa}\exp\left(-2\kappa\left(t-t_0-\frac{\pi}{\omega_i-\omega_j}\right)\right)\quad
\mathrm{for~}t\ge t_0+\frac{\pi}{\omega_i-\omega_j}.
\end{align*}
If $\omega_i=\omega_j$, then after modulo $2\pi$ translations, $\lim_{t\to\infty}\left(\theta_i(t)-\theta_j(t)\right)=0$ exponentially:
\begin{align*}
|\theta_i(t)-\theta_j(t)|\le \pi \exp\left(-\frac{\kappa R_0}{2}\left(1-\frac{4\|\Omega\|_\infty^2}{\kappa^2R_0^2}\right)\left(t-t_0\right)\right)\quad \mathrm{ for~}t\ge t_0.
\end{align*}
\end{enumerate}
\end{corollary}

Simply put, we have the following, which complements Theorem \ref{thm:pc2}.
\begin{theorem}
Let $R_0=R(\Theta^0)>0$. Then
\[
\kappa_{\mathrm{pc}}(\Theta^0,\Omega)\le  \begin{cases}
\frac{2}{{R_0}^{3/2}}\|\Omega\|_\infty,&\mbox{if }{R_0}\le 1,\\
[2(2-{R_0})+\frac{4}{3\sqrt{3}}({R_0}-1)]\|\Omega\|_\infty,&\mbox{if }1<{R_0}\le 2.
\end{cases}
\]
\end{theorem}

Theorems \ref{themainthm} and \ref{sincosmaincor} are derived from Corollary \ref{bestsincosmainthm} in two different ways. For the first estimate in Theorem \ref{sincosmaincor}, concentration of the initial order parameter near $1$ gives
\[
m\{R_0<1-\varepsilon/5\}\le\exp(-\varepsilon^2N/25).
\]
The second estimate in that theorem follows from the small-value bound of Lemma \ref{lem:order-param-conc}, applied to the initial exceptional set
\[
\left\{R_0\le\left(\frac{2\sqrt2\|\Omega\|_\infty}{\kappa}\right)^{2/3}\right\}.
\]
This yields the factor $\left(\sqrt{\pi e/2}\,\|\Omega\|_\infty^{1/3}/\kappa^{1/3}\right)^N$. The volumetric argument shows that, for $\kappa>0$ and almost every initial datum, $R(t)$ eventually reaches any prescribed level below $1$. Restarting the bootstrap at such a time proves Theorem \ref{themainthm} for $\kappa>2\|\Omega\|_\infty$, while quantitative volume estimates give the improvement in time in Theorem \ref{sincosmaincor-time}.

The value $2\|\Omega\|_\infty$ for $\kappa$ comes from the case $R_0=1$ of \eqref{large_kappa_2}, but note that in the endpoint case $R_0=2$ of \eqref{large_kappa_2}, the value $\frac{4}{3\sqrt{3}}\|\Omega\|_\infty$ for $\kappa$ is the maximum of $\kappa_c(\Omega)$ ranging over all $\Omega$ with a fixed $\|\Omega\|_\infty$. In other words, for a fixed value of $\|\Omega\|_\infty$, $\kappa_{\mathrm{crit}}\coloneqq \frac{4}{3\sqrt{3}}\|\Omega\|_\infty$ is another critical threshold for the existence of equilibria in the sense that if $\kappa>\kappa_\mathrm{crit}$ then  \eqref{Winfree}, by Theorem \ref{standardpreciseCOD}, has an asymptotically stable equilibrium, while if $\kappa<\kappa_\mathrm{crit}$ and $\omega_1=\cdots=\omega_N=\|\Omega\|_\infty$ then not only does \eqref{Winfree} fail to have any equilibria, but there is a positive measure of solutions to \eqref{Winfree} satisfying $\theta_i(t)\to\infty$ as $t\to \infty$, $i\in [N]$ \cite{ha2017emergence,oukil2017synchronization}. Currently, Corollary \ref{bestsincosmainthm} tells us that when $R_0=\Omega(1)$, $\kappa>\kappa_\mathrm{crit}+\Omega(2-R_0)$ (which approaches the threshold $\kappa_\mathrm{crit}$ as $R_0$ approaches $2$) guarantees the oscillator death phenomenon. See subsubsection \ref{subsubsec:death} for a detailed discussion of possible future work in this direction.

\subsection{Application of the {\L}ojasiewicz gradient theorem}\label{subsec:loja}
For $S=I'$, the Winfree model on a symmetric network has the gradient structure recorded in \cite[Proposition~2.4]{ha2016collective}. For the all-to-all coupling in \eqref{GenWinfree}, the potential on the universal cover is
\[
V(\Theta)\coloneqq -\sum_{i=1}^N \omega_i \theta_i-\frac \kappa {2N} \left[\sum_{i=1}^N I(\theta_i)\right]^2,
\]
Its gradient gives the vector field through
\[
\dot{\Theta}=-\nabla_\Theta V(\Theta).
\]
The {\L}ojasiewicz gradient theorem \cite{lojasiewicz1963propriete,lojasiewicz1982trajectoires} then gives convergence of bounded trajectories for analytic interactions.
\begin{proposition}\label{Loja}
Let $\Theta=\Theta(t)$ be the solution to \eqref{GenWinfree} with parameters $\{\omega_i\}_{i=1}^N$ and $\kappa$, initial data $\{\theta_i^0\}_{i=1}^N$, and real-analytic interaction functions $I$ and $S$ such that $S=I'$. If
\[
\sup_{t\ge 0}\|\Theta(t)\|_{\infty}<\infty,
\]
then there exists $\Theta^\infty\in\mathbb{R}^N$ such that
\[
\lim_{t\to\infty} \|\Theta(t)-\Theta^\infty\|_{\infty}=0 \quad and \quad \lim_{t\to\infty} \|\dot{\Theta}(t)\|_{\infty}=0.
\]
\end{proposition}
\begin{remark}\,
\begin{enumerate}[(a)]
\item The gradient flow formulation of the Kuramoto model was given by \cite{van1993lyapunov}, and the consequence of the {\L}ojasiewicz gradient theorem was recorded in \cite{dong2013synchronization,ha2013formation}.

\item To know about the speed of the convergence toward the equilibrium guaranteed by the {\L}ojasiewicz gradient theorem, one needs to know the exact exponent in the {\L}ojasiewicz gradient inequality, which tells us whether the convergence is exponential or algebraic \cite[Theorem 4.7]{bolte2007lojasiewicz}. In \cite{li2015lojasiewicz}, the {\L}ojasiewicz exponent was computed for certain equilibria of the Kuramoto model: for example, equilibria lying within a quarter-circle have {\L}ojasiewicz exponent $\frac 12$. By the Winfree-to-Kuramoto embedding (described right after Question \ref{Ku-Wi} in Section \ref{sec:history} below), this tells us that equilibria of the Winfree model within $(-\frac{\pi}{4},\frac{\pi}{4})^N$ have {\L}ojasiewicz exponent $\frac 12$, and convergence towards it is exponential; this agrees with Theorem \ref{standardpreciseCOD}.
\end{enumerate}
\end{remark}

Proposition \ref{Loja} supplies convergence and decay of the velocities once boundedness of the lifted trajectory has been proved. The quantitative bound
\[
\sup_{t\ge0}\theta_i(t)-\inf_{t\ge0}\theta_i(t)<2\pi
\]
in Corollary \ref{bestsincosmainthm}(b) is a separate conclusion of the trapping argument, not a consequence of the {\L}ojasiewicz theorem. This distinction also explains Remark \ref{rem:loja-cor}.

On the other hand, when the assumption that $I$ and $S$ are real analytic with $S=I'$ is unavailable, then we will not be able to deduce the convergence of the system \eqref{GenWinfree}. That is why, in the next subsection, we will have to content ourselves with proving only boundedness of solutions.

\subsection{General interaction functions}\label{subsec:interaction}

We have stated Theorems \ref{themainthm}, \ref{sincosmaincor}, and \ref{bestsincosmainthm} in the case of the special Winfree model \eqref{Winfree} for concreteness. The proof methods are robust enough to allow for many other interaction functions $I$ and $S$ considered in the literature.

In the case where $I$ is everywhere strictly positive and $S$ changes sign at least once, it will be straightforward to rigorously prove the oscillator death phenomenon for all initial data with a uniformly bounded $\kappa$; see Proposition \ref{trivial}. (This will serve as a toy model for our bootstrapping proof method in subsection \ref{subsec:toy}.) Most of the interaction functions $I$ and $S$ considered in the literature are such that $I$ is nonnegative with a unique zero and $S$ changes sign around that zero, as in the standard choice \eqref{standard}.

Thus, except in subsection \ref{subsec:toy}, we will assume that there exists an angle $\alpha_0\in (0,\pi)$, exponents $p,q\ge 1$, and constants $c_1,c_2>0$ and $c_3\in (0,1]$ such that
\begin{equation}\label{c_1}
S(\theta)\le -c_1 (\pi-\theta)^p \mbox{ for }  \theta\in [\alpha_0,\pi],\quad S(\theta)\ge c_1 (\theta+\pi)^p \mbox{ for }  \theta\in [-\pi,-\alpha_0],
\end{equation}
\begin{equation}\label{c_2}
0\le I(\theta)\le c_2 (\pi-|\theta|)^q \mbox{ for }  \theta\in [-\pi,\pi],
\end{equation}
and
\begin{equation}\label{c_3}
\min_{|\phi|\le \max\{|\theta|,\alpha_0\}}I(\phi)\ge c_3 I(\theta)\mbox{ for }\theta\in [-\pi, \pi].
\end{equation}
Condition \eqref{c_2} says that $I$ has a zero at $\pi$ and has a power-type upper bound at $\pi$ of exponent $q$. Similarly, \eqref{c_1} says that $S$ has a zero at $\pi$ and has a power-type lower bound on its absolute value of exponent $p$ at $\pi$. Also, note that \eqref{c_3} is satisfied with $c_3=\frac{I(\alpha_0)}{I(0)}$ if $I$ is an even function strictly positive on $(-\pi,\pi)$ and is monotone decreasing on $[0,\pi]$; thus we may consider \eqref{c_3} to be a ``weak monotonicity condition'' on $I$. There is nothing special about the location of the zero $\pi$; it can be at any other location, and as long as we have power-type behavior as in \eqref{c_1} and \eqref{c_2} and the weak monotonicity condition \eqref{c_3} around that zero, the argument of this paper will follow through via a simple translation.

These assumptions are robust under various modifications one may wish to make for $S$ and $I$. For example, in the prototypical case $S(\theta)=-\sin\theta$ and $I(\theta)=1+\cos\theta$, one would take $p=1$, $q=2$, $\alpha_0=\frac{\pi}{2}$, and $c_1$, $c_2$, and $c_3$ to be universal constants. In \cite{ariaratnam2001phase} one also considers the examples $I_n(\theta)=(1+\cos\theta)^n$ for $n\ge 1$, which satisfies \eqref{c_2} with $q=2n$ and \eqref{c_3}. Gallego, Montbri\'o, and Paz\'o \cite{gallego2017synchronization} consider the ``rectified Poisson kernel'' $I_r(\theta)=\frac{(1-r)(1+\cos\theta)}{1-2r\cos\theta+r^2}$, $r\in (-1,1)$, which satisfies \eqref{c_2} with $q=2$ and \eqref{c_3}, and many functions $I$ that are compactly supported in $(-\pi,\pi)$, which satisfy \eqref{c_2} with any $q\ge 1$. Likewise, many natural kernels that approximate $\delta_{0}$ and which have a zero at $\pi$ satisfy \eqref{c_2} and \eqref{c_3}.

As in the prototypical case, we introduce an \emph{average influence}:
\begin{equation}\label{gen_order_parameter}
R\coloneqq \frac 1N \sum_{j=1}^N I(\theta_j),
\end{equation}
with which the ordinary differential equation of \eqref{GenWinfree} can be rewritten in the mean-field form
\begin{equation}\label{GenWinfree_orderparam}
\dot{\theta}_i=\omega_i+\kappa RS(\theta_i).
\end{equation}
The analogue of Corollary \ref{bestsincosmainthm} for the general model \eqref{GenWinfree} is the following theorem.

\begin{theorem}\label{mainthm}
Assume that the $2\pi$-periodic Lipschitz functions $S,I:\mathbb{R}\to\mathbb{R}$ satisfy the constraints \eqref{c_1}, \eqref{c_2}, and \eqref{c_3}. Let $\{\theta_i(t)\}_{i=1}^N$ be the solution to \eqref{GenWinfree} with parameters $\{\omega_i\}_{i=1}^N$, $\kappa$, and $\{\theta_i^0\}_{i=1}^N$, and interaction functions $S$ and $I$.  Suppose ${R_0}\coloneqq R(0)>0$ and
\begin{equation}\label{gen_large_kappa}
\kappa>\max\left\{\frac{2(2c_2)^{p/q}}{c_1c_3}\cdot\frac{\|\Omega\|_\infty}{{R_0}^{1+\frac pq}},\frac{2}{c_1c_3(\pi-\alpha_0)^p}\cdot \frac{\|\Omega\|_\infty}{{R_0}}\right\}.
\end{equation}
Then
\begin{enumerate}[(a)]
\item $R(t)> \frac{c_3 {R_0}}{2}$ for all $t\ge 0$, and
\item $\sup_{t\ge 0} \theta_i(t)-\inf_{t\ge 0}\theta_i(t)<2\pi$ for all $i\in [N]$.
\end{enumerate}

\end{theorem}
\begin{remark}
    By Proposition \ref{Loja}, we can deduce further that if $S$ and $I$ are real-analytic with $S=I'$, then for all $i\in [N]$, the limit $\lim_{t\to\infty}\theta_i(t)$ exists and $\lim_{t\to\infty}\dot{\theta}_i(t)=0$. We will omit this type of observation in the following discussion.
\end{remark}

This theorem tells us that as long as $\kappa\gtrsim \frac{\|\Omega\|_\infty}{{R_0}^{1+\frac pq}}$ we can guarantee boundedness of the solution $\{\theta_i(t)\}_{i=1}^N$. This theorem also tells us why we get the exponent $-\frac 32$ in Corollary \ref{bestsincosmainthm}, namely that it is because $p=1$ and $q=2$ in that case.

The reason why we have stated Corollary \ref{bestsincosmainthm} separately from Theorem \ref{mainthm} was to exhibit that it is possible to optimize the proof so as to obtain reasonable (but not sharp) constants in the expression for $\kappa$, namely $2\cdot \frac{\|\Omega\|_\infty}{{R_0}^{3/2}}$ for $0<R_0\le 1$; that being said, the focus of Theorem \ref{mainthm} is to obtain the power type dependence on ${R_0}$ without caring much about the constants.

The analogue of Theorem \ref{standardpreciseCOD} for general interaction functions $I$ and $S$ is given as follows.
\begin{theorem}[{\cite[Theorem 2.2]{ha2015emergence}}]\label{preciseCOD}
Consider the setting of the Winfree model \eqref{GenWinfree}. Assume that the interaction functions $S$ and $I$ are $2\pi$-periodic and $C^2$, that $S$ is odd and $I$ is even, and that there are angles $0<\theta_*<\theta^*\le \pi$ such that
\begin{align*}
S\le 0 \mbox{ on }[0,\theta^*],\quad S'\le 0,~S''\ge 0\mbox{ on }[0,\theta_*],\\
I\ge 0,~I'\le 0 \mbox{ on }[0,\theta^*],\quad I''\le 0\mbox{ on }[0,\theta_*],\\
(SI)'< 0 \mbox{ on }(0,\theta_*),\quad (SI)'> 0\mbox{ on }(\theta_*,\theta^*).
\end{align*}
Fix $\alpha\in (0,\theta^*)$, and let $\alpha^\infty$ be the unique solution to $(SI)(\alpha^\infty)=(SI)(\alpha)$,  $\alpha^\infty\in (0,\theta_*]$. Assume
\[
\kappa>-\frac{\|\Omega\|_\infty}{S(\alpha)I(\alpha)}.
\]

Then the system \eqref{GenWinfree} has a unique equilibrium in the set $[-\alpha,\alpha]^N$, with the equilibrium lying in the smaller open set $(-\alpha^\infty,\alpha^\infty)^N$, and for every initial datum $\Theta^0\in [-\alpha,\alpha]^N$, the solution $\Theta=\Theta(t)$ to \eqref{GenWinfree} converges exponentially to that unique equilibrium.
\end{theorem}

The conditions on $S$ and $I$ roughly mean that $S$ and $I$ mimic \eqref{standard}.

Given the differences between the special case \eqref{standard} and the more general case of \eqref{c_1}, \eqref{c_2}, and \eqref{c_3}, we can redefine the critical and pathwise critical coupling strengths as follows.
\begin{definition}\label{def:crit-IS}
Fix a frequency vector $\Omega=\{\omega_1,\cdots,\omega_N\}\in \mathbb{R}^N$ and $2\pi$-periodic Lipschitz interaction functions $I$ and $S$.
\begin{enumerate}[(a)]
\item The \emph{critical coupling strength} is
\[
\kappa_{\mathrm{c}}(\Omega,I,S)\coloneqq \inf\{\kappa_*>0:\mathrm{for~}\kappa>\kappa_*,\mathrm{~system~}\eqref{GenWinfree}\mathrm{~admits~a~uniformly~bounded~solution}\}.
\]
\item For an initial phase vector $\Theta^0=\{\theta^0_1,\cdots,\theta^0_N\}\in \mathbb{R}^N$, the \emph{pathwise critical coupling strength} is
\[
\kappa_{\mathrm{pc}}(\Theta^0,\Omega,I,S)\coloneqq \inf\{\kappa_*>0:\mathrm{for~}\kappa>\kappa_*,\mathrm{~the~solution~}\{\theta_i(t)\}_{i=1}^N\mathrm{~to~}\eqref{GenWinfree}\mathrm{~is~bounded}\}.
\]
\end{enumerate}
\end{definition}
Then Theorem \ref{mainthm} can be summarized as follows.
\begin{theorem} For a fixed initial position vector $\Theta^0$ with $R_0>0$, a natural frequency vector $\Omega$, and $2\pi$-periodic Lipschitz interaction functions $I$ and $S$ satisfying \eqref{c_1}, \eqref{c_2}, and \eqref{c_3}, we have
\[
\kappa_{\mathrm{pc}}(\Theta^0,\Omega,I,S)\le  \max\left\{\frac{2(2c_2)^{p/q}}{c_1c_3}\cdot\frac{\|\Omega\|_\infty}{{R_0}^{1+\frac pq}},\frac{2}{c_1c_3(\pi-\alpha_0)^p}\cdot \frac{\|\Omega\|_\infty}{{R_0}}\right\}.
\]
\end{theorem}

One can deduce the following theorem from Theorem \ref{mainthm} similarly to the deduction of Theorem \ref{sincosmaincor} from Corollary \ref{bestsincosmainthm}. Again, it is based on the intuition that if we are sampling the initial data $\theta_i^0$ independently from a fixed distribution on $[-\pi,\pi]$, then by a simple Chernoff bound, the initial average influence $R_0$ will be not too far from its expected value.
\begin{theorem}\label{maincor}
Assume we are given parameters $\{\omega_i\}_{i=1}^N$ and $\kappa$, and $2\pi$-periodic Lipschitz interaction functions $S$ and $I$ that satisfy the constraints \eqref{c_1}, \eqref{c_2}, and \eqref{c_3}. For each initial data $\{\theta_i^0\}_{i=1}^N$, denote by $\{\theta_i(t)\}_{i=1}^N$ the solution to \eqref{GenWinfree}.

Let $\mu$ be a probability measure on $[-\pi,\pi]$, and denote and assume
\begin{equation}\label{mu-pos}
R^*\coloneqq \int_{-\pi}^\pi I(\theta)d\mu(\theta)>0.
\end{equation}
Assume $\kappa$ is large enough so that
\begin{equation}\label{cor-kappa-large}
\kappa>\max\left\{\frac{4(4c_2)^{p/q}}{c_1c_3}\cdot\frac{\|\Omega\|_\infty}{(R^*)^{1+\frac pq}},\frac{4}{c_1c_3(\pi-\alpha_0)^p}\cdot \frac{\|\Omega\|_\infty}{R^*}\right\}.
\end{equation}
Then
\begin{align*}
&\mu^{\otimes N}\left\{\{\theta_i^0\}_{i=1}^N\in [-\pi,\pi]^N:R(t)> \frac{c_3 R^*}{4}~\forall t\ge 0,\sup_{t\ge 0} \theta_i(t)-\inf_{t\ge 0}\theta_i(t)<2\pi ~\forall i\in [N]\right\}\\
&\ge 1-\exp\left(-\frac{(R^*)^2N}{2 \|I\|_{\sup}^2 }\right),
\end{align*}
where $\mu^{\otimes N}$ denotes the product measure on $[-\pi,\pi]^N$.
\end{theorem}

Under more assumptions, we can prove an improvement as $\kappa$ becomes large.
\begin{theorem}\label{maincor-kappalarge}
Assume we are given parameters $\{\omega_i\}_{i=1}^N$ and $\kappa$, and $2\pi$-periodic Lipschitz interaction functions $S$ and $I$ that satisfy the constraints \eqref{c_1}, \eqref{c_2}, and \eqref{c_3}. In addition, let $\mu$ be a probability measure on $[-\pi,\pi]$ under which the function $I$ satisfies\footnote{This is equivalent to requiring $\mu\{\theta\in [-\pi,\pi]:I(\theta)<\varepsilon\}\lesssim \varepsilon^\beta$, by Lemma \ref{lem:extreme-close} and the discussion preceding it. We remark that we will use the following (standard) asymptotic notation. For $P, ~Q>0$, the notations $P\lesssim Q$, $Q\gtrsim P$, $P=O(Q)$, and $Q=\Omega(P)$ mean that $P\le KQ$ for a universal constant $K\in (0,\infty)$, and $P\asymp Q$ means $(P\lesssim Q) \wedge (Q\lesssim P)$. If we need to allow for dependence on parameters, we indicate this by subscripts. For example, in the presence of auxiliary parameters $\psi, \xi$, the notations $P\lesssim_{\psi,\xi}Q$, $Q\gtrsim_{\psi,\xi}P$, $P=O_{\psi,\xi}(Q)$, $Q=\Omega_{\psi,\xi}(P)$ mean that $P\le K(\psi,\xi)Q$ where $K(\psi,\xi)\in (0,\infty)$ may depend only on $\psi$ and $\xi$, and $P\asymp_{\psi,\xi} Q$ means $(P\lesssim_{\psi,\xi}Q)\wedge (Q\lesssim_{\psi,\xi} P)$.}
\begin{equation}\label{eq:I-moment}
\int_{-\pi}^\pi e^{-\lambda I(\theta)}d\mu(\theta)\le \frac{C_\mu}{\lambda^\beta},\quad \lambda>0,
\end{equation}
for some $C_\mu,\beta>0$.
For each initial data $\{\theta_i^0\}_{i=1}^N$, denote by $\{\theta_i(t)\}_{i=1}^N$ the solution to \eqref{GenWinfree}.

Assume $\kappa$ is large enough so that
\begin{equation}\label{eq:kappa-large-cor}
\kappa>\max\left\{\frac{2(2c_2)^{p/q}}{c_1c_3}\left(\frac{e}{\beta}\right)^{1+\frac pq}C_\mu^{\frac 1\beta (1+\frac pq)}\cdot\|\Omega\|_\infty,\frac{2eC_\mu^{1/\beta}}{c_1c_3(\pi-\alpha_0)^p\beta}\cdot {\|\Omega\|_\infty}\right\}.
\end{equation}
Then
\begin{align*}
&\mu^{\otimes N}\Big\{\{\theta_i^0\}_{i=1}^N\in [-\pi,\pi]^N:\sup_{t\ge 0} \theta_i(t)-\inf_{t\ge 0}\theta_i(t)<2\pi ~\forall i\in [N]\Big\}\\
&\ge 1-C_\mu^N\left(e/\beta\right)^{\beta N}\max\left\{\left(\frac{
2}{c_1c_3}\frac{\|\Omega\|_\infty}{\kappa}\right)^{\frac{1}{1+p/q}}(2c_2)^{\frac{p/q}{1+p/q}},\frac{2}{c_1c_3(\pi-\alpha_0)^p}\frac{\|\Omega\|_\infty}{\kappa}\right\}^{\beta N},
\end{align*}
where $\mu^{\otimes N}$ denotes the product measure on $[-\pi,\pi]^N$.
\end{theorem}
Condition \eqref{eq:I-moment} roughly means that when $\theta\sim\mu$, $I(\theta)$ is not too concentrated near $0$.

Under different conditions on $I$ and $S$, we can take advantage of the time flow and can deduce the following uniform oscillator death result which is an analogue of Theorem \ref{themainthm}.
\begin{theorem}\label{themainthm-IS}
Assume that the $2\pi$-periodic Lipschitz functions $S,I:\mathbb{R}\to\mathbb{R}$ satisfy the constraints \eqref{c_1}, \eqref{c_2}, and \eqref{c_3}, along with the additional constraints
\begin{equation}\label{c_4}
S'(\theta)\ge c_4(I_*-I(\theta)), \quad \mathrm{for~a.e.~}\theta\in \mathbb{R},
\end{equation}
for some $I_*>0$ and $c_4>0$,
\begin{equation}\label{c_5}
I'(\theta)S(\theta)\ge 0,\quad \mathrm{for~a.e.~}\theta\in \mathbb{R},
\end{equation}
and the nondegeneracy condition
\begin{equation}\label{c_6}
I(\theta)>0,\quad\forall \theta\in (-\pi,\pi).
\end{equation}
Assume
\[
\kappa>\max\left\{\frac{2(2c_2)^{p/q}}{c_1c_3}\cdot\frac{\|\Omega\|_\infty}{{I_*}^{1+\frac pq}},\frac{2}{c_1c_3(\pi-\alpha_0)^p}\cdot \frac{\|\Omega\|_\infty}{{I_*}}\right\}.
\]
Then for Lebesgue almost every initial data $\Theta^0$, the solution $\Theta=\Theta(t)$ to \eqref{GenWinfree} satisfies the following assertions.
\begin{enumerate}[(a)]
\item (Uniform lower bound on the average influence) $\liminf_{t\to\infty}R(t)\ge \frac{c_3I_*}{2}$.
\item (Oscillator death) $\limsup_{t\to\infty}\theta_i(t)-\liminf_{t\to\infty}\theta_i(t)<2\pi$ for all $i\in [N]$.
\end{enumerate}
\end{theorem}
For example, the interaction functions considered above, namely $S(\theta)=-\sin\theta$ and $I(\theta)=I_n(\theta)$ or $I_r(\theta)$ where $I_n(\theta)=(1+\cos\theta)^n$ for $n\ge 1$ and $I_r(\theta)=\frac{(1-r)(1+\cos\theta)}{1-2r\cos\theta+r^2}$ for $r\in (-1,1)$, all satisfy \eqref{c_1}, \eqref{c_2}, \eqref{c_3}, \eqref{c_4}, \eqref{c_5}, and \eqref{c_6}.

\begin{theorem}\label{thm:general-ae-couplings}
	Fix $\Omega\in\mathbb R^N$ and $2\pi$-periodic Lipschitz functions $I,S$ satisfying \eqref{c_1}, \eqref{c_2}, \eqref{c_3}, \eqref{c_4}, \eqref{c_5}, and \eqref{c_6}, and set
	\[
	C=\max\left\{
	\frac{2(2c_2)^{p/q}}{c_1c_3}\frac{\|\Omega\|_\infty}{I_*^{1+p/q}},
	\frac{2}{c_1c_3(\pi-\alpha_0)^p}\frac{\|\Omega\|_\infty}{I_*}
	\right\}.
	\]
	For almost every initial datum, the solution to \eqref{GenWinfree} is bounded for almost every $\kappa>C$. For every prescribed countable set $\Lambda\subset(C,\infty)$, boundedness holds simultaneously for all $\kappa\in\Lambda$ outside a single null set of initial data.
\end{theorem}
\begin{proof}
	Boundedness is a Borel condition in $(\Theta^0,\kappa)$, since it can be expressed as
	\[
	\bigcup_{m\in\mathbb N}\bigcap_{t\in\mathbb Q_{\ge0}}
	\{\|\Theta(t)\|_\infty\le m\}.
	\]
	Apply Theorem \ref{themainthm-IS}, Fubini's theorem, and a countable-union argument as in the proof of Theorem \ref{thm:pc2}.
\end{proof}
This conclusion does not establish the stronger bound
$\kappa_{\mathrm{pc}}(\Theta^0,\Omega,I,S)\le C$ almost everywhere; uniformity of the exceptional set in $\kappa$ remains to be proved.

With a further assumption on $I$, we can prove the following quantitative improvement as time flows, which is an analogue of Theorem \ref{sincosmaincor-time}.
\begin{theorem}\label{maincor-quant}
Assume we are given parameters $\{\omega_i\}_{i=1}^N$ and $\kappa$ and $2\pi$-periodic Lipschitz interaction functions $S$ and $I$ that satisfy the constraints \eqref{c_1}, \eqref{c_2}, \eqref{c_3}, \eqref{c_4}, \eqref{c_5}, and the following strengthening of \eqref{c_6}: for some $r\ge 1$ and $c_5>0$,
\begin{equation}\label{c_7}
I(\theta)\ge c_5 (\pi-|\theta|)^r \mbox{ for }  \theta\in [-\pi,\pi],
\end{equation}
For each initial data $\{\theta_i^0\}_{i=1}^N$, denote by $\{\theta_i(t)\}_{i=1}^N$ the solution to \eqref{GenWinfree}.

If $\kappa$ is large enough so that
\begin{equation}\label{eq:kappa-large-maincor-quant}
\kappa>\max\left\{\frac{4(4c_2)^{p/q}}{c_1c_3}\cdot\frac{\|\Omega\|_\infty}{{I_*}^{1+\frac pq}},\frac{4}{c_1c_3(\pi-\alpha_0)^p}\cdot \frac{\|\Omega\|_\infty}{{I_*}}\right\},
\end{equation}
then, for every $T\ge 0$,
\begin{align*}
&m\left\{\{\theta_i^0\}_{i=1}^N\in [-\pi,\pi]^N:R(t)> \frac{c_3I_*}{4}~\forall t\ge T \mbox{ and }\sup_{t\ge T} \theta_i(t)-\inf_{t\ge T}\theta_i(t)<2\pi ~\forall i\in [N]\right\}\\
&\ge 1-\left(\exp\left(\frac{rI_*^2}{2\|I\|^2_{\sup}}\right)+\frac{c_4c_5\pi^r r^{r}}{2e\Gamma(1/r)^r(1+r/N)}\kappa  I_*T\right)^{-N/r}.
\end{align*}

If \eqref{eq:kappa-large-maincor-quant} holds and $\kappa$ also satisfies
\begin{equation}\label{eq:kappa-large-maincor-quant-2}
    \kappa>\max\left\{\frac{4e(4ec_2)^{p/q}}{c_1c_3}\left(\frac{\Gamma(1/r)^r}{c_5\pi^r r^{r-1}}\right)^{1+p/q},\frac{4e\Gamma(1/r)^r}{c_1c_3c_5(\pi-\alpha_0)^p \pi^r r^{r-1}}\right\}\|\Omega\|_\infty,
\end{equation}
then\footnote{If $\Omega=0$ so that the expression is undefined, the measure is $1$.}
\begin{align*}
&m\left\{\{\theta_i^0\}_{i=1}^N\in [-\pi,\pi]^N:\sup_{t\ge T} \theta_i(t)-\inf_{t\ge T}\theta_i(t)<2\pi ~\forall i\in [N]\right\}\\
&\ge \quad  1-
\left(\frac{e^{1/r}\Gamma(1/r)}{c_5^{1/r}\pi r^{1-1/r}}\right)^N
\left(\min\left\{\left(\frac{c_1c_3}{2(2c_2)^{p/q}}\right)^{\frac{1}{1+p/q}}\left(\frac{\kappa}{\|\Omega\|_\infty}\right)^{\frac{1}{1+p/q}},\frac{c_1c_3(\pi-\alpha_0)^p}{2}\frac{\kappa}{\|\Omega\|_\infty}\right\}+\frac{c_4\kappa r I_*}{2(1+r/N)}T\right)^{-N/r}.
\end{align*}
\end{theorem}

We consider these results to be confirmations of the existence of the oscillator death regime for the interaction functions $I$ and $S$.

\subsection{Partial oscillator death}\label{subsec:partialdeath}
In the proofs of Theorems \ref{generalbestsincosmainthm} and \ref{mainthm}, we will be able to deduce statements about the behavior of subpopulations. For $\mathcal{B}\subseteq [N]$ we denote
\[
\Omega_\mathcal{B}\coloneqq (\omega_i)_{i\in \mathcal{B}},\quad \Theta_\mathcal{B}\coloneqq (\theta_i)_{i\in \mathcal{B}}.
\]
Likewise, we use $\|\cdot\|_\infty$ to denote the $\ell^\infty$-norm:
\[
\|\Omega_\mathcal{B}\|_\infty\coloneqq \max_{i\in \mathcal{B}} |\omega_i|,\quad \|\Theta_\mathcal{B}\|_\infty\coloneqq \max_{i\in \mathcal{B}} |\theta_i|.
\]
\begin{definition}
Let $\Theta=(\theta_1,\cdots,\theta_N)$ be a solution to the Winfree model \eqref{GenWinfree}. 
 For $\mathcal{B}\subseteq [N]$, we say that the ensemble $\Theta$ exhibits $\mathcal{B}$-\emph{partial oscillator death} if the subensemble $\Theta_\mathcal{B}$ is bounded:
\[
\sup_{t\ge 0}\|\Theta_\mathcal{B}(t)\|_\infty<\infty.
\]

\end{definition}

We will use $\mathcal{B}$ to denote a subset of $[N]$ for which we wish to prove $\mathcal{B}$-partial oscillator death. We sometimes introduce a further subset $\mathcal{A}\subset\mathcal{B}$ of oscillators that we intend to trap within a proper subinterval $(-\pi+\delta,\pi-\delta)$ of $[-\pi,\pi]$.
We introduce the notation $a\vee b\coloneqq\max\{a,b\}$ and $a^+\coloneqq a\vee 0$.

For the influence function $I_n(\theta)=a_n(1+\cos\theta)^n$ with $S=-\sin$, Ko, Ha, and Yoon obtained partial death when a sufficiently large subensemble initially lies in a prescribed arc \cite[Theorem~4.2]{ko2023emerging}. They also discussed the case where roughly half the population initially lies in $(-\pi/2,\pi/2)$ \cite[Remark~4.4]{ko2023emerging}. Proposition~\ref{bootstrap} uses the individual initial phase and frequency contributions in \eqref{part_is_big} to give a criterion for a prescribed subensemble.

\begin{proposition}[Criterion for partial oscillator death in \eqref{Winfree}]\label{bootstrap}
Let $\Theta=\Theta(t)$ be the solution to \eqref{Winfree} with parameters $\{\omega_i\}_{i=1}^N$, $\kappa$, and $\{\theta_i^0\}_{i=1}^N$. Suppose $0<\rho \le 2$ satisfies
\begin{equation}\label{partial_large_kappa}
\kappa>\frac{{\|\Omega_\mathcal{B}\|_\infty}}{\rho},
\end{equation}
and $\mathcal{B}\subseteq  [N]$ satisfies
\begin{equation}\label{part_is_big}
\frac 1N\sum_{i\in \mathcal{B}} \left(\cos\theta_i^0+\sqrt{1-\frac{\omega_i^2}{\kappa^2\rho^2}}\right)^+> \rho.
\end{equation}
Then
\begin{enumerate}[(a)]
\item $R(t)\ge \rho$ for all $t\ge 0$,
\item Let $i\in \mathcal{B}$, $0<\alpha< \sqrt{1-\frac{{\omega_i^2}}{\kappa^2 \rho^2}}$, and suppose there is a time $t_0\ge 0$ such that $\cos \theta_i(t_0)\ge-\alpha$. Then $\cos\theta_i(t)\ge -\alpha$ for $t\ge t_0$, and $\cos \theta_i(t)\ge \alpha$ for $t\ge t_0+\frac{\pi}{\kappa \rho \sqrt{1-\alpha^2}-{|\omega_i|}}$. In particular, $\liminf_{t\to\infty} \cos\theta_i(t)\ge \sqrt{1-\frac{{\omega_i}^2}{\kappa^2 \rho^2}}$.
\item For all $i\in \mathcal{B}$,
\[
\sup_{t\ge 0} \theta_i(t)-\inf_{t\ge 0}\theta_i(t)<2\pi.
\]
\item Let $i,j\in \mathcal{B}$, $0<\alpha< \sqrt{1-\frac{{\omega_i^2\vee \omega_j^2}}{\kappa^2 \rho^2}}$, and suppose there is a time $t_0\ge 0$ such that $\cos \theta_i(t_0)\ge \alpha$ and $\cos \theta_j(t_0)\ge \alpha$. If $\omega_i>\omega_j$, then after modulo $2\pi$ translations,
\[
\theta_i(t)-\theta_j(t)\le \frac{\omega_i-\omega_j}{\kappa\rho\alpha}+\pi\exp\left(-\kappa\rho\alpha\left(t-t_0\right)\right)\mbox{ for }t\ge t_0,
\]
and
\begin{align*}
\theta_i(t)-\theta_j(t)\ge \frac{\omega_i-\omega_j}{2\kappa}-\frac{\omega_i-\omega_j}{2\kappa}\exp\left(-2\kappa\left(t-t_0-\frac{\pi}{\omega_i-\omega_j}\right)\right)\quad 
\mbox{for }t\ge t_0+\frac{\pi}{\omega_i-\omega_j}.
\end{align*}
If $\omega_i=\omega_j$, then after modulo $2\pi$ translations, $\lim_{t\to\infty}\left(\theta_i(t)-\theta_j(t)\right)=0$ exponentially:
\[
|\theta_i(t)-\theta_j(t)|\le \pi \exp\left(-\kappa\rho\alpha\left(t-t_0\right)\right)\mbox{ for }t\ge t_0.
\]
\item If $\mathcal{B}= [N]$, then for all $i\in [N]$, the limit $\lim_{t\to\infty}\theta_i(t)$ exists, and
\[
\lim_{t\to\infty}\dot{\theta}_i(t)=0.
\]
\end{enumerate}
\end{proposition}

\begin{proposition}[Criterion for partial oscillator death in \eqref{GenWinfree}]\label{general-pod}
Let $\Theta=\Theta(t)$ be the solution to \eqref{GenWinfree} with parameters $\{\omega_i\}_{i=1}^N$, $\kappa$, and initial data $\{\theta_i^0\}_{i=1}^N$, and let the $2\pi$-periodic Lipschitz functions $I$ and $S$ satisfy \eqref{c_1}, \eqref{c_2}, and \eqref{c_3}. Suppose $0<\rho\le \|I\|_{\mathrm{sup}}$ and $\mathcal{A}\subset \mathcal{B}\subset  [N]$ satisfy
\begin{equation}\label{pod-cond}
\frac 1N \sum_{i\in \mathcal{A}}I(\theta_i^0)> \frac{\rho}{c_3},\quad \kappa>\frac{\|\Omega_\mathcal{B}\|_\infty}{\rho c_1(\pi-\alpha_0)^p}
\end{equation}
and
\begin{equation}\label{pod-cond-2}
\theta_i^0\in \left[-\pi+\left(\frac{{\|\Omega_\mathcal{B}\|_\infty}}{\kappa \rho c_1}\right)^{1/p},\pi-\left(\frac{{\|\Omega_\mathcal{B}\|_\infty}}{\kappa \rho c_1}\right)^{1/p}\right],\quad \forall i \in \mathcal{A}.
\end{equation}
Then
\begin{enumerate}[(a)]
    \item $R(t)> \rho$  for $t\in [0,\infty)$,
    \item $\theta_i(t)\in [-\pi+\left(\frac{{\|\Omega_\mathcal{B}\|_\infty}}{\kappa \rho c_1}\right)^{1/p},\pi-\left(\frac{{\|\Omega_\mathcal{B}\|_\infty}}{\kappa \rho c_1}\right)^{1/p}]$ for $i\in \mathcal{A}$ and $t\ge 0$,
    \item $\sup_{t\ge 0} \theta_i(t)-\inf_{t\ge 0}\theta_i(t)<2\pi$ for $i\in \mathcal{B}$.
\end{enumerate}
\end{proposition}
The proofs of Propositions \ref{bootstrap} and \ref{general-pod} will be based on a bootstrapping method to prove a global lower bound on the order parameter or average influence $R$. We will prove Corollary \ref{bestsincosmainthm} and Theorem \ref{mainthm} by seeing that the conditions of Propositions \ref{bootstrap} and \ref{general-pod} are satisfied with $\mathcal{B}= [N]$.

\subsection{Equilibria of the prototypical Winfree model}\label{subsec:equilibria}
The asymptotic dynamics described in Corollary \ref{bestsincosmainthm} are optimal in the sense that for any $0<\rho\le 2$ there exists an equilibrium state of \eqref{Winfree} whose order parameter is roughly $\rho$.
\begin{theorem}\label{main_eq_thm}
Fix any $\rho\in (0,2]$. For any $N\ge 1$, $\kappa\neq 0$ and $\{\omega_i\}_{i=1}^N$ such that
\[
N\ge \frac{2}{\rho}, \quad \frac{\|\Omega\|_\infty}{|\kappa|}<\frac{\rho^{1.5}}{16},
\]
there exists an equilibrium initial data $\{\theta_i^0\}_{i=1}^N$ to \eqref{Winfree} with order parameter $\in [\frac 14\rho,\frac 32 \rho]$.

Moreover, we can take $\{\theta_i^0\}_{i=1}^N$ so that there exists $1\le m\le N$ with $\frac{\rho}{4}<\frac mN \le \frac{\rho}{2}$ such that
\[
\frac{2|\omega_i|}{3\rho|\kappa| }\le|\theta_i^0|\le \frac{2\pi  |\omega_i|}{\rho |\kappa|},\quad\mbox{for } i \in [m],
\]
and
\[
\frac{2|\omega_i|}{3\rho|\kappa|}\le|\theta_i^0-\pi|\le \frac{2\pi  |\omega_i|}{\rho |\kappa|},\quad\mbox{for } i\in [N]\setminus [m].
\]
\end{theorem}

\begin{remark}\label{main_eq_thm_rmk}
By Theorem \ref{main_eq_thm}, Corollary \ref{bestsincosmainthm} is optimal in the sense that even if we raise $\kappa$ to arbitrary high values, there exist configurations such that (a) we cannot make the limiting value of $R(t)$ to be much larger than $R_0$, (c) we cannot bring closer to $0$ the oscillators initially close to $\pi$, which may make up a certain proportion of the oscillators, and (d) we cannot make the other oscillators closer to $0$ than the order of $\frac{|\omega_i|}{\kappa R_0}$. Furthermore, it is not hard to construct \emph{fixed} initial configurations that exhibit the aforementioned properties (a), (c), and (d) as we raise $\kappa$ to higher values; simply take a set $\mathcal{C}\subset [N]$ large enough with
\[
\theta_i^0=\pi,~\omega_i=0,\quad i\in \mathcal{C},
\]
so that $\theta_i(t)=\pi$ for $t\ge 0$ and $i\in \mathcal{C}$, and $\sup_{t\ge 0}R(t)\le \frac{2(N-|\mathcal{C}|)}{N}$. It is not clear to us what happens when $\omega_i\neq 0$ for $i\in \mathcal{C}$; understanding this seems to be the main challenge in understanding oscillator death for generic initial data. We remark that for a fixed generic initial data where $\theta_i^0\not\equiv \pi \mod 2\pi$, as we raise $\kappa$ to higher values more and more oscillators will satisfy the conditions of Corollary \ref{bestsincosmainthm} (g) and thus will satisfy $\liminf_{\kappa\to\infty}\liminf_{t\to\infty}\cos\theta_i(t)=1$.
\end{remark}

In the proof of Theorem \ref{main_eq_thm} we will analyze the basic structure of the equilibria of \eqref{Winfree}. From this we will construct a polynomial $W(r)$ of degree at most $2^{N+1}$ with the property that if $r_0$ is a root of $W(r)$ of order $m$, then there can only be at most $m$ distinct equilibria of \eqref{Winfree} with order parameter $r_0$, where we say that two equilibria are distinct if they are distinct modulo $2\pi$. From this we obtain the following theorem.
\begin{theorem}\label{GlobalFiniteness}
For fixed data $(N,\{\omega_i\}_{i=1}^N,\kappa)$ with $\kappa\neq 0$, there are at most $2^{N+1}$ distinct (modulo $2\pi$) equilibria of the system \eqref{Winfree}.
\end{theorem}


\section{Previous work and directions for possible future research}\label{sec:history}

In this section, we discuss previous works on the Winfree model \eqref{GenWinfree} from the standpoint of dynamical systems of ordinary differential equations, and highlight the context of our theorems and their significance. We then relate the Winfree model to the Kuramoto model, in which Ha and the author \cite{ha2020asymptotic} obtained results analogous to Theorem \ref{themainthm} and Corollary \ref{bestsincosmainthm}. During these discussions, we will suggest some questions for future research.

Those who are interested only in the proof of our results may skip to the next section.

\subsection{Oscillator death versus oscillator locking}

\subsubsection{Winfree's original work: a perturbative linear periodic system}
In 1967, Winfree \cite{winfree1967biological} proposed the mathematical model \eqref{GenWinfree} to describe the behavior of weakly coupled biological oscillators such as the pacemaker cells in the human heart\footnote{Under this interpretation, oscillator death can be viewed as cardiac arrest.} or the synchronous flashing of fireflies \cite{buck1966biology}; see \cite{pikovsky2001synchronization,winfree2002emerging} for more real-world applications. The modeling process can be described as follows. Individually, under the absence of other oscillators, each ($2\pi$-periodic) oscillator $\theta_i$ would undergo uniform motion at the rate of its intrinsic frequency $\omega_i$. Under the presence of other oscillators, however, each oscillator feels a frequency perturbation determined by the other oscillators' phases. The interaction between the oscillators happens via a \textit{mean-field} $R$ where each oscillator contributes by the amount $I(\theta)$ determined by its phase $\theta$, and the mean-field $R$ is given as the average of all contributions:
\[
R(\Theta)=\frac 1N \sum_{j=1}^N I(\theta_j).
\]
The strength of the mean-field is denoted by the multiplicative factor of the coupling strength $\kappa$.
Each oscillator reacts to the mean-field with sensitivity $S(\theta)$ depending on its phase $\theta$.
From this we derive the Winfree model \eqref{GenWinfree}. One could modify \eqref{GenWinfree} to incorporate more realistic behavior. For example, one could consider different connection networks (other than the all-to-all uniform connection of \eqref{GenWinfree}), time-delayed interactions, or frustrations, which could arise from the attenuation of the influence $I$ due to some spatial distribution of the oscillators. We will not pursue these directions in this paper.

Winfree considers \eqref{GenWinfree} to be a perturbation of a periodic system, by assuming that the frequencies $\omega_i$ are concentrated around a positive value $\omega_{\mathrm{average}}$, $\kappa>0$ is very small compared to $\omega_{\mathrm{average}}$, and the phases $\theta_i$ are also concentrated around a moving average. Winfree informally takes the infinite $N$ limit to assume the phases and frequencies have a distribution. Under these assumptions, he describes the structure of the solution as a perturbation of the periodic solution $\theta_i(t)=\omega_{\mathrm{average}}t$ and computes some macroscopic quantities associated to the solution.

\subsubsection{Ariaratnam and Strogatz: rich nonlinear dynamics}
Ariaratnam and Strogatz \cite{ariaratnam2001phase} realized that the Winfree model in the form of the nonlinear ordinary differential equation of \eqref{GenWinfree} is much more versatile than Winfree's intention as a linear perturbation of a periodic system. As mentioned in the first section, after fixing $\kappa>0$ and $\gamma\in [0,1]$, they (nonrandomly) sampled $\omega_i$ according to the uniform distribution on $[1-\gamma,1+\gamma]$ and numerically simulated the prototypical model \eqref{Winfree}. They classified the asymptotic behavior of the model based on the rotation numbers $\rho_i=\lim_{t\to\infty}\theta_i(t)/t$, and observed that the shape of the distribution of the $\rho_i$'s did not depend on the initial data $\theta_i^0$. Based on their data, they drew a phase diagram for different $\gamma$ and $\kappa$ depicting different asymptotic behavior. They observed not only a complete locking regime ($\rho_i=\rho>0$ for all $i$ for some $\rho$) for small $\kappa$ and $\gamma$, which corresponds to Winfree's heuristics, but also a complete death regime ($\rho_i=0$ for all $i$) for large $\kappa$, a partial death regime ($\rho_i=0$ for some $i$) for large $\gamma$ and intermediate $\kappa$, a partial locking regime ($\rho_i=\rho$ for some $i$ for some $\rho$) for small $\kappa$ and intermediate $\gamma$, and an incoherence regime (other behavior) for small $\kappa$.

Intuitively, in the ordinary differential equation \eqref{Winfree_orderparam}, there is a competition between the linear first-order force $\omega_i$ that drives the system towards periodic motion, and the nonlinear first-order interaction force $-\kappa R \sin\theta_i$ that tries to keep each $\theta_i$ near $0$ modulo $2\pi$. Winfree's heuristics is that for small $\kappa>0$ and for $\omega_i$ concentrated near a positive quantity, the linear term $\omega_i$ dominates, while the nonlinear interaction term $-\kappa R \sin\theta_i$ simply moderates the small difference of the $\omega_i$'s among the oscillators, so that the collection of the oscillators rotates with a common frequency. If the $\kappa$ is very large so that the nonlinear interaction term dominates, each oscillator would stay bounded near $0$ and its intrinsic frequency $\omega_i$ would simply tell us how far away the oscillator stays away from $0$. If the frequency distribution is wide enough compared to the magnitude of either $\kappa$ or the mean frequency, then different asymptotic behavior will happen among the oscillators, and we will observe either partial death or partial locking, or other incoherent behavior.

According to the phase diagram of \cite{ariaratnam2001phase}, for a fixed distribution of oscillators, as one increases $\kappa$, the system \eqref{Winfree} experiences phase transitions between the regimes such as
\begin{align}\label{situation}
\begin{aligned}
\mbox{Incoherence} &\quad \Longrightarrow \quad \mbox{Partial death} \quad \Longrightarrow \quad \mbox{Complete death}\\
\mbox{Incoherence} &\quad \Longrightarrow \quad \mbox{Partial locking} \quad \Longrightarrow \quad \mbox{Complete death}\\
\mbox{Incoherence} &\quad \Longrightarrow \quad \mbox{Partial locking}\quad \Longrightarrow \quad \mbox{Complete locking} \quad \Longrightarrow \quad \mbox{Complete death}.
\end{aligned}
\end{align}
From this standpoint, this paper presents sufficient conditions for the transition towards partial death and complete death. 

\subsubsection{Rigorous nonlinear dynamics: oscillator death}\label{subsubsec:death}
Most of the well-known literature on the Winfree model seem to be either based on numerical simulations of the Winfree model or its variants such as \cite{ariaratnam2001phase,gallego2017synchronization}, or describe the infinite-$N$ limit \cite{angelini2004phase,basnarkov2009critical, gallego2017synchronization,quinn2007singular,pazo2020winfree} (there are also similar works on variants of the Winfree model, such as \cite{ha2021collective,o2016dynamics,pazo2020winfree,giannuzzi2007phase}). The numerical simulations are significant since they show that diverse asymptotic dynamics emerge for the Winfree model, while also showing that generally the asymptotic dynamics do not depend on the initial data. The works on the infinite-$N$ limit are significant since for example they tell us exactly where the phase transitions happen between the different regimes. However, there are only a handful of rigorous results in the finite-$N$ case. Ermentrout and Kopell \cite{ermentrout1990oscillator} extensively studied the oscillator death phenomenon for $N=2$, while the works \cite{ha2015emergence,ha2016emergent,ha2017emergence,oukil2017synchronization,oukil2019invariant}, although treating arbitrarily large $N$, possess the limitation that they require the initial data to be well-controlled.

A rigorous analysis and proof of oscillator death in the finite-$N$ case was given by Ha, Park, and the author in \cite{ha2015emergence}; see Theorem \ref{preciseCOD}. This work rigorously proves that for a fixed generic initial data the oscillator death phenomenon happens for large $\kappa$, and establishes a nonlinear stability statement for an equilibrium for \eqref{GenWinfree} and identifies its basin of attraction.

We remark that Theorem \ref{preciseCOD} deduces convergence of the system \eqref{GenWinfree} while only requiring $C^2$-regularity on $S$ and $I$, and not even requiring that $S=I'$. This suggests that it might be possible to deduce the convergence statement of Theorem \ref{themainthm} and Corollary \ref{bestsincosmainthm} without appealing to the {\L}ojasiewicz gradient theorem, and that under suitable assumptions we might be able to extend this to general interaction functions.

\begin{question}
Suppose $S$ and $I$ satisfy the conditions of Theorems \ref{mainthm} and \ref{preciseCOD}, and possibly some other regularity conditions, but not necessarily real analyticity. Assume \eqref{gen_large_kappa}. Can we conclude the convergence of $\Theta(t)$?
\end{question}

The subsequent work \cite{ha2016emergent} showed that a similar framework holds with general symmetric connectivity graphs. However, these works \cite{ha2015emergence,ha2016emergent} have the drawback that they require all of the initial data to lie on some proper subinterval of $[-\pi,\pi]$, on  which the sufficient coupling strength $\kappa$ depends. Thus, when one is sampling $\theta_i^0$ uniformly randomly from the entire interval $[-\pi,\pi]$, the $\kappa$ typically required by Theorem \ref{preciseCOD} will grow arbitrarily large as $N\to\infty$.

The significance of Corollary \ref{bestsincosmainthm} and Theorem \ref{mainthm} is that they prove oscillator death while only requiring control on the order parameter. This is good news since the order parameter is a macroscopic quantity that is easier to control, as shown in Theorems \ref{themainthm}, \ref{sincosmaincor}, \ref{sincosmaincor-time}, \ref{maincor}, \ref{maincor-kappalarge}, \ref{themainthm-IS}, and \ref{maincor-quant}.

Although Corollary \ref{bestsincosmainthm} and Theorem \ref{mainthm} fail to address the case of initial data with very small order parameter, the initial data $\Theta^0$ such that $R^0$ is close to $0$ is unstable in the sense that the divergence of the flow \eqref{Winfree} is positive. Thus, most solutions to \eqref{Winfree} escape the small $R$ regime in finite time, after which we apply Corollary \ref{bestsincosmainthm} or Theorem \ref{mainthm} to obtain Theorems \ref{themainthm}, \ref{sincosmaincor}, \ref{sincosmaincor-time}, \ref{maincor}, \ref{maincor-kappalarge}, \ref{themainthm-IS}, and \ref{maincor-quant}, namely oscillator death for Lebesgue a.e. initial data.

\subsubsection{Rigorous nonlinear dynamics: oscillator locking}

There are also rigorous works \cite{ha2017emergence,oukil2017synchronization,oukil2019invariant} on oscillator locking, which verify Winfree's original heuristics. These works prove the existence and nonlinear stability of traveling solutions to \eqref{Winfree} in the finite-$N$ regime, when $\kappa$ is small and the intrinsic frequencies are tightly concentrated around a fixed positive value. To be precise, we consider the following definitions.

\begin{definition}\label{def:locking}
Let $\Theta=(\theta_1,\cdots,\theta_N)$ be a solution to the Winfree model \eqref{GenWinfree}.
\begin{enumerate}[(a)]
\item We say that the ensemble $\Theta$ exhibits \emph{(complete) oscillator locking} if
\[
\sup_{t\ge 0}|\theta_i(t)-\theta_j(t)|<\infty,\quad \forall i,j\in [N],
\]
and
\[
\liminf_{t\to\infty}\frac{|\theta_i(t)|}{t}> 0,\quad \forall i\in [N].
\]
Thus the asymptotic rotation speed stays bounded away from zero; we do not require the existence of a rotation number.
\item For $\mathcal{B}\subseteq[N]$, we say that the ensemble $\Theta$ exhibits $\mathcal{B}$-\emph{partial oscillator locking} if the subensemble $\Theta_\mathcal{B}$ exhibits oscillator locking:
\[
\sup_{t\ge 0}|\theta_i(t)-\theta_j(t)|<\infty,\quad \forall i,j\in \mathcal{B},
\]
and
\[
\liminf_{t\to\infty}\frac{|\theta_i(t)|}{t}> 0,\quad \forall i\in\mathcal{B}.
\]
\end{enumerate}
\end{definition}

The mechanism of oscillator locking in \cite{oukil2017synchronization,oukil2019invariant,ha2017emergence} is as follows. Suppose $\kappa$ is small, the $\omega_i$ are concentrated around a fixed positive quantity, say $\omega_i\in [1-\gamma,1+\gamma]$, and the $\theta_i^0$'s are concentrated on a small arc. Then the term $\omega_i\approx 1$ dominates in \eqref{GenWinfree}, and the entire population will rotate around the circle. One can compute how the infinitesimal distance between two oscillators change after one revolution, namely, if all the $\omega_i$'s are $1$, then the log of the ratio between the infinitesimal distances between $\theta_i$ and $\theta_j$ before and after one revolution is $\kappa\int_0^{2\pi}\frac{I(\theta)S'(\theta)}{1+\kappa I(\theta)S(\theta)}d\theta$ (provided $1+\kappa IS>0$). Thus, if
\begin{equation}\label{SyncHypo}
\int_0^{2\pi}\frac{I(\theta)S'(\theta)}{1+\kappa I(\theta)S(\theta)}d\theta<0
\end{equation}
(this is called the ``synchronization hypothesis'' in \cite{oukil2017synchronization}), the oscillators will become more concentrated after each revolution, and hence will stay close together for all time, exhibiting oscillator locking. The point of the works \cite{oukil2017synchronization,oukil2019invariant, ha2017emergence} is to ``nonlinearize'' the above ``linear'' analysis when $\kappa$, $\gamma$, and the diameter of the initial data are small.

More precisely, \cite{oukil2017synchronization} states that if \eqref{SyncHypo} holds, $\kappa$ and $\gamma$ are small, and the oscillators are initially close to each other, then they will stay close to each other for all time. In their sequel work \cite{oukil2019invariant} they obtain more refined stability properties of the solution operator under the same hypotheses. On the other hand, the work \cite{ha2017emergence} shows in the special case $S(\theta)=-\sin\theta$ and $I(\theta)=1+\cos\theta$ that if at least $\frac{4}{4+\pi}$ of the oscillators are concentrated on a small arc, then partial locking happens. It seems very likely that a version of this partial locking statement holds for general interaction functions $S$ and $I$ which satisfy \eqref{SyncHypo}.

\begin{question}[Rigorous partial locking for general $S$ and $I$]
Suppose the coupling strength $\kappa_0$ and $2\pi$-periodic Lipschitz functions $S$ and $I$ satisfy $1+\kappa_0 IS>0$ and \eqref{SyncHypo} for $0<\kappa\le \kappa_0$. Do there exist $\rho\in (0,1)$, $\gamma\in (0,1]$, and $\Delta_1,\Delta_2>0$ such that if $\mathcal{A}\subset\mathcal{B}\subset [N]$, $|\mathcal{A}|/N\ge\rho$, $\omega_i\in [1-\gamma \kappa,1+\gamma\kappa ]$ for all $i\in \mathcal{B}$, $\operatorname{diam}\{\theta_i^0\}_{i\in \mathcal{A}}<\Delta_1$, and $0<\kappa<\kappa_0$, then $\operatorname{diam}\{\theta_i(t)\}_{i\in \mathcal{A}}<\Delta_2$ for all $t\ge 0$ and the ensemble $\Theta$ exhibits $\mathcal{B}$-partial oscillator locking?
\end{question}

There still remains the question whether (partial) locking states may emerge from generic initial data. For this, Winfree \cite{winfree1967biological} provides the following insight:
\begin{displayquote}
... if [$I(\theta)$] adds constructively to an entraining rhythm which is the sum of all the other [$I(\theta_i)$], and if the spectral density of oscillators is sufficiently great, then any infinitesimal periodicity in [$R$] is able to make a few oscillators coherent, which then sufficiently augment the influence rhythm to entrain a few more of nearby frequency. Like spontaneous combustion, the growth of [$R(\theta)$] is autocatalytic. The phase-locked nucleus expands along the frequency axis until it has engulfed a bandwidth under $N(f)$ that cannot further expand except by conscripting more oscillators to boost $[|R(\phi)|]$ than are available in the rarified population just beyond its frequency axis boundaries... \cite[p.27]{winfree1967biological}
\end{displayquote}
It would seem that new ideas are needed to elaborate on this intuition to show the emergence of (partial) locking from generic initial data.

Using the ideas of this paper along with the ideas from \cite{oukil2017synchronization,oukil2019invariant,ha2017emergence}, it may be possible to rigorously demonstrate the emergence of a mixture of partial death and partial locking, described below.

\begin{question}[Rigorous partial death with partial locking]\label{ques:deathlockingmixture}
For simplicity, let $S$ and $I$ be given as in \eqref{standard}. Do there exist $\rho_1,\rho_2,\gamma\in (0,1]$ with $\rho_1+\rho_2\le 1$, and $\kappa>0$ such that if $\mathcal{B},\mathcal{C}\subset [N]$ are mutually disjoint subsets with $|\mathcal{B}|/N\ge \rho_1$, $|\mathcal{C}|/N\ge \rho_2$, $\omega_i\in [-\gamma,\gamma]$ for all $i\in \mathcal{B}$, and $\omega_i\in [1-\gamma,1+\gamma]$ for all $i\in \mathcal{C}$, then for some initial data $\Theta^0$, the solution $\Theta=\Theta(t)$ to \eqref{Winfree} exhibits $\mathcal{B}$-partial oscillator death and $\mathcal{C}$-partial oscillator locking?
\end{question}
Here is one idea to answer Question \ref{ques:deathlockingmixture} in the affirmative. As the ensemble $\mathcal{C}$ rotates around the unit circle at more or less the same speed, the average contribution of $\mathcal{C}$ to $R(t)$ will be positive. However, the proof of Corollary \ref{bestsincosmainthm} tells us that a \emph{pointwise} lower bound on $R$ can be exploited to prove oscillator death (see Lemmas \ref{gen-a-priori} and \ref{Simplest-a-priori}). One might be able to adapt the methodology of this paper to exploit an \emph{average} lower bound on $R$ (as opposed to a pointwise one) to show $\mathcal{B}$-partial oscillator death.

\subsection{Analogy with the Kuramoto model}\label{subsec:kuramoto}

The \textit{Kuramoto model} is the following initial value problem for the $N\ge 1$ variables $\theta_1,\cdots,\theta_N\in\mathbb{R}$, with the same parameters $\omega_1,\cdots,\omega_N\in\mathbb{R}$, $\kappa\in \mathbb{R}$, and $\theta_1^0,\cdots,\theta_N^0\in \mathbb{R}$ as in \eqref{Winfree}:
\begin{equation}\label{Ku}
\begin{cases}
\dot{\theta}_i(t)=\omega_i+\frac \kappa N \sum_{j=1}^N \sin (\theta_j(t)-\theta_i(t)),\quad t>0,\\
\theta_i(0)=\theta_i^0,
\end{cases}
\quad i\in [N].
\end{equation}
This model was proposed by Kuramoto in \cite{kuramoto1975international,kuramoto2003chemical} as a variant of the Winfree model \eqref{Winfree} that is more amenable to mathematical analysis, namely he was able to demonstrate a phase-transition phenomenon when increasing $\kappa$ from $0$ to some finite value. This model has since found applications in Josephson junction arrays \cite{watanabe1994constants}; see \cite{acebron2005kuramoto,pikovsky2015dynamics,strogatz2000kuramoto} for an exposition on this topic.

Unlike the Winfree model, the Kuramoto model \eqref{Ku} has the conservation law
\[
\frac{d}{dt}\left[\sum_{i=1}^N \theta_i(t)-\left(\sum_{i=1}^N \omega_i\right) t\right]=0,
\]
and admits a corresponding $\mathrm{U}(1)$-symmetry, namely \eqref{Ku} is invariant under the Galilean transformation
\[
\omega_i\mapsto \omega_i-\nu,\quad \theta_i^0\mapsto \theta_i^0-\vartheta,\quad \theta_i(t)\mapsto \theta_i(t)-\nu t-\vartheta,
\]
for fixed $\nu,\vartheta\in\mathbb{R}$. Therefore, compared to the Winfree model \eqref{Winfree}, in the analysis of the Kuramoto model \eqref{Ku}, the diameter of the intrinsic frequencies $\omega_i$,
\[
D(\Omega)\coloneqq \max_{i,j\in [N]} |\omega_i-\omega_j|,
\]
becomes the correct analogue of the maximum frequency $\|\Omega\|_\infty$, and the notion of relative equilibria, namely solutions $\{\theta_i(t)\}_{i=1}^N$ such that
\[
\theta_i(t)-\theta_j(t)=\theta_i^0-\theta_j^0,\quad \forall i,j\in [N],\quad \forall t\ge 0,
\]
becomes the correct analogue of the notion of (absolute) equilibria.

One can also introduce order parameters, namely $R\in [0,1]$, $\phi\in \mathbb{R}$ such that
\[
R e^{\mathrm{i}\phi}=\frac 1N \sum_{j=1}^N e^{\mathrm{i}\theta_j},
\]
which simplify the ODE of \eqref{Ku} to
\[
\dot\theta_i=\omega_i-\kappa R\sin(\theta_i-\phi)
\]
and which satisfy
\begin{equation}\label{R-canonical}
R=\frac 1N \sum_{j=1}^N \cos(\theta_j-\phi).
\end{equation}

One of the main motivations for this paper was to create analogues for the Winfree model of some results that were previously known for the Kuramoto model, stated as follows. More precisely, Corollary \ref{bestsincosmainthm} is the analogue of Theorem \ref{Kumainthm1}, and Theorem \ref{themainthm} is the analogue of Theorem \ref{Kumainthm2}.
\begin{theorem}[{\cite[Theorem 3.2]{ha2020asymptotic}}]\label{Kumainthm1}
Suppose that $\Theta^0$, $\Omega$ and $\kappa$ satisfy
\[
R_0 \coloneqq   R(\Theta^0) > 0,\quad D(\Omega)>0, \quad \kappa >1.6\frac{D(\Omega)}{R_0^2},
\]
and let $\Theta = \Theta(t)$ be the solution to \eqref{Ku}. Then, the following assertions hold.
\begin{enumerate}[(a)]
\item \emph{Asymptotic complete phase-locking} occurs, namely
\[
\exists \lim_{t\to\infty}\left(\theta_i(t)-\theta_j(t)\right),~\mathrm{and}~\lim_{t\to\infty} \left(\dot{\theta}_i(t)-\dot{\theta}_j(t)\right)=0,\quad \forall i,j\in [N].
\]

\item Let
\[
\gamma(R_0)=
\begin{cases}
0.5+\frac{0.35}{0.94}R_0& 0<R_0\le 0.94, \\
1-2.5(1-R_0)& 0.94<R_0\le 1.
\end{cases}
\]
Then there exists $\mathcal{A}\subset [N]$ with $|\mathcal{A}|/N\ge \gamma(R_0)$ such that for all large enough time $t$, $\Theta_\mathcal{A}(t)$ lies on an arc of length $\le \frac{3\pi}{4(2\gamma(R_0)-1)}\frac{D(\Omega)}{\kappa}$, up to modulo $2\pi$-translations.

\item The aforementioned stable ensemble $\Theta_\mathcal{A}$ becomes ordered in accordance with its natural frequencies: if $i,j\in \mathcal{A}$ and $\omega_i\ge \omega_j$, then
\[
\frac{\omega_i-\omega_j}{\kappa}\le \liminf_{t\rightarrow\infty}[\theta_i(t)-\theta_j(t)]\le \limsup_{t\rightarrow\infty}[\theta_i(t)-\theta_j(t)]
\le c\frac{\omega_i-\omega_j}{\kappa}
\]
after modulo $2\pi$-translations, where the constant $c$ depends only on $\gamma(R_0)$ and $D(\Omega)/\kappa$.
\end{enumerate} 
\end{theorem}
\begin{theorem}[{\cite[Theorem 3.3]{ha2020asymptotic}}]\label{Kumainthm2}
Suppose that $\Omega$ and $\kappa$ satisfy
\[
\kappa >1.6N D(\Omega)>0.
\]
Then for Lebesgue almost every initial datum $\Theta^0$, the solution $\Theta = \Theta(t)$ to \eqref{Ku} satisfies the following assertions.
\begin{enumerate}[(a)]
\item \emph{Asymptotic complete phase-locking} occurs, namely
\[
\exists \lim_{t\to\infty}\left(\theta_i(t)-\theta_j(t)\right),~\mathrm{and}~\lim_{t\to\infty} \left(\dot{\theta}_i(t)-\dot{\theta}_j(t)\right)=0,\quad \forall i,j\in [N].
\]

\item Let
\[
\gamma_N=
0.5+\frac{0.35}{0.94\sqrt{N}}
\]
Then there exists $\mathcal{A}\subset [N]$ with $|\mathcal{A}|/N\ge \gamma_N$ such that for all large enough time $t$, $\Theta_\mathcal{A}(t)$ lies on an arc of length $\le \frac{3\pi}{4(2\gamma_N-1)}\frac{D(\Omega)}{\kappa}$, up to modulo $2\pi$-translations.

\item The aforementioned stable ensemble $\Theta_\mathcal{A}$ becomes ordered in accordance with its natural frequencies: if $i,j\in \mathcal{A}$ and $\omega_i\ge \omega_j$, then
\[
\frac{\omega_i-\omega_j}{\kappa}\le \liminf_{t\rightarrow\infty}[\theta_i(t)-\theta_j(t)]\le \limsup_{t\rightarrow\infty}[\theta_i(t)-\theta_j(t)]
\le c\frac{\omega_i-\omega_j}{\kappa}
\]
after modulo $2\pi$-translations, where the constant $c$ depends only on $\gamma_N$ and $D(\Omega)/\kappa$.
\end{enumerate} 
\end{theorem}

One connection between the Winfree model and the Kuramoto model is that their ODEs are both equations of the form
\begin{equation}\label{GenKu0}
\dot\theta_i=\omega_i+A\cos\theta_i+B\sin\theta_i,\quad i\in [N],
\end{equation}
or alternatively of the form
\begin{equation}\label{GenKu}
\begin{cases}
\dot{\theta}_i(t)=\omega_i(\Theta(t))+ \kappa R(t) \sin (\phi(t)-\theta_i(t)),\quad t>0,\\
\theta_i(0)=\theta_i^0,
\end{cases}
\quad i\in [N].
\end{equation}
with $\omega_i(t)=\omega(\Theta(t))$, $A(t)=A(\Theta(t))$, $B(t)=B(\Theta(t))$, $R(t)=R(\Theta(t))$, $\phi=\phi(\Theta(t))$ being functions of $\Theta$.

The equations \eqref{GenKu0} and \eqref{GenKu} were shown to be highly integrable by Watanabe and Strogatz \cite{watanabe1994constants} in the case $\omega_i$ are identical. See also Chen, Engelbrecht, and Mirollo \cite{chen2017hyperbolic} for a formulation of \eqref{GenKu} in terms of a flow on the hyperbolic disc, and the work of Ha, Kim, Park and the author \cite{ha2021constants} for higher-dimensional generalizations of the Watanabe-Strogatz constants of motion.

Due to the similarity of Corollary \ref{bestsincosmainthm} and Theorem \ref{Kumainthm1}, it seems natural to pose the following question.
\begin{question}\label{Ku-Wi}
	For which functions $\phi$ and $R$ with \eqref{R-canonical} do there exist a constant $c_1$ and exponent $p>0$ such that if $\Theta=\Theta(t)$ is the solution to \eqref{GenKu} with $\omega_i(\Theta)=\omega_i\in \mathbb{R}$, $R_0=R(\Theta^0)>0$, and $\kappa>c_1\frac{\|\Omega\|_\infty}{R_0^p}$, then
	\[
	\sup_{t\ge 0} |\theta_i(t)-\theta_j(t)|<\infty,\quad i,j=1,\cdots, N
	\]
	holds? If so, is there a constant $c_2$ and an exponent $q>0$ such that there exists a subset $\mathcal{A}\subset [N]$ with $|\mathcal{A}|/N\ge \frac 12+c_2R_0^q$ such that
	\[
	\limsup_{t\to\infty}|\theta_i(t)-\theta_j(t)|\lesssim_{R_0} \frac{\|\Omega\|_\infty}{\kappa}
	\]
	holds after modulo $2\pi$ translations? What about $\phi$ and $R$ which do not satisfy \eqref{R-canonical}?
\end{question}

A more direct connection between the Winfree and Kuramoto model is that the Winfree model \eqref{Winfree} can be embedded into the Kuramoto model \eqref{Ku}: following \cite[Lemma 3.3]{moreno2026relaxation}, the Winfree system \eqref{Winfree} with parameters $N\ge 1$, $\kappa>0$, $\Omega\in \mathbb{R}^N$, $\Theta^0\in \mathbb{R}^N$ and solution $\Theta(t)\in \mathbb{R}^N$ one can consider the Kuramoto model with parameters $\tilde{N}=4N$, $\tilde{\kappa}=2\kappa$, $\tilde{\Omega}=(\Omega,-\Omega,\vec{0},\vec{0})\in \mathbb{R}^{4N}$, $\tilde{\Theta}^0=(\Theta^0,-\Theta^0,\vec{0},\vec{0})\in \mathbb{R}^{4N}$; the resulting solution is given by
\[
\tilde{\Theta}(t) = (\Theta(t),-\Theta(t),\vec{0},\vec{0})\in \mathbb{R}^{4N},\quad t\ge 0.
\]

Compared to the Kuramoto model, the Winfree model \eqref{Winfree} has the $\mathrm{U}(1)$ symmetry broken; it appears that this makes the analysis of complete oscillator death relatively easier than that of asymptotic complete phase-locking in the Kuramoto model, but makes the analysis of complete locking harder. In fact, when comparing Corollary \ref{bestsincosmainthm} versus Theorem \ref{Kumainthm1}, we see that Corollary \ref{bestsincosmainthm} requires the weaker lower bound ${R_0}^{-1.5}$ on $\kappa$ than ${R_0}^{-2}$ of Theorem \ref{Kumainthm1}, while controlling a comparable number of oscillators (Corollary \ref{bestsincosmainthm} guarantees the behavior of $O(R_0)$ of the oscillators and $2N$ ``virtual'' oscillators described above, and so in total guarantees the behavior of $\frac 12+O(R_0)$ of the oscillators, just as in Theorem \ref{Kumainthm1}). Also, comparing Theorem \ref{themainthm} versus Theorem \ref{Kumainthm2}, we can see that we only require a weaker condition $\kappa>2\|\Omega\|_\infty$ in Theorem \ref{themainthm} while we require the somewhat unreasonably large condition $\kappa>1.6ND(\Omega)$ in Theorem \ref{Kumainthm2}; this is because the volumetric argument using the divergence which gives Theorem \ref{themainthm} from Corollary \ref{bestsincosmainthm} and Theorem \ref{Kumainthm2} from Theorem \ref{Kumainthm1} tells us that the order parameter $R$ reaches the universal constant $1$ for the Winfree model while it reaches the small amount $\frac{1}{\sqrt{N}}$ for the Kuramoto model. This is because the Winfree model is effectively an $N$-dimensional submanifold of a $4N$-dimensional Kuramoto model, so the notions of volume in the Winfree and Kuramoto models are unrelated.



\section{An order parameter bootstrapping argument}\label{sec:bootstrap}

The purpose of this section is to prove Theorems \ref{generalbestsincosmainthm} and \ref{mainthm}, Corollary \ref{bestsincosmainthm}, and Propositions \ref{bootstrap} and \ref{general-pod}. Roughly, the proof idea is as follows. Choosing a suitable $\rho\in (0,R_0)$, we will find a subpopulation $\Theta_\mathcal{B}$ that we have some control over, in the sense that for all sufficiently small times $t\ge 0$ we have $R(t)\ge \frac 1N \sum_{i\in \mathcal{B}}I(\theta_i)>\rho$. By ``continuous induction'', we will be able to guarantee this property for all times $t\ge 0$.

We first prove toy versions of Theorem \ref{mainthm}, which tell us why it is important to have a uniform-in-time lower bound on the order parameter when proving oscillator death. We then prove Theorem \ref{mainthm} itself, which is a version of Corollary \ref{bestsincosmainthm} (which in turn is a consequence of Theorem \ref{generalbestsincosmainthm}) for more general interaction functions, before proving Theorem \ref{generalbestsincosmainthm}, because the argument is simpler in this more general setting and so it is easier to see how the bootstrapping proof works and why we get the exponent $1+\frac pq$ in Theorem \ref{mainthm}. We will then prove Theorem \ref{generalbestsincosmainthm} by specializing this proof method to the system \eqref{Winfree} while optimizing the argument to gain some control on the constants involved.

We will also obtain criteria for partial oscillator death in this section, namely we will prove Propositions \ref{bootstrap} and \ref{general-pod}. Complete oscillator death will be proven by an application of these criteria.

We will prove the toy versions of Theorem \ref{mainthm} in subsection \ref{subsec:toy}, Theorem \ref{mainthm} and Proposition \ref{general-pod} in subsection \ref{subsec:mainthm}, and Theorem \ref{generalbestsincosmainthm}, Corollary \ref{bestsincosmainthm}, and Proposition \ref{bootstrap} in subsection \ref{subsec:sincosmainthm}.

\subsection{Two toy models}\label{subsec:toy}
\subsubsection{Each oscillator constrains itself}
Suppose
\begin{equation}\label{eq:toy-0}
I\ge 0,\quad \min_{\theta\in [-\pi,\pi]}I(\theta)S(\theta)<0<\max_{\theta\in [-\pi,\pi]}I(\theta)S(\theta).
\end{equation}
In the single oscillator case $N=1$ the ODE becomes
\[
\dot\theta_1=\omega_1+\kappa I(\theta)S(\theta)
\]
and it is easy to see that $\kappa>\frac{|\omega_1|}{\min\{-\min IS, \max IS\}}$ implies boundedness of $\theta_1$. It is not hard to extend this argument to higher $N$ if we are willing to pay the price of a $\kappa$ proportional to $N$.
\begin{proposition}[Criterion for partial oscillator death under \eqref{eq:toy-0}]\label{verytrivial}
Let $\Theta=\Theta(t)$ be the solution to \eqref{GenWinfree} with parameters $\{\omega_i\}_{i=1}^N$, $\kappa$, and initial data $\{\theta_i^0\}_{i=1}^N$, and with the $2\pi$-periodic Lipschitz interaction functions $I$ and $S$ satisfying \eqref{eq:toy-0}.

Fix $\mathcal{B}\subset [N]$. If
\begin{equation}\label{kappa-N-trivial}
    \kappa>\frac{N\|\Omega_\mathcal{B}\|_\infty}{\min\{-\min IS,\max IS\}},
\end{equation}
then
\[
\sup_{t\ge 0} \theta_i(t)-\inf_{t\ge 0}\theta_i(t)<2\pi \quad \forall i\in\mathcal{B}.
\]

\end{proposition}
\begin{proof}
Let $\alpha_0,\beta_0\in [-\pi,\pi]$ such that
\[
\alpha_0\in\arg\min IS, \quad \beta_0\in \arg\max IS,
\]
so that $(IS)(\alpha_0)<0$ and $(IS)(\beta_0)>0$. In particular, $I(\alpha_0)>0$, $S(\alpha_0)<0$, $I(\beta_0)>0$, and $S(\beta_0)>0$. For any $i\in \mathcal{B}$ we see by \eqref{kappa-N-trivial} that $\theta_i(t_0)\equiv\alpha_0 \mod 2\pi$ implies
\begin{equation}\label{waste}
\dot\theta_i(t_0)=\omega_i+\frac{\kappa}{N}(IS)(\alpha_0)+\frac{\kappa}{N}\sum_{j\neq i}I(\theta_j(t_0))S(\alpha_0)\le \omega_i+\frac{\kappa}{N}(IS)(\alpha_0)\stackrel{\mathclap{\eqref{kappa-N-trivial}}}{<}0,
\end{equation}
and similarly that $\theta_i(t_0)\equiv\beta_0 \mod 2\pi$ implies $\dot\theta_i(t_0)>0$. Choose lifts $b$ of $\beta_0$ and $a$ of $\alpha_0$ with $b<a<b+2\pi$. Each interval $[b+2k\pi,a+2k\pi]$ is forward invariant. If the initial phase lies in one of them, its range is at most $a-b<2\pi$. Otherwise it lies in a complementary interval $(a+2k\pi,b+2(k+1)\pi)$. If it never leaves that interval, its range is again less than $2\pi$. If it first exits through the left endpoint, the preceding finite trajectory segment has maximum strictly below $b+2(k+1)\pi$, and the subsequent trajectory stays in $[b+2k\pi,a+2k\pi]$. Its total range is therefore strictly less than $2\pi$. Exit through the right endpoint is treated symmetrically.
\end{proof}
In particular, Proposition \ref{verytrivial} tells us that the pathwise critical coupling strength is finite everywhere:
\[
\kappa_{\mathrm{pc}}(\Theta^0,\Omega,I,S)\le \frac{N\|\Omega\|_\infty}{\min\{-\min IS,\max IS\}}.
\]
However, this is not an effective bound for large $N$.

\subsubsection{Importance of a lower bound on the average influence}
Note that in \eqref{waste} we have wasted many nonpositive terms. The main idea of the bootstrapping argument for Theorems \ref{generalbestsincosmainthm} and \ref{mainthm} is to guarantee a uniform-in-time lower bound on the average influence $R$. To see why we need such a condition, we examine the case when the interaction functions $I$ and $S$ satisfy the following:
\begin{equation}\label{eq:toy-I}
I(\theta)\ge \rho\quad  \forall \theta\in [-\pi,\pi]
\end{equation}
for some $\rho>0$ and
\begin{equation}\label{eq:toy-S}
    \min_{\theta\in [-\pi, \pi]}S(\theta)<0<\max_{\theta\in [-\pi, \pi]}S(\theta).
\end{equation}

\begin{proposition}[Criterion for partial oscillator death under \eqref{eq:toy-I}-\eqref{eq:toy-S}]\label{trivial}
Let $\Theta=\Theta(t)$ be the solution to \eqref{GenWinfree} with parameters $\{\omega_i\}_{i=1}^N$, $\kappa$, and initial data $\{\theta_i^0\}_{i=1}^N$, and with the $2\pi$-periodic Lipschitz interaction functions $I$ and $S$ satisfying \eqref{eq:toy-I} and \eqref{eq:toy-S}.

Fix $\mathcal{B}\subset [N]$. If
\begin{equation}\label{kappa-trivial}
    \kappa>\frac{\|\Omega_\mathcal{B}\|_\infty}{\rho \min\{-\min S,\max S\}},
\end{equation}
then
\[
\sup_{t\ge 0} \theta_i(t)-\inf_{t\ge 0} \theta_i(t)<2\pi \quad \forall i\in\mathcal{B}.
\]

\end{proposition}
\begin{proof}
Let $\alpha_0,\beta_0\in [-\pi,\pi]$ be such that
\[
\alpha_0\in\arg\min S, \quad \beta_0\in \arg\max S,
\]
so that $S(\alpha_0)<0$ and $S(\beta_0)>0$. For any $i\in \mathcal{B}$ we see by \eqref{GenWinfree_orderparam} and \eqref{kappa-trivial} (and the observation that $R\ge \rho $) that $\theta_i(t_0)\equiv\alpha_0 \mod 2\pi$ implies $\dot\theta_i(t_0)<0$ and that $\theta_i(t_0)\equiv\beta_0 \mod 2\pi$ implies $\dot\theta_i(t_0)>0$. Choose lifts $b$ of $\beta_0$ and $a$ of $\alpha_0$ with $b<a<b+2\pi$. Each interval $[b+2k\pi,a+2k\pi]$ is forward invariant. If the initial phase lies in one of them, its range is at most $a-b<2\pi$. Otherwise it lies in a complementary interval $(a+2k\pi,b+2(k+1)\pi)$. If it never leaves that interval, its range is again less than $2\pi$. If it first exits through the left endpoint, the preceding finite trajectory segment has maximum strictly below $b+2(k+1)\pi$, and the subsequent trajectory stays in $[b+2k\pi,a+2k\pi]$. Its total range is therefore strictly less than $2\pi$. Exit through the right endpoint is treated symmetrically.
\end{proof}

Therefore, for $I$ and $S$ satisfying \eqref{eq:toy-I} and \eqref{eq:toy-S}, complete oscillator death happens for any initial data whenever $\kappa$ exceeds the fixed value given in \eqref{kappa-trivial} with $\mathcal{B}=[N]$:
\[
\kappa>\frac{\|\Omega\|_\infty}{\rho \min\{-\min S,\max S\}}.
\]
In other words, the pathwise critical coupling strength (recall Definition \ref{def:crit-IS}) is uniformly bounded as follows:
\[
\kappa_{\mathrm{pc}}(\Theta^0,\Omega,I,S)\le \frac{\|\Omega\|_\infty}{\rho \min\{-\min S,\max S\}}.
\]
Moreover, by choosing different $\mathcal{B}$, we obtain various sufficient coupling strengths for partial oscillator death. We have a full proof of the existence of the oscillator death regime (and partial oscillator death regime) in this case.

The rest of this section is an imitation of this ideal case. We will find conditions that serve as a proxy for \eqref{eq:toy-I} and thus guarantee oscillator death. The price for assuming these conditions will ultimately be paid either in the form of a large coupling strength depending on the initial data as in Corollary \ref{bestsincosmainthm} and Theorem \ref{mainthm}, or by sacrificing a certain singular set of initial data (which, to our advantage, turns out to have asymptotically small measure) as in Theorems \ref{themainthm}, \ref{sincosmaincor}, \ref{sincosmaincor-time}, \ref{maincor}, \ref{maincor-kappalarge}, \ref{themainthm-IS}, and \ref{maincor-quant}.

\subsection{The general interaction function case}\label{subsec:mainthm}

The purpose of this subsection is to prove Theorem \ref{mainthm} and Proposition \ref{general-pod}.

Recall conditions \eqref{c_1}, \eqref{c_2}, and \eqref{c_3}: there is a constant $\alpha_0\in (0,\pi)$ such that
\begin{equation*}\tag{\ref{c_1}}
    S(\theta)\le -c_1 (\pi-\theta)^p \mbox{ for }  \theta\in [\alpha_0,\pi],\quad S(\theta)\ge c_1 (\theta+\pi)^p \mbox{ for }  \theta\in [-\pi,-\alpha_0],
\end{equation*}
\[\tag{\ref{c_2}}
0\le I(\theta)\le c_2 (\pi-|\theta|)^q \mbox{ for }  \theta\in [-\pi,\pi],
\]
\[\tag{\ref{c_3}}
\min_{|\phi|\le \max\{|\theta|,\alpha_0\}}I(\phi)\ge c_3 I(\theta),\mbox{ for }\theta\in [-\pi, \pi].
\]

The starting point of the analysis is the observation that, if we make the a priori assumption that the order parameter $R(t)$ is uniformly bounded below in time, then we have some control on the behavior of some of the variables.
\begin{lemma}[A priori estimate]\label{gen-a-priori}
Let $\Theta=\Theta(t)$ be the solution to \eqref{GenWinfree} with parameters $\{\omega_i\}_{i=1}^N$, $\kappa$, and initial data $\{\theta_i^0\}_{i=1}^N$, and let $I$ and $S$ satisfy \eqref{c_1}, \eqref{c_2}, and \eqref{c_3}. Suppose $T>0$, $0<\rho \le \|I\|_{\sup}$, and $\mathcal{B}\subset [N]$ satisfy
\[
R(t)\ge \rho, ~\forall t\in [0,T), \quad and\quad \kappa>\frac{{\|\Omega_\mathcal{B}\|_\infty}}{\rho c_1 (\pi-\alpha_0)^p}.
\]
Let $i\in \mathcal{B}$ be such that $\theta_i^0\in \left[-\pi+\left(\frac{{\|\Omega_\mathcal{B}\|_\infty}}{\kappa \rho c_1}\right)^{1/p},\pi-\left(\frac{{\|\Omega_\mathcal{B}\|_\infty}}{\kappa \rho c_1}\right)^{1/p}\right]$. Then $|\theta_i(t)|\le\max\{ |\theta_i^0|,\alpha_0\}$ and $I(\theta_i(t))\ge c_3 I(\theta_i^0)$ for $t\in [0,T)$.
\end{lemma}
\begin{proof}
We compute that for any $t_0\in [0,T)$, if $\theta_i(t_0)\in \left[\alpha_0,\pi-\left(\frac{{\|\Omega_\mathcal{B}\|_\infty}}{\kappa \rho c_1}\right)^{1/p}\right]$ (note that our choice of $\kappa$ makes this a nonempty interval) then
\[
\dot{\theta}_i(t_0)=\omega_i+\kappa R(t_0)S(\theta_i(t_0))\stackrel{\eqref{c_1}}{\le} {\|\Omega_\mathcal{B}\|_\infty}-\kappa \rho c_1 (\pi-\theta_i(t_0))^p \le 0,
\]
and similarly if $\theta_i(t_0)\in \left[-\pi+\left(\frac{{\|\Omega_\mathcal{B}\|_\infty}}{\kappa \rho c_1}\right)^{1/p},-\alpha_0\right]$ then $\dot{\theta}_i(t_0)\ge 0$. By a standard exit-time argument, we have $|\theta_i(t)|\le\max\{ |\theta_i^0|,\alpha_0\}$ for $t\in [0,T)$. This then implies $I(\theta_i(t))\ge c_3 I(\theta_i^0)$ by \eqref{c_3}.
\end{proof}

All in all, Lemma \ref{gen-a-priori} says that if we have an a priori uniform-in-time lower bound for the average influence $R=\frac 1N\sum_{j=1}^N I(\theta_j)$, then for those $\theta_i$ that are far from $\pi$ (modulo $2\pi$) we have a lower bound on $I(\theta_i)$. If we can find a large enough number of such $\theta_i$'s, we would be able to deduce a lower bound on their average $R$. The main idea of our proof of Theorem \ref{mainthm} is to formalize this into a bootstrapping argument so as to guarantee a lower bound on $R(t)$ for all $t\ge 0$.

\begin{proof}[Proof of Proposition \ref{general-pod}]
Define
\[
\mathcal{T}\coloneqq \{T>0: R(t)> \rho \mbox{ for }t\in [0,T)\}.
\]
Noting from $c_3\le 1$ and \eqref{pod-cond}-(i) that 
\[
{R(0)}\ge \frac 1N \sum_{i\in \mathcal{A}}I(\theta_i^0)> \frac{\rho}{c_3}\ge\rho,
\]
by continuity of $R(t)$ we may find a small $\delta>0$ such that $R(t)>\rho $ for $t\in [0,\delta)$. It then follows that $\delta\in \mathcal{T}$, so $T^*\coloneqq \sup \mathcal{T}>0$. It is immediate that $\mathcal{T}\supset(0,T^*)$. 

Now by \eqref{pod-cond} and \eqref{pod-cond-2}, Lemma \ref{gen-a-priori} applies up to time $T^*$, so  we have $I(\theta_i(t))\ge c_3 I(\theta_i^0)$ for all $i\in \mathcal{A}$ and all $t\in [0,T^*)$, which gives
\[
R(t)\ge \frac 1N\sum_{i\in \mathcal{A}}I(\theta_i(t))\ge c_3 \cdot \frac 1N\sum_{i\in \mathcal{A}}I(\theta_i^0),\quad t\in [0,T^*).
\]
If we had $T^*<\infty$, then by continuity of $R$ we would have
\[
R(T^*)\ge  c_3 \cdot \frac 1N\sum_{i\in \mathcal{A}}I(\theta_i^0)\stackrel{\eqref{pod-cond}-(\mathrm{i})}{>}\rho,
\]
and again by continuity of $R$ there would exist a $\delta>0$ such that $R(t)>\rho$ for $t\in [T^*,T^*+\delta)$, contradicting the definition of $T^*$. Thus, we must have $T^*=\infty$, i.e., $R(t)> \rho$ for all $t\ge 0$. This is assertion (a) of Proposition \ref{general-pod}.

Since Lemma \ref{gen-a-priori} applies up to time $T^*=\infty$, we have assertion (b) of Proposition \ref{general-pod}.

It remains to prove assertion (c) of Proposition \ref{general-pod}, i.e., $\sup_{t\ge 0} \theta_i(t)-\inf_{t\ge 0}\theta_i(t)<2\pi$ for all $i\in\mathcal{B}$. Fix $i\in\mathcal B$. If $\|\Omega_\mathcal B\|_\infty=0$, then $\dot\theta_i=\kappa R(t)S(\theta_i)$ with $R(t)>\rho$. Conditions \eqref{c_1} and periodicity give $S(\pi)=0$ and at least one further zero in $[-\alpha_0,\alpha_0]$. Uniqueness traps each scalar trajectory between consecutive zeros, whose gap is strictly less than $2\pi$, or makes it stationary at a zero. Thus the conclusion holds. Assume henceforth $\|\Omega_\mathcal B\|_\infty>0$ and consider two cases.
\begin{itemize}
\item[Case 1.] $\theta_i^0\in \left[-\pi+\left(\frac{\|\Omega_\mathcal{B}\|_\infty}{\kappa \rho c_1}\right)^{1/p},\pi-\left(\frac{\|\Omega_\mathcal{B}\|_\infty}{\kappa \rho c_1}\right)^{1/p}\right]\mod 2\pi$.

Then, by Lemma \ref{gen-a-priori}, we have $\theta_i(t)\in \left[- \max\{|\theta_i^0|,\alpha_0\},\max\{|\theta_i^0|,\alpha_0\}\right] \mod 2\pi$ for all $t\ge 0$, from which the assertion $\sup_{t\ge 0} \theta_i(t)-\inf_{t\ge 0}\theta_i(t)<2\pi$ easily follows.

\item[Case 2.] $\theta_i^0\notin \left[-\pi+\left(\frac{\|\Omega_\mathcal{B}\|_\infty}{\kappa \rho c_1}\right)^{1/p},\pi-\left(\frac{\|\Omega_\mathcal{B}\|_\infty}{\kappa \rho c_1}\right)^{1/p}\right] \mod 2\pi$.

Then either $\theta_i(t)\notin \left[-\pi+\left(\frac{\|\Omega_\mathcal{B}\|_\infty}{\kappa \rho c_1}\right)^{1/p},\pi-\left(\frac{\|\Omega_\mathcal{B}\|_\infty}{\kappa \rho c_1}\right)^{1/p}\right] \mod 2\pi$ for all $t\ge 0$, in which case $\theta_i(t)$ is trapped in an interval of length $<2\pi$ for all time and the assertion follows, or there will be the smallest time $t^*>0$ at which $\theta_i(t^*)=-\pi+\left(\frac{\|\Omega_\mathcal{B}\|_\infty}{\kappa \rho c_1}\right)^{1/p}\mbox{ or }\pi-\left(\frac{\|\Omega_\mathcal{B}\|_\infty}{\kappa \rho c_1}\right)^{1/p}\mod 2\pi$. In the latter case, we can apply Lemma \ref{gen-a-priori} starting at time $t^*$ (this is possible since \eqref{GenWinfree} is a time-autonomous system) and deduce that the set $\{\theta_i(t):t\ge 0\}$ will be contained in an interval either of the form $\left(-\pi-\left(\frac{\|\Omega_\mathcal{B}\|_\infty}{\kappa \rho c_1}\right)^{1/p},\pi-\left(\frac{\|\Omega_\mathcal{B}\|_\infty}{\kappa \rho c_1}\right)^{1/p}\right]\mod 2\pi$ or $\left[-\pi+\left(\frac{\|\Omega_\mathcal{B}\|_\infty}{\kappa \rho c_1}\right)^{1/p},\pi+\left(\frac{\|\Omega_\mathcal{B}\|_\infty}{\kappa \rho c_1}\right)^{1/p}\right)\mod 2\pi$, and the phase range is strictly less than $2\pi$ by a simple compactness argument.
\end{itemize}
\end{proof}

We are now ready to prove Theorem \ref{mainthm}. We will apply our criterion for partial oscillator death given in Proposition \ref{general-pod}. By choosing the coupling strength $\kappa$ to be sufficiently large depending on the initial data, we will be able to find a subensemble $\mathcal{A}$ which satisfies the conditions \eqref{pod-cond}-(i) and \eqref{pod-cond-2}.

\begin{proof}[Proof of Theorem \ref{mainthm}]
After modulo $2\pi$ shifts, we may assume that $\theta_i^0\in [-\pi,\pi)$ for all $i\in [N]$. Define
\[
\mathcal{A}=\{i\in [N]:I(\theta_i^0)\ge \frac{{R_0}}{2} \}.
\]
Then
\begin{equation}\label{gen_lb}
\frac 1N\sum_{i\in \mathcal{A}}I(\theta_i^0)\stackrel{\eqref{gen_order_parameter}}{=}{R_0}-\frac 1N\sum_{i\in [N]\setminus \mathcal{A}}I(\theta_i^0)>{R_0}-\frac{{R_0}}{2}=\frac{{R_0}}{2}.
\end{equation}
We can also see by \eqref{c_2} that 
\begin{equation}\label{gen_range}
|\theta_i^0|\le \pi-\left(\frac{{R_0}}{2c_2}\right)^{1/q}\stackrel{\eqref{gen_large_kappa}-(\mathrm{i})}{\le} \pi -\left(\frac{{\|\Omega\|_\infty}}{\kappa \cdot \frac{c_3 R_0}{2}\cdot c_1}\right)^{1/p},\quad \forall i\in \mathcal{A}.
\end{equation}

We may now see that \eqref{gen_lb} and \eqref{gen_range} satisfies the hypotheses of Proposition \ref{general-pod} with $\rho=\frac{c_3 R_0}{2}$ and $\mathcal{B}=[N]$, with \eqref{gen_large_kappa}-(ii) becoming \eqref{pod-cond}-(ii). Now Proposition \ref{general-pod} (a,c) give statements (a) and (b).
\end{proof}

\subsection{The special interaction function case}\label{subsec:sincosmainthm}

The purpose of this subsection is to prove Theorem \ref{generalbestsincosmainthm} and derive Corollary \ref{bestsincosmainthm} from it. We will prove Proposition \ref{bootstrap} along the way. Note that, with Theorem \ref{mainthm} at hand, we already have a version of Corollary \ref{bestsincosmainthm} with worse constants. We are presenting a refined version of the proof of Theorem \ref{mainthm} to show how one may optimize the analysis to obtain a reasonable coupling strength that is closer to the range considered in \cite{ariaratnam2001phase}. Thus, in this subsection, we will try to be precise in our analysis to get good constants.

As in the previous subsection, the starting point of the analysis is the observation that, if we make the a priori assumption that the order parameter $R(t)$ is uniformly bounded below in time, then we have some control on the dynamics of some of the variables. The following is the analogue of Lemma \ref{gen-a-priori} for system \eqref{Winfree}.
\begin{lemma}[A priori estimate]\label{Simplest-a-priori}
Let $\Theta=\Theta(t)$ be the solution to \eqref{Winfree} with parameters $\{\omega_i\}_{i=1}^N$, $\kappa$, and $\{\theta_i^0\}_{i=1}^N$. Suppose $T>0$, $0<\rho \le 2$, and $\mathcal{B}\subseteq[N]$ satisfy
\[
R(t)\ge \rho, ~\forall t\in [0,T), \quad \mathrm{and}\quad \kappa>\frac{{\|\Omega_\mathcal{B}\|_\infty}}{\rho}.
\]
Let $i\in \mathcal{B}$. The following hold.
\begin{enumerate}[(a)]
\item If $-\sqrt{1-\frac{{\omega_i}^2}{\kappa^2 \rho^2}}\le \cos \theta^0_i\le \sqrt{1-\frac{{\omega_i}^2}{\kappa^2 \rho^2}}$, then $\cos \theta_i(t)\ge \cos\theta_i^0$ for $t\in [0,T)$.
\item If $\cos \theta^0_i> \sqrt{1-\frac{{\omega_i}^2}{\kappa^2 \rho^2}}$, then $\cos \theta_i(t)\ge \sqrt{1-\frac{{\omega_i}^2}{\kappa^2 \rho^2}}$ for $t\in [0,T)$.
\item We have $\cos \theta_i(t)\ge \cos \theta_i^0-1+\sqrt{1-\frac{{\omega_i}^2}{\kappa^2 \rho^2}}$ for $t\in [0,T)$.
\item For a given  $0<\alpha< \sqrt{1-\frac{{\omega_i}^2}{\kappa^2 \rho^2}}$,  if $T>\frac{\pi}{\kappa \rho \sqrt{1-\alpha^2}-{|\omega_i|}}$ and if $\cos\theta_i^0\ge -\alpha$, then  $\cos \theta_i(t)\ge \alpha$ for $t\in \left[\frac{\pi}{\kappa \rho \sqrt{1-\alpha^2}-{|\omega_i|}},T\right)$.
\setcounter{name}{\value{enumi}}
\end{enumerate}

\end{lemma}
\begin{proof}
We first observe that for any $t_0\in [0,T)$ such that $\cos\theta_i(t_0)\in \left[-\sqrt{1-\frac{{\omega_i}^2}{\kappa^2 \rho^2}},\sqrt{1-\frac{{\omega_i}^2}{\kappa^2 \rho^2}}\right]$, we have
\begin{align*}
\left.\frac{d}{dt}\right|_{t=t_0}\cos\theta_i=&-\dot{\theta}_i(t_0) \sin \theta_i(t_0)=-\omega_i\sin\theta_i(t_0)+\kappa R(t_0)\sin^2\theta_i(t_0)\\
\ge& (-|\omega_i|+\kappa \rho |\sin\theta_i(t_0)|)|\sin\theta_i(t_0)|\ge 0.
\end{align*}
In other words, $\cos\theta_i(t)$ behaves as a non-decreasing function of $t$ whenever its value lies in  \linebreak$\left[-\sqrt{1-\frac{{\omega_i}^2}{\kappa^2 \rho^2}},\sqrt{1-\frac{{\omega_i}^2}{\kappa^2 \rho^2}}\right]$. A simple comparison argument gives assertions (a) and (b).

Assertion (c) immediately follows from assertions (a) and (b).

To show assertion (d), we calculate that, for any $t_0\in [0,T)$ such that $\theta_i(t_0)\in [\cos^{-1}\alpha,\pi-\cos^{-1}\alpha]\mod 2\pi$, we have
\[
\dot{\theta}_i(t_0)=\omega_i-\kappa R(t_0)\sin \theta_i(t_0)\le |\omega_i|-\kappa \rho \sqrt{1-\alpha^2}<0.
\]
Similarly, if $\theta_i(t_0)\in [-\pi+\cos^{-1}\alpha,-\cos^{-1}\alpha]\mod 2\pi$, we have
\[
\dot{\theta}_i(t_0)\ge -|\omega_i|+\kappa \rho \sqrt{1-\alpha^2}>0.
\]
A simple comparison argument now gives assertion (d).
\end{proof}
As a direct corollary, the Winfree model \eqref{Winfree} is well-controlled if we have a positive lower bound on the order parameter for all time. Compared to the previous subsection, additional information about the asymptotic dynamics follows thanks to the {\L}ojasiewicz gradient theorem (Proposition \ref{Loja}).
\begin{lemma}\label{simple-cor-1}
Let $\Theta=\Theta(t)$ be the solution to \eqref{Winfree} with parameters $\{\omega_i\}_{i=1}^N$, $\kappa$, and $\{\theta_i^0\}_{i=1}^N$. Suppose $0<\rho \le 2$ and $\mathcal{B}\subseteq [N]$ satisfy
\[
R(t)\ge \rho, ~\forall t\in [0,\infty), \quad and\quad \kappa>\frac{{\|\Omega_\mathcal{B}\|_\infty}}{\rho}.
\]
\begin{enumerate}[(a)]
\item For $i\in \mathcal{B}$, we have $\cos\theta_i(t)\ge \cos\theta_i^0 -1 + \sqrt{1-\frac{\omega_i^2}{\kappa^2 \rho^2}}$ for all $t\ge 0$.
\item Let $i\in \mathcal{B}$ and $0<\alpha< \sqrt{1-\frac{{\omega_i}^2}{\kappa^2 \rho^2}}$ be such that $\cos \theta^0_i\ge-\alpha$. Then $\cos \theta_i(t)\ge \alpha$ for $t\ge \frac{\pi}{\kappa \rho \sqrt{1-\alpha^2}-{|\omega_i|}}$.
\item Let $i,j\in \mathcal{B}$ and $0<\alpha< \sqrt{1-\frac{\omega_i^2 \vee \omega_j^2}{\kappa^2 \rho^2}}$ be such that $\cos \theta^0_i\ge \alpha$ and $\cos \theta^0_j\ge \alpha$. If $\omega_i>\omega_j$, then after modulo $2\pi$ translations,
\[
\theta_i(t)-\theta_j(t)\le \frac{\omega_i-\omega_j}{\kappa\rho\alpha}+\pi\exp\left(-\kappa\rho\alpha t\right)\mbox{ for }t\ge 0,
\]
and
\begin{align*}
\theta_i(t)-\theta_j(t)\ge \frac{\omega_i-\omega_j}{2\kappa}-\frac{\omega_i-\omega_j}{2\kappa}\exp\left(-2\kappa\left(t-\frac{\pi}{\omega_i-\omega_j}\right)\right)\quad
\mbox{for }t\ge \frac{\pi}{\omega_i-\omega_j}.
\end{align*}
If $\omega_i=\omega_j$, then after modulo $2\pi$ translations, $\lim_{t\to\infty}\left(\theta_i(t)-\theta_j(t)\right)=0$ exponentially:
\[
|\theta_i(t)-\theta_j(t)|\le \pi \exp\left(-\kappa\rho\alpha t\right)\mbox{ for }t\ge 0.
\]
\item For all $i\in \mathcal{B}$,
\[
\sup_{t\ge 0} \theta_i(t)-\inf_{t\ge 0}\theta_i(t)<2\pi.
\]
\item If $\mathcal{B}=[N]$, then for all $i\in [N]$, the limit $\lim_{t\to\infty}\theta_i(t)$ exists, and
\[
\lim_{t\to\infty}\dot{\theta}_i(t)=0.
\]
\end{enumerate}
\end{lemma}
\begin{proof}
Assertions (a) and (b) follow immediately from Lemma \ref{Simplest-a-priori}.

For assertion (c), note that by Lemma \ref{Simplest-a-priori} that $\theta_i(t),\theta_j(t)\in [-\cos^{-1}\alpha,\cos^{-1}\alpha]$ for $t\ge 0$. Using $\dot\theta_i-\dot\theta_j=\omega_i-\omega_j-\kappa R(\sin\theta_i-\sin\theta_j)$, we have
\begin{equation}\label{eq:diffineq-pos}
	\theta_i(t)\ge \theta_j(t)\Rightarrow \omega_i-\omega_j-2\kappa(\theta_i(t)-\theta_j(t)) \le\dot\theta_i(t)-\dot\theta_j(t)\le \omega_i-\omega_j-\kappa\rho\alpha(\theta_i(t)-\theta_j(t))
\end{equation}
\begin{equation}\label{eq:diffineq-neg}
	\theta_i(t)\le \theta_j(t)\Rightarrow \dot\theta_i(t)-\dot\theta_j(t)\ge \omega_i-\omega_j-\kappa\rho\alpha(\theta_i-\theta_j)\ge \omega_i-\omega_j.
\end{equation}
Solving these differential inequalities gives the desired statement.

To prove assertion (d), fix an $i\in \mathcal{B}$. If $\omega_i=0$, the equation $\dot\theta_i=-\kappa R\sin\theta_i$ makes each interval between consecutive multiples of $\pi$ invariant, so the phase range is at most $\pi$. We may therefore assume $\omega_i\ne0$ in the following cases.
\begin{itemize}
\item[Case 1.] $\theta_i^0\in \left[-\cos^{-1}\Big(-\sqrt{1-\frac{{\omega_i}^2}{\kappa^2 \rho^2}}\Big),\cos^{-1}\Big(-\sqrt{1-\frac{{\omega_i}^2}{\kappa^2 \rho^2}}\Big)\right]\mod 2\pi$.

Then $\theta_i(t)\in \left[-\cos^{-1}\Big(-\sqrt{1-\frac{{\omega_i}^2}{\kappa^2 \rho^2}}\Big),\cos^{-1}\Big(-\sqrt{1-\frac{{\omega_i}^2}{\kappa^2 \rho^2}}\Big)\right]\mod 2\pi$ for all $t\ge 0$, and $\theta_i(t)$ will belong to the same connected component, which has length $<2\pi$. Assertion (d) follows in this case.
\item[Case 2.] $\theta_i^0\in \left(\cos^{-1}\Big(-\sqrt{1-\frac{{\omega_i}^2}{\kappa^2 \rho^2}}\Big),2\pi-\cos^{-1}\Big(-\sqrt{1-\frac{{\omega_i}^2}{\kappa^2 \rho^2}}\Big)\right)\mod 2\pi$.

Then either $\theta_i(t)\in \left(\cos^{-1}\Big(-\sqrt{1-\frac{{\omega_i}^2}{\kappa^2 \rho^2}}\Big),2\pi-\cos^{-1}\Big(-\sqrt{1-\frac{{\omega_i}^2}{\kappa^2 \rho^2}}\Big)\right)\mod 2\pi$ for all $t\ge 0$ while staying in the same connected component, in which case assertion (d) follows, or there will be the smallest time $t^*>0$ at which $\theta_i(t^*)=\cos^{-1}\Big(-\sqrt{1-\frac{{\omega_i}^2}{\kappa^2 \rho^2}}\Big)\mbox{ or }2\pi-\cos^{-1}\Big(-\sqrt{1-\frac{{\omega_i}^2}{\kappa^2 \rho^2}}\Big)\mod 2\pi$. In the latter case, we can apply Lemma \ref{Simplest-a-priori} starting at time $t^*$ (this is possible since \eqref{Winfree} is a time-autonomous system) and deduce that the set $\{\theta_i(t):t\ge 0\}$ will be contained in an interval either of the form 
\[
\left(\cos^{-1}\Big(-\sqrt{1-\frac{{\omega_i}^2}{\kappa^2 \rho^2}}\Big),2\pi+\cos^{-1}\Big(-\sqrt{1-\frac{{\omega_i}^2}{\kappa^2 \rho^2}}\Big)\right]\mod 2\pi
\]
or
\[
\left[-\cos^{-1}\Big(-\sqrt{1-\frac{{\omega_i}^2}{\kappa^2 \rho^2}}\Big),2\pi-\cos^{-1}\Big(-\sqrt{1-\frac{{\omega_i}^2}{\kappa^2 \rho^2}}\Big)\right)\mod 2\pi,
\]
and assertion (d) holds with a strict inequality, by a simple compactness argument.
\end{itemize}
Assertion (e) follows from assertion (d) by Proposition \ref{Loja}.

\end{proof}

To guarantee the a priori lower bound, we can use a bootstrapping argument as in the previous subsection, which allows us to deduce the criterion for partial oscillator death, namely Proposition \ref{bootstrap}.

\begin{proof}[Proof of Proposition \ref{bootstrap}]
It is clear, by Lemma \ref{simple-cor-1}, that (b)-(e) follow once we prove (a). To prove (a), define
\[
\mathcal{T}\coloneqq \{T>0: R(t)\ge \rho \mbox{ for }t\in [0,T)\}.
\]
Noting that
\[
{R_0}\ge \frac 1N\sum_{i\in \mathcal{B}} (1+\cos\theta_i^0)\ge  \frac 1N\sum_{i\in \mathcal{B}} \left(\cos\theta_i^0+\sqrt{1-\frac{\omega_i^2}{\kappa^2\rho^2}}\right)^+\stackrel{\mathclap{\eqref{part_is_big}}}{>} \rho,
\]
by continuity of $R(t)$ we may find a small $\delta>0$ such that $R(t)>\rho$ for $t\in [0,\delta)$. This proves that $\delta\in \mathcal{T}$, i.e., $\mathcal{T}$ is not empty. It is then immediate that $\mathcal{T}\supset(0,T^*)$, where $T^*\coloneqq \sup \mathcal{T}>0$. Suppose, for contradiction, that $T^*<\infty$. Then, for any $t\in [0,T^*)$, we have
\[
R(t)\ge \frac 1N \sum_{i\in \mathcal{B}}(1+\cos\theta_i(t))\stackrel{\mathrm{Lemma }~\ref{Simplest-a-priori} (c)}{\ge} \frac 1N\sum_{i\in \mathcal{B}} \left(\cos\theta_i^0+\sqrt{1-\frac{\omega_i^2}{\kappa^2\rho^2}}\right)^+
\]
so that by continuity of $R(t)$, we have
\[
R(T^*)\ge\frac 1N\sum_{i\in \mathcal{B}} \left(\cos\theta_i^0+\sqrt{1-\frac{\omega_i^2}{\kappa^2\rho^2}}\right)^+\stackrel{\mathclap{\eqref{part_is_big}}}{>} \rho.
\]
Again by continuity of $R(t)$, we may find a small $\delta>0$ such that $R(t)>\rho$ for $t\in [0,T^*+\delta)$. 
This shows $T^*+\delta\in \mathcal{T}$, contradicting the definition of $T^*$. Hence, we must have $T^*=\infty$, and this completes the proof of assertion (a).
\end{proof}

We are now ready to prove Theorem \ref{generalbestsincosmainthm}.
\begin{proof}[Proof of Theorem \ref{generalbestsincosmainthm}]
Let
\[
\rho={R_0}-\mu>0.
\]
We easily observe that condition \eqref{large_kappa_general} implies
\begin{equation}\label{condition}
-1+\mu > -\sqrt{1-\frac{\|\Omega\|_\infty^2}{\kappa^2\rho^2}}.
\end{equation}
We now confirm that our choice of $\rho$ along with $\mathcal{B}=[N]$ satisfies the hypotheses of Proposition \ref{bootstrap}. Indeed, \eqref{part_is_big} follows from
\begin{equation*}
\frac 1N\sum_{i=1}^N \left(\cos\theta_i^0+\sqrt{1-\frac{\omega_i^2}{\kappa^2\rho^2}}\right)^+\ge \frac 1N\sum_{i=1}^N (1+\cos\theta_i^0)-1+\sqrt{1-\frac{\|\Omega\|_\infty^2}{\kappa^2\rho^2}}\stackrel{\mathclap{\eqref{condition}}}{>}R_0-\mu=\rho,
\end{equation*}
and \eqref{partial_large_kappa} follows from
\begin{equation*}
\kappa^2\stackrel{\eqref{large_kappa_general}}{>}\frac{\|\Omega\|_\infty^2}{\rho^2\mu(2-\mu)}\ge \frac{\|\Omega\|_\infty^2}{\rho^2}.
\end{equation*}

Assertion (a) of Theorem \ref{generalbestsincosmainthm} follows from assertion (a) of Proposition \ref{bootstrap}, assertion (b) of Theorem \ref{generalbestsincosmainthm} follows from assertions (c) and (e) of Proposition \ref{bootstrap}. Assertion (c) of Theorem \ref{generalbestsincosmainthm} follows from assertion (b) of Proposition \ref{bootstrap} by setting $\alpha=1-\mu$, because of \eqref{condition} and $\cos\theta_i(t_0)\ge -1+\mu=-\alpha$ by definition, so $\cos\theta_i(t)\ge 1-\mu$ for 
\[
t\ge t_0+\frac{\pi}{\kappa \rho \sqrt{1-\alpha^2}-\|\Omega\|_\infty}=t_0+\frac{\pi}{\kappa \rho \sqrt{\mu(2-\mu)}-\|\Omega\|_\infty}.
\]
Likewise assertion (e) of Theorem \ref{generalbestsincosmainthm} follows from assertion (d) of Proposition \ref{bootstrap}.

For assertion (d) of Theorem \ref{generalbestsincosmainthm}, note that if we denote $R_\infty=\lim_{t\to\infty}R(t)$, then by assertion (b) of Theorem \ref{generalbestsincosmainthm} and \eqref{Winfree_orderparam} we have $\lim_{t\to\infty} \sin \theta_i(t)=\frac{\omega_i}{\kappa R_\infty}$, so that
\[
\lim_{t\to\infty} \theta_i(t)=\sin^{-1}\frac{\omega_i}{\kappa R_\infty}\mathrm{~or~}\pi-\sin^{-1}\frac{\omega_i}{\kappa R_\infty}\mod 2\pi.
\]
Since $R_\infty\ge\rho$, we have
\[
\cos\left(\pi-\sin^{-1}\frac{\omega_i}{\kappa R_\infty}\right)
=-\sqrt{1-\frac{\omega_i^2}{\kappa^2R_\infty^2}}
\le-\sqrt{1-\frac{\|\Omega\|_\infty^2}{\kappa^2\rho^2}}
<-1+\mu.
\]
Thus an orbit on the negative-cosine limiting branch has its supremum of $\cos\theta_i$ strictly below $-1+\mu$: equality at a finite time would trigger part (c), while a supremum approached only as $t\to\infty$ would contradict the displayed strict limiting inequality. This proves the dichotomy in (d).
\end{proof}

We are now ready to prove Corollary \ref{bestsincosmainthm}, which turns out to be special cases of Theorem \ref{generalbestsincosmainthm}.
\begin{proof}[Proof of Corollary \ref{bestsincosmainthm}]
Assertions (a), (b), (c),  and (d) of Corollary \ref{bestsincosmainthm} just follow from the case
\[
\mu=\frac{3+{R_0}-\sqrt{{R_0}^2-2{R_0}+9}}{4}\in (0,\min\{R_0,1\})
\]
of Theorem \ref{generalbestsincosmainthm} and from \begin{equation}\label{eq:sqrt-ineq}
\sqrt{R_0^2-2{R_0}+9}\ge 3-\frac 13 {R_0}.
\end{equation}
Indeed, we need to check that \eqref{large_kappa_general} holds. Because of \eqref{large_kappa_1} and \eqref{large_kappa_2}, it is enough to check
\begin{equation}\label{technical2}
K_c^2\ge\frac{1}{\rho^2\mu (2-\mu )}.
\end{equation}
We will verify this in Appendix A using elementary one-variable calculus. Also, we verify that
\[
R_0-\mu=\frac{3{R_0}-3+\sqrt{{R_0}^2-2{R_0}+9}}{4}\stackrel{\eqref{eq:sqrt-ineq}}{>}\frac{2}{3}R_0,
\]
\[
1-\mu=\frac{1-{R_0}+\sqrt{{R_0}^2-2{R_0}+9}}{4}\stackrel{\eqref{eq:sqrt-ineq}}{>}1-\frac{1}{3}R_0,
\]
and we verify the time that appears in statements (c) follows from the following upper bound:
\[
\frac{\pi}{\kappa \rho \sqrt{\mu(2-\mu)}-\|\Omega\|_\infty}\stackrel{\eqref{technical2}}{\le} \frac{\pi}{\rho\sqrt{\mu(2-\mu)}}(\kappa-K_c\|\Omega\|_\infty)^{-1}\stackrel{\eqref{technical2}}{\le}\frac{\pi K_c}{\kappa-K_c\|\Omega\|_\infty}.
\]

For assertion (e), retain the optimized value $\mu$ used above and the resulting bound $R(t)\ge\rho=R_0-\mu>2R_0/3$. Set $\alpha=1-R_0/3$. We have
$0<\alpha<1-\mu<\sqrt{1-\|\Omega\|_\infty^2/(\kappa^2\rho^2)}$.
Apply Lemma \ref{simple-cor-1}(c), starting at time $t_0$, and use
\[
\kappa\rho\alpha\ge\frac{2\kappa R_0(3-R_0)}9.
\]
Weakening the contraction rate to this lower bound gives precisely the stated upper and equal-frequency estimates; the lower estimate is unchanged.

Assertions (f), (g), (h) and (i) of Corollary \ref{bestsincosmainthm} follow from the case $\mu=\frac{4\|\Omega\|_\infty^2}{\kappa^2R_0^2}$ of Theorem \ref{generalbestsincosmainthm}. Note that because $\kappa>\frac{2\sqrt{2}\|\Omega\|_\infty}{R_0^{3/2}}$, we have $\mu<\frac{R_0}{2}\le \min\{R_0,1\}$. Condition \eqref{large_kappa_general} follows from $2-\mu\ge 1$ and $R_0-\mu> R_0/2$. The time that appears in statement (g) follows from the following upper bound:
\[
\frac{\pi}{\kappa(R_0-\mu)\sqrt{\mu(2-\mu)}-\|\Omega\|_\infty}\le \frac{\pi}{\kappa (R_0-4\|\Omega\|^2_\infty/\kappa^2R_0^2)\cdot 2\|\Omega\|_\infty/\kappa R_0-\|\Omega\|_\infty}=\frac{\pi}{\|\Omega\|_\infty\left(1-\frac{8\|\Omega\|_\infty^2}{\kappa^2R_0^3}\right)}.
\]

\end{proof}


\section{Volumetric arguments}\label{sec:large_deviations}
The purpose of this section is to prove Theorems \ref{themainthm}, \ref{sincosmaincor}, \ref{sincosmaincor-time}, \ref{maincor}, \ref{maincor-kappalarge}, \ref{themainthm-IS}, and \ref{maincor-quant}. First, in subsection \ref{subsec:large_deviations}, we will recall some basic large deviations theory to provide estimates on the volume of certain singular sets and then derive Theorems \ref{maincor} and \ref{maincor-kappalarge} from Theorem \ref{mainthm}, and Theorem \ref{sincosmaincor} from Corollary \ref{bestsincosmainthm}. Then, in subsection \ref{subsec:instability}, we will use the divergence to examine the evolution of the volume of the singular sets, to prove Theorem \ref{themainthm} from Theorem \ref{generalbestsincosmainthm}, Theorem \ref{sincosmaincor-time} from Corollary \ref{bestsincosmainthm}, and Theorems \ref{themainthm-IS} and \ref{maincor-quant} from Theorem \ref{mainthm}.

\subsection{Some large deviations theory}\label{subsec:large_deviations}
We first begin by discussing some time-independent large deviations theory to prove Theorems \ref{sincosmaincor}, \ref{maincor}, and \ref{maincor-kappalarge}.

In this paragraph, we recall some standard results and notations in large deviations theory (the statements of this paragraph can all be found in \cite[Section 3.1]{van2014probability}). For a given $\sigma>0$, a real-valued random variable $X$ is said to be $\sigma^2$-\textit{subgaussian} if $X$ is integrable and
\[
\mathbb{E}[\exp(\lambda(X-\mathbb{E}X))]\le \exp\left(\frac{\lambda^2 \sigma^2}{2}\right),\quad \forall \lambda\in \mathbb{R}.
\]
It is clear that if the independent random variables $X$ and $Y$ are $\sigma^2$-subgaussian and $\tau^2$-subgaussian, respectively, then $X+Y$ is $(\sigma^2+\tau^2)$-subgaussian. The routine Chernoff bound states that if $X$ is a $\sigma^2$-subgaussian random variable, then
\begin{equation}\label{Chernoff}
\mathbb{P}[X\ge\mathbb{E}X+ t],\mathbb{P}[X\le \mathbb{E}X- t]\le \exp\left(-\frac{t^2}{2\sigma^2}\right),\quad \forall  t>0.
\end{equation}
Thus, if $X_1,\cdots,X_n$ are independent random variables that are each $\sigma^2$-subgaussian, then $\frac{X_1+\cdots+X_n}{n}$ is $\sigma^2/n$-subgaussian so that
\[
\mathbb{P}\left[\frac{X_1+\cdots+X_n}{n}\ge\frac{\mathbb{E}X_1+\cdots+\mathbb{E}X_n}{n}+ t\right],\mathbb{P}\left[\frac{X_1+\cdots+X_n}{n}\le\frac{\mathbb{E}X_1+\cdots+\mathbb{E}X_n}{n}- t\right]\le \exp\left(-\frac{nt^2}{2\sigma^2}\right),\quad \forall  t>0.
\]
In addition, the well-known Hoeffding lemma states that if $X$ is a real-valued random variable, with $a\le X\le b$ almost surely for some $a,b\in \mathbb{R}$, then $X$ is $(b-a)^2/4$-subgaussian.

With these basic tools, we are ready to prove Theorem \ref{maincor}.

\begin{proof}[Proof of Theorem \ref{maincor}]
The distribution $\mu^{\otimes N}$ means that we are independently and identically sampling $\theta_i^0\sim \mu$, so ${R_0}=\frac 1N\sum_{j=1}^N I(\theta_j^0)$ is the sum of the independent random variables $\frac 1N I(\theta_j^0)$, $j=1,\cdots, N$. By the Hoeffding lemma, $\frac 1N I(\theta_j^0)$ is $\frac{\|I\|_{\sup}^2}{4N^2}$-subgaussian, so by independence, ${R_0}=\sum_{j=1}^N \frac 1N I(\theta_j^0)$ is $\frac{\|I\|_{\sup}^2}{4N}$-subgaussian. As $\mathbb{E}_{\mu^{\otimes N}}{R_0}=R^*$, we conclude

\begin{align*}
&\mu^{\otimes N}\Big\{\{\theta_i^0\}_{i=1}^N\in [-\pi,\pi]^N:\Theta(t)\mbox{ satisfies }R(t)> \frac{c_3 R^*}{4}~\forall t\ge 0 \mbox{ and }\sup_{t\ge 0} \theta_i(t)-\inf_{t\ge 0}\theta_i(t)<2\pi ~\forall i\in [N]\Big\}\\
&\stackrel{\mathclap{\mathrm{Theorem~} \ref{mainthm},\eqref{cor-kappa-large}}}{\ge}\qquad \mu^{\otimes N}\Big\{\{\theta_i^0\}_{i=1}^N\in [-\pi,\pi]^N:{R_0}>\frac{R^*}2\Big\}\\
&=1-\mu^{\otimes N}\left\{\{\theta_i^0\}_{i=1}^N\in [-\pi,\pi]^N:{R_0}\le R^*-\frac{R^*}2\right\}\\
&\stackrel{\eqref{Chernoff}}{\ge} 1-\exp\left(-\frac{(R^*/2)^2}{2\cdot \|I\|_{\sup}^2 /4N}\right)=1-\exp\left(-\frac{(R^*)^2N}{2 \|I\|_{\sup}^2 }\right).
\end{align*}
\end{proof}

If the distribution $\mu$ is such that $I(\theta^0)$, $\theta^0\sim\mu$ has better subgaussian concentration, we may improve the measure estimate of Theorem \ref{maincor}. We will exploit better concentration behavior in Theorem \ref{sincosmaincor}.

\begin{proposition}\label{prop:circleConc}
    Suppose $X\in [-\pi,\pi]$ is a random variable that is uniform with respect to Lebesgue measure. Then $\cos X$ is $1/2$-subgaussian.
\end{proposition}
\begin{proof}
Symmetry gives $\mathbb E\cos X=0$. Expanding the exponential and integrating term by term gives
\[
\mathbb E e^{\lambda\cos X}
=\frac1{2\pi}\int_{-\pi}^{\pi}e^{\lambda\cos\theta}\,d\theta
=\sum_{m=0}^{\infty}\frac{(\lambda^2/4)^m}{(m!)^2}
\le\sum_{m=0}^{\infty}\frac{(\lambda^2/4)^m}{m!}
=e^{\lambda^2/4},\qquad\lambda\in\mathbb R.
\]
This is precisely $1/2$-subgaussianity.
\end{proof}

On the other hand, given a random variable $X$ and a real number $M$, if we know the concrete upper bound $X\le M$, what can we say about how close $X$ can be to $M$?
Trivially $\mathbb{E}e^{\lambda X}\le e^{\lambda M}$ for $\lambda>0$. If $\mathbb{P}[X\ge M-\varepsilon]\lesssim \varepsilon^\alpha$  for some $\alpha>0$ (this can happen, say, if the distribution of $X$ is absolutely continuous with respect to Lebesgue measure near $M$ and the density function is of the power type $O((M-x)^{\alpha-1})$ for $x<M$ near $M$), then we have the bound $\mathbb{E}e^{\lambda X}\le O(e^{\lambda M}/{\lambda^\alpha})+e^{\lambda (M-\Omega(1))}$ which is stronger for large $\lambda$. Conversely, if it is the case that we have the bound $\mathbb{E}e^{\lambda X}\le \frac{C}{\lambda^\alpha}e^{\lambda M}$ for some constants $C,\alpha>0$, then we know the distribution of $X$ has power-type behavior near $M$, as follows; this has implications for sums of i.i.d.~bounded random variables of this power-type behavior.
\begin{lemma}\label{lem:extreme-close}
    Let $X$ be a random variable such that for some constants $C,\alpha>0$, we have $\mathbb{E}e^{\lambda X}\le \frac{C}{\lambda^\alpha}e^{\lambda M}$ for all $\lambda>0$. Then
    \[
    \mathbb{P}[X\ge M-\varepsilon]\le C\left(\frac{e\varepsilon}{\alpha}\right)^\alpha,\quad \varepsilon>0.
    \]
    Furthermore, if $X_1,\cdots,X_N$ are independent and identically distributed copies of $X$, where $N\in \mathbb{N}$, then
    \[
    \mathbb{P}\left[\frac{X_1+\cdots+X_N}{N}\ge M-\varepsilon\right]\le \left(C\left(\frac{e\varepsilon}{\alpha}\right)^\alpha\right)^N.
    \]
\end{lemma}
\begin{proof}
    The first assertion follows from observing that
    \[
    \mathbb{P}[X\ge M-\varepsilon]\le e^{-\lambda (M-\varepsilon)}\mathbb{E}e^{\lambda X}\le \frac{Ce^{\varepsilon \lambda}}{\lambda^\alpha},\quad \lambda>0
    \]
    and optimizing over $\lambda$ by choosing $\lambda=\frac{\alpha}{\varepsilon}>0$. The second assertion follows from the first assertion by observing that
    \[
    \mathbb{E}\exp\left(\lambda\cdot \frac{X_1+\cdots+X_N}{N}\right)=(\mathbb{E}e^{\lambda X_1/N})^N\le \frac{C^NN^{N\alpha}}{\lambda^{N\alpha}}e^{\lambda M}.
    \]
\end{proof}

\begin{proposition}\label{prop:circle-extreme}
Suppose $X\in [-\pi,\pi]$ is a random variable that is uniform with respect to Lebesgue measure.
Then
    \[
    \mathbb{E}e^{\lambda \cos X}\le 
    \frac 12 \sqrt{\frac{\pi}{2\lambda}}e^\lambda,
    \quad \lambda>0.
    \]
\end{proposition}
\begin{proof}
    We are to estimate the integral
\[
\frac{1}{2\pi}\int_{-\pi}^\pi e^{\lambda \cos\theta}d\theta.
\]
Laplace's method tells us that the correct asymptotics is $\frac{e^\lambda}{\sqrt{2\pi\lambda}} (1+o(1))$ as $\lambda\to\infty$. We content ourselves with proving a weaker bound, using the fact that $\cos\theta\le 1-\frac{2}{\pi^2}\theta^2$ for $\theta\in [-\pi,\pi]$:
\[
    \frac{1}{2\pi}\int_{-\pi}^\pi e^{\lambda \cos\theta}d\theta\le \frac{e^\lambda}{2\pi}\int_{-\pi}^\pi e^{-2\lambda \theta^2/\pi^2}d\theta\le \frac{e^\lambda}{2\pi}\int_{-\infty}^\infty e^{-2\lambda \theta^2/\pi^2}d\theta =\frac 12 \sqrt{\frac{\pi}{2\lambda}}e^\lambda.
\]
\end{proof}

We thus obtain the following information on the distribution of the order parameter $R$.
\begin{lemma}\label{lem:order-param-conc}
For $t\in (0,1)$,
    \begin{equation*}
m\Big\{\{\theta_i^0\}_{i=1}^N\in [-\pi,\pi]^N:~{R_0}\le t\Big\}\le
\min\left\{
\exp\left(-(1-t)^2 N\right),
\left(\frac{\sqrt{\pi e t}}{2}\right)^N
\right\}.
\end{equation*}
\end{lemma}
\begin{proof}
Under the measure $m$, ${R_0}=\frac 1N\sum_{j=1}^N(1+\cos\theta_j^0)$ is the sum of the independent random variables $\frac 1N(1+\cos\theta_j^0)$, and as each $\frac 1N(1+\cos\theta_j^0)$ is $\frac 1{2N^2}$-subgaussian by Proposition \ref{prop:circleConc}, ${R_0}=\sum_{j=1}^N \frac 1N(1+\cos\theta_j^0)$ is $\frac 1{2N}$-subgaussian with mean $1$. Thus
\begin{equation*}
m\Big\{\{\theta_i^0\}_{i=1}^N\in [-\pi,\pi]^N:~{R_0}\le t\Big\}
\le \exp\left(-\frac{(1-t)^2}{1/N}\right)=\exp\left(-(1-t)^2N\right).
\end{equation*}

On the other hand, from Proposition \ref{prop:circle-extreme} we have
\[
\mathbb{E}e^{-\lambda \cos\theta_j^0}=\mathbb{E}e^{\lambda \cos\theta_j^0}\le \frac 12\sqrt{\frac{\pi}{2\lambda}}e^\lambda,\quad \lambda>0,
\]
so that by Lemma \ref{lem:extreme-close} with $M=1$,
\[
\mathbb{P}\left[R_0\le t\right]=\mathbb{P}\left[\frac{-\cos\theta_1^0-\cdots-\cos\theta_N^0}{N}\ge 1-t \right]\le \left(\frac{\sqrt{\pi e t}}{2}\right)^N.
\]
\end{proof}



\begin{proof}[Proof of Theorem \ref{sincosmaincor}]
By Corollary \ref{bestsincosmainthm}, if $R_0\ge 1-\frac\varepsilon 5$, since
\[
K_c\le \frac 2{(1-\varepsilon/5)^{3/2}}< 2+\varepsilon\quad (\because\varepsilon\in (0,1]),
\]
the condition $\kappa\ge (2+\varepsilon)\|\Omega\|_\infty$ implies
\begin{enumerate}[(a)]
    \item $R(t)\ge \frac{3{R_0}-3+\sqrt{{R_0}^2-2{R_0}+9}}{4}\ge \frac{\sqrt{2}}{2}-\frac{3\varepsilon}{20}$ for all $t\ge 0$,
    \item for all $i\in [N]$, the limit $\lim_{t\to\infty}\theta_i(t)$ exists, and
\[
\sup_{t\ge 0} \theta_i(t)-\inf_{t\ge 0}\theta_i(t)<2\pi,\quad \lim_{t\to\infty}\dot{\theta}_i(t)=0.
\]
\end{enumerate}
Thus
\begin{align*}
&m\left\{\{\theta_i^0\}_{i=1}^N\in [-\pi,\pi]^N:~\forall t\ge 0~ R(t)\ge \frac{1}{\sqrt{2}} -\frac{3\varepsilon}{20},\quad  \forall i\in [N]~\sup_{t\ge 0} \theta_i(t)-\inf_{t\ge 0}\theta_i(t)<2\pi\right\}\\
&\ge m\Big\{\{\theta_i^0\}_{i=1}^N\in [-\pi,\pi]^N:~{R_0}\ge 1-\frac \varepsilon {5}\Big\}.
\end{align*}

On the other hand, by Lemma \ref{lem:order-param-conc},
\begin{equation*}
m\left\{\{\theta_i^0\}_{i=1}^N\in [-\pi,\pi]^N:~{R_0}\ge 1-\frac \varepsilon {5}\right\}=1-m\left\{\{\theta_i^0\}_{i=1}^N\in [-\pi,\pi]^N:~{R_0}< 1-\frac \varepsilon {5}\right\}\ge 1-\exp\left(-\frac{\varepsilon^2 N}{25}\right).
\end{equation*}

Similarly, if $\kappa>\left(\frac{\pi e}{2}\right)^{3/2}\|\Omega\|_\infty\ge 2\sqrt{2}\|\Omega\|_\infty$ and $R_0> \left(\frac{2\sqrt{2}\|\Omega\|_\infty}{\kappa}\right)^{2/3}$, then by Corollary \ref{bestsincosmainthm}, for all $i\in [N]$, the limit $\lim_{t\to\infty}\theta_i(t)$ exists, and
\[
\sup_{t\ge 0} \theta_i(t)-\inf_{t\ge 0}\theta_i(t)<2\pi,\quad \lim_{t\to\infty}\dot{\theta}_i(t)=0.
\]

So by Lemma \ref{lem:order-param-conc},
\begin{align*}
    &m\Big\{\{\theta_i^0\}_{i=1}^N\in [-\pi,\pi]^N: \exists \lim_{t\to\infty}\theta_i(t), ~\sup_{t\ge 0} \theta_i(t)-\inf_{t\ge 0}\theta_i(t)<2\pi,~ \lim_{t\to\infty}\dot{\theta}_i(t)=0 ~\mbox{for all } i\in [N]\Big\}\\
&\ge m\left\{\{\theta_i^0\}_{i=1}^N\in [-\pi,\pi]^N:~R_0>\left(\frac{2\sqrt{2}\|\Omega\|_\infty}{\kappa}\right)^{2/3}\right\}\ge 1-\left(\sqrt{\frac{\pi e}{2}}\frac{\|\Omega\|_\infty^{1/3}}{\kappa^{1/3}}\right)^N.
\end{align*}
\end{proof}

\begin{proof}[Proof of Theorem \ref{maincor-kappalarge}]
The distribution $\mu^{\otimes N}$ means that we are independently and identically sampling $\theta_i^0\sim \mu$, so ${R_0}=\frac 1N\sum_{j=1}^N I(\theta_j^0)$ is the sum of the independent random variables $\frac 1N I(\theta_j^0)$, $j=1,\cdots, N$. By \eqref{eq:I-moment} and Lemma \ref{lem:extreme-close} with $M=0$, we have that
\begin{equation}\label{eq:F-1}
\mu^{\otimes N}\left\{\{\theta_i^0\}_{i=1}^N\in [-\pi,\pi]^N:{R_0}=\frac 1N\sum_{j=1}^NI(\theta_j^0) \le \varepsilon\right\}\le \left(C_\mu\left(\frac{e\varepsilon}{\beta}\right)^\beta\right)^N.
\end{equation}
By condition \eqref{eq:kappa-large-cor}, we may find $\varepsilon>0$ such that $C_\mu\left(\frac{e\varepsilon}{\beta}\right)^\beta<1$,
\begin{equation}\label{eq:F-2}
\varepsilon>\max\left\{\left(\frac{
2}{c_1c_3}\frac{\|\Omega\|_\infty}{\kappa}\right)^{\frac{1}{1+p/q}}(2c_2)^{\frac{p/q}{1+p/q}},\frac{2}{c_1c_3(\pi-\alpha_0)^p}\frac{\|\Omega\|_\infty}{\kappa}\right\}
\end{equation}
and
\[
\kappa>\max\left\{\frac{2(2c_2)^{p/q}}{c_1c_3}\cdot\frac{\|\Omega\|_\infty}{{\varepsilon}^{1+\frac pq}},\frac{2}{c_1c_3(\pi-\alpha_0)^p}\cdot \frac{\|\Omega\|_\infty}{{\varepsilon}}\right\}.
\]
By Theorem \ref{mainthm}, we have
\begin{align*}
&\mu^{\otimes N}\Big\{\{\theta_i^0\}_{i=1}^N\in [-\pi,\pi]^N:\sup_{t\ge 0} \theta_i(t)-\inf_{t\ge 0}\theta_i(t)<2\pi ~\forall i\in [N]\Big\}\\
&\ge \mu^{\otimes N}\left\{\{\theta_i^0\}_{i=1}^N\in [-\pi,\pi]^N:{R_0}=\frac 1N\sum_{j=1}^NI(\theta_j^0)> \varepsilon\right\}\\
&\ge 1-\mu^{\otimes N}\left\{\{\theta_i^0\}_{i=1}^N\in [-\pi,\pi]^N:{R_0}=\frac 1N\sum_{j=1}^NI(\theta_j^0)\le \varepsilon\right\}.
\end{align*}
Combining this inequality with \eqref{eq:F-1} and then letting $\varepsilon$ decrease to the right-hand side of \eqref{eq:F-2} gives the stated result. Hypothesis \eqref{eq:kappa-large-cor} ensures that the required choices of $\varepsilon$ exist.
\end{proof}

The proofs of Theorems \ref{themainthm}, \ref{sincosmaincor-time}, \ref{themainthm-IS}, and \ref{maincor-quant} require estimates on the measure of initial data whose order parameter stays small up to time $T$. We derive these estimates in the next subsection.

\subsection{Instability of states with small order parameter}\label{subsec:instability}
In this subsection, we prove Theorems \ref{themainthm}, \ref{sincosmaincor-time}, \ref{themainthm-IS}, and \ref{maincor-quant}. We will control the volume of a singular set using the divergence of the flows \eqref{GenWinfree} and \eqref{Winfree}; the argument of this section is inspired by Section 6 of \cite{ha2020asymptotic}.
\subsubsection{The prototypical Winfree model \eqref{Winfree}}

In this subsubsection, we prove Theorems \ref{themainthm} and \ref{sincosmaincor-time}. The idea is that the set of states with sufficiently small order parameters cannot be positively invariant under the Winfree flow \eqref{Winfree}. By the identification $\theta_i\leftrightarrow e^{\mathrm{i}\theta_i}$, we may consider \eqref{Winfree} to be a flow on the $N$-torus $\bbt^N$. 

For any $\delta>0$ and $N$, we define the set
\[ U^0_{1-\delta} \coloneqq  \left\{\Theta \in \bbt^N:~
R(\Theta)<1-\delta \right\}. \]
We will show that  $ U^0_{1-\delta}$ is almost surely not positively invariant under the Winfree flow \eqref{Winfree} with $\kappa > 0$, i.e., if the flow starts at a generic point in $U^0_{1-\delta}$, it cannot stay inside the region for  all time, regardless of the magnitude of the natural frequencies. Then, as an application of this result, we will prove our main Theorem \ref{themainthm}.

First, we begin with a heuristic argument for the fact that the set $U^0_{1-\delta}$ is not positively invariant. For a given natural frequency vector $\Omega$,  consider the following two cases: \newline
\begin{itemize}
\item Case A ($\kappa \gg \|\Omega\|_\infty$): According to numerical simulations, the flow converges to the equilibrium of Theorem \ref{preciseCOD}, and so the order parameter must eventually exceed $1+\cos \frac{\pi}{3}=\frac 32>1$.

\item Case B (the uncoupled system): Suppose $\kappa=0$ and $\omega_i\ne0$ for every $i\in[N]$. Then
\[
\theta_i(t)=\theta_i^0+\omega_it,
\qquad
\lim_{T\to\infty}\frac1T\int_0^T\cos\theta_i(t)\,dt=0.
\]
Consequently, the temporal average of the order parameter is
\[
\langle R\rangle\coloneqq\lim_{T\to\infty}\frac1T\int_0^T R(t)\,dt
=\frac1N\sum_{i=1}^N\langle1+\cos\theta_i\rangle=1.
\]
Thus, for any $\delta>0$, such an uncoupled trajectory cannot remain in $U^0_{1-\delta}$ for all $t\ge0$.
\end{itemize}

Hence, it is reasonable to expect that given $\kappa >0$, $\Omega=(\omega_1,\cdots,\omega_N)$ and any $\delta>0$, a generic solution to \eqref{Winfree} cannot stay in the region $U^0_{1-\delta}$ for all $t\ge 0$. Now, we will prove this heuristic argument rigorously using the divergence of the flow \eqref{Winfree}. First, note that the Winfree model gives the integral curves to the following vector field:
\begin{equation}\label{vf-1}
    F: \mathbb{T}^N\rightarrow T\mathbb{T}^N,\quad F=(F_1,\cdots,F_N),
\end{equation}
where
\begin{equation}\label{vf-2}
    F_i(\Theta)=\omega_i-\frac{\kappa}{N}\sum_{j=1}^{N} (1+\cos \theta_j)\sin\theta_i.
\end{equation}
Let us denote the flow of this vector field by $\Phi_t$, i.e., $\Phi_t(\Theta^0)=\Theta(t)$ is the solution to \eqref{Winfree} with initial data $\Theta^0$. Also, for any subset $U\subset \bbt^N$ we denote
\[
\Phi_t(U)=\{\Phi_t(\Theta):\Theta\in U\}.
\]
Finally, we endow $\bbt^N$ with the standard Lebesgue measure $m$, normalized so that $m(\bbt^N)=1$. Now the divergence of the vector field $F$ is
\begin{equation}\label{div}
    \nabla \cdot F=\sum_{i=1}^{N} \frac{\partial F_i}{\partial \theta_i}=-\frac{\kappa}{N}\sum_{i, j=1,\cdots, N} (1+\cos\theta_j)\cos\theta_i+\frac{\kappa}{N}\sum_{i=1}^N\sin^2\theta_i=\kappa\left(NR(1-R)+\frac 1N \sum_{i=1}^N\sin^2\theta_i\right)
\end{equation}
Note that the divergence is independent of the choice of natural frequencies $\omega_i$, and depends only on $\kappa$, $N$ and $\Theta$. 

With \eqref{div}, we can rigorously prove the heuristic argument in the beginning of this subsubsection.
\begin{proposition}\label{nodispersed}
Let $\kappa>0$, $N\ge 2$,\footnote{If $N=1$, the dynamics of \eqref{Winfree} is extremely well-understood; see Proposition \ref{verytrivial}.} and $\Omega=(\omega_1,\cdots,\omega_N)$ be fixed, and consider the flow $\Phi_t$ generated by \eqref{vf-1} and \eqref{vf-2}. Then, for any $0<\delta<1$,  the Borel set
\[
U_{1-\delta}^\infty=\{ \Theta^0\in \bbt^N: R(\Phi_t(\Theta^0))<1-\delta\mbox{ for all } t\ge 0\}
\]
has measure zero. Also, for each $T\ge 0$, the open set 
\[
U_{1-\delta}^T\coloneqq \{\Theta^0\in\bbt^N: R(\Phi_t(\Theta^0))<1-\delta \mbox{ for all }0\le t \le T\}
\]
has measure
\begingroup\small
\begin{equation}\label{eq:meas-ineq}
m(U_{1-\delta}^T)\le 
\begin{cases}
\frac 1{N/2+1} \left(\frac{\sqrt{\pi e \delta}}{2}\right)^N\left(1-\exp\left(-\kappa N \delta(1-\delta)T\right)\right)+\exp\left(-N\delta^2-\kappa N \delta(1-\delta)T\right),& \mathrm{if~} 0<\delta<\frac 14\mathrm{~and~}T\le T_0,\\
\left(\frac{4}{\pi e\delta}+\frac{4\kappa(1-\delta)}{\pi e}(T-T_0)\right)^{-\frac{N}{2}},&\mathrm{if~}0<\delta<\frac 14\mathrm{~and~}T>T_0,\\
\frac 1{N/2+1} \left(\frac{\sqrt{\pi e}}{4}\right)^N\left(1-\exp\left(-\frac{3\kappa N T}{16}\right)\right)+\exp\left(-N\delta^2-\frac{3\kappa N T}{16}\right),& \mathrm{if~} \frac 14\le \delta<\frac 12\mathrm{~and~}T\le T_0,\\
\left(\frac{16}{\pi e}+\frac{3\kappa}{\pi e}(T-T_0)\right)^{-\frac{N}{2}},&\mathrm{if~}\frac 14\le \delta<\frac 12\mathrm{~and~}T>T_0,\\
\left(e^{2\delta^2}+\frac{4\kappa\delta T}{\pi e}\right)^{-\frac{N}{2}},&\mathrm{if~}\frac 12 \le \delta<\frac 34,
\\
\left(\frac{4}{\pi e(1-\delta)}+\frac{4\kappa\delta T}{\pi e}\right)^{-\frac{N}{2}},&\mathrm{if~}\frac 34 \le \delta<1,
\end{cases}
\end{equation}
\endgroup
where
\[
T_0 = 
\begin{cases}
\frac{1}{\kappa N\delta(1-\delta)}\log\left(\left(1+\frac{2}{N}\right)\left(\frac{2}{\sqrt{\pi \delta}e^{\frac 12 +\delta^2}}\right)^N-\frac 2N\right)&\mathrm{if~}0<\delta<\frac 14,\\
\frac{16}{3\kappa N}\log\left(\left(1+\frac{2}{N}\right)\left(\frac{4}{\sqrt{\pi}e^{\frac 12 +\delta^2}}\right)^N-\frac 2N\right)&\mathrm{if~}\frac 14\le \delta<\frac 12.
\end{cases}
\]
\end{proposition}
\begin{proof}
Note that 
\[
U_{\varepsilon}^0 =\{\Theta^0\in\bbt^N: R(\Theta^0)<\varepsilon\}, \quad \varepsilon\in (0,1),
\]
which by Lemma \ref{lem:order-param-conc} has measure
\[
m(U_\varepsilon^0)\le \left(\frac{\sqrt{\pi e \varepsilon}}{2}\right)^N, \quad \varepsilon\in (0,1).
\]
By definition, we have
\[
\Phi_t(U_{1-\delta}^T)\subset U_{1-\delta}^0,\quad \forall~0\le t\le T.
\]
By the fact that the divergence of a flow determines the rate of change of volume, we have for $0\le t\le T$
\begin{equation}\label{eq:div-ineq}
\frac{d}{dt}m\left(\Phi_t\left(U_{1-\delta}^T\right)\right)
=\int_{\Phi_t\left(U_{1-\delta}^T\right)}(\nabla\cdot F)(\Theta)dm(\Theta)
\stackrel{\eqref{div}}{\ge}\kappa N \int_{\Phi_t\left(U_{1-\delta}^T\right)}R(1-R) dm(\Theta).
\end{equation}

We now prove \eqref{eq:meas-ineq} case by case.

\begin{enumerate}
\item 
Suppose $0<\delta<\frac 14$. We have
\begin{align*}
    \frac{d}{dt}m\left(\Phi_t\left(U_{1-\delta}^T\right)\right)
&\stackrel{\mathclap{\eqref{eq:div-ineq}}}{\ge}\kappa N \int_{\Phi_t\left(U_{1-\delta}^T\right)\setminus U^0_\delta}R(1-R) dm(\Theta)+\kappa N \int_{\Phi_t\left(U_{1-\delta}^T\right)\cap U^0_\delta}R(1-R) dm(\Theta)\\
&\ge \kappa N \delta(1-\delta) m\left(\Phi_t\left(U_{1-\delta}^T\right)\setminus U^0_\delta\right)+\kappa N(1-\delta)\int_{\Phi_t\left(U_{1-\delta}^T\right)\cap U^0_\delta}R dm(\Theta),
\end{align*}
which we estimate term by term:
\begin{align*}
\kappa N \delta(1-\delta) m\left(\Phi_t\left(U_{1-\delta}^T\right)\setminus U^0_\delta\right)&\ge \kappa N \delta(1-\delta)m\left(\Phi_t\left(U_{1-\delta}^T\right)\right)-\kappa N \delta(1-\delta)m\left( \Phi_t\left(U_{1-\delta}^T\right)\cap U^0_\delta\right)
\end{align*}
and
\begin{align*}
    \kappa N(1-\delta)\int_{\Phi_t\left(U_{1-\delta}^T\right)\cap U^0_\delta}R dm(\Theta)&=\kappa N(1-\delta)\int_{\Phi_t\left(U_{1-\delta}^T\right)\cap U^0_\delta}\int_0^\delta \mathbbm{1}_{R\ge s} dsdm(\Theta)\\
    &=\kappa N(1-\delta)\int_0^\delta m\left(\Phi_t\left(U_{1-\delta}^T\right)\cap U^0_\delta\setminus  U^0_s\right) ds\\
    &\ge \kappa N(1-\delta)\int_0^\delta \max\left\{0,m\left(\Phi_t\left(U_{1-\delta}^T\right)\cap U^0_\delta\right)-m\left( U^0_s\right)\right\} ds\\
    &\ge \kappa N(1-\delta)\int_0^\delta \max\left\{0,m\left(\Phi_t\left(U_{1-\delta}^T\right)\cap U^0_\delta\right)-\left(\frac{\sqrt{\pi e s}}{2}\right)^N\right\} ds\\
    &=\kappa N(1-\delta)\frac{4}{\pi e}\frac{1}{1+2/N}m\left(\Phi_t\left(U_{1-\delta}^T\right)\cap U^0_\delta\right)^{1+\frac 2N},
\end{align*}
where in the last equality we used $m\left(\Phi_t\left(U_{1-\delta}^T\right)\cap U^0_\delta\right)\le m\left( U^0_\delta\right)\le \left(\frac{\sqrt{\pi e \delta}}{2}\right)^N$.
Harvesting estimates, we obtain
\begin{align}\label{eq:diffeq-1}
\begin{aligned}
    \frac{d}{dt}m\left(\Phi_t\left(U_{1-\delta}^T\right)\right)&\ge \kappa N \delta(1-\delta)m\left(\Phi_t\left(U_{1-\delta}^T\right)\right)-\kappa N \delta(1-\delta)m\left( \Phi_t\left(U_{1-\delta}^T\right)\cap U^0_\delta\right)\\
    &\quad +\kappa N(1-\delta)\frac{4}{\pi e}\frac{1}{1+2/N}m\left(\Phi_t\left(U_{1-\delta}^T\right)\cap U^0_\delta\right)^{1+\frac 2N}\\
    &\ge \kappa N \delta(1-\delta)m\left(\Phi_t\left(U_{1-\delta}^T\right)\right)+\kappa N (1-\delta)\inf_{0\le y\le m\left(\Phi_t\left(U_{1-\delta}^T\right)\right)}\left(-\delta y+\frac {4}{\pi e}\frac 1{1+2/N}y^{1+\frac 2N}\right)\\
    &=
    \begin{cases}
        \kappa N(1-\delta)\frac{4}{\pi e}\frac 1{1+2/N}m\left(\Phi_t\left(U_{1-\delta}^T\right)\right)^{1+2/N}& \mathrm{if~}m\left(\Phi_t\left(U_{1-\delta}^T\right)\right)\le \left(\frac{\sqrt{\pi e \delta}}{2}\right)^{N}\\
        \kappa N \delta(1-\delta)\left(m\left(\Phi_t\left(U_{1-\delta}^T\right)\right)-\frac{1}{N/2+1}\left(\frac{\sqrt{\pi e \delta}}{2}\right)^{N}\right)& \mathrm{if~}m\left(\Phi_t\left(U_{1-\delta}^T\right)\right)> \left(\frac{\sqrt{\pi e \delta}}{2}\right)^{N}
    \end{cases}
    \\
    &\ge 
    \begin{cases}
        \frac{2\kappa N(1-\delta)}{\pi e}m\left(\Phi_t\left(U_{1-\delta}^T\right)\right)^{1+2/N}& \mathrm{if~}m\left(\Phi_t\left(U_{1-\delta}^T\right)\right)\le \left(\frac{\sqrt{\pi e \delta}}{2}\right)^{N}\\
        \kappa N \delta(1-\delta)\left(m\left(\Phi_t\left(U_{1-\delta}^T\right)\right)-\frac{1}{N/2+1}\left(\frac{\sqrt{\pi e \delta}}{2}\right)^{N}\right)& \mathrm{if~}m\left(\Phi_t\left(U_{1-\delta}^T\right)\right)> \left(\frac{\sqrt{\pi e \delta}}{2}\right)^{N}
    \end{cases}
\end{aligned}
\end{align}
where in the second inequality we put $y=m\left( \Phi_t\left(U_{1-\delta}^T\right)\cap U^0_\delta\right)$, and in the first equality we optimized to find $y=m\left(\Phi_t\left(U_{1-\delta}^T\right)\right)$ if $m\left(\Phi_t\left(U_{1-\delta}^T\right)\right)\le \left(\frac{\sqrt{\pi e \delta}}{2}\right)^{N}$ and $y=\left(\frac{\sqrt{\pi e \delta}}{2}\right)^N$ if $m\left(\Phi_t\left(U_{1-\delta}^T\right)\right)> \left(\frac{\sqrt{\pi e \delta}}{2}\right)^{N}$, and in the third inequality we used $N\ge 2$.

Solving this differential inequality \eqref{eq:diffeq-1} along with the terminal condition
\[
m\left(\Phi_T\left(U_{1-\delta}^T\right)\right)\le m(U^0_{1-\delta})\stackrel{\mathrm{Lemma~}\ref{lem:order-param-conc}}{\le} \exp(-N\delta^2)
\]
gives the stated inequality \eqref{eq:meas-ineq} in the case $0<\delta<\frac 14$.

\item 
On the other hand, suppose $\frac 14\le \delta<\frac 12$. We have
\begin{align*}
    \frac{d}{dt}m\left(\Phi_t\left(U_{1-\delta}^T\right)\right)
&\stackrel{\mathclap{\eqref{eq:div-ineq}}}{\ge}\kappa N \int_{\Phi_t\left(U_{1-\delta}^T\right)\setminus U^0_{1/4}}R(1-R) dm(\Theta)+\kappa N \int_{\Phi_t\left(U_{1-\delta}^T\right)\cap U^0_{1/4}}R(1-R) dm(\Theta)\\
&\ge \frac{3}{16}\kappa N  m\left(\Phi_t\left(U_{1-\delta}^T\right)\setminus U^0_{1/4}\right)+\frac 34\kappa N\int_{\Phi_t\left(U_{1-\delta}^T\right)\cap U^0_{1/4}}R dm(\Theta),
\end{align*}
which we again estimate term by term:
\begin{align*}
\frac{3}{16}\kappa N  m\left(\Phi_t\left(U_{1-\delta}^T\right)\setminus U^0_{1/4}\right)&\ge \frac{3}{16}\kappa N m\left(\Phi_t\left(U_{1-\delta}^T\right)\right)-\frac{3}{16}\kappa N m\left( \Phi_t\left(U_{1-\delta}^T\right)\cap U^0_{1/4}\right)
\end{align*}
and
\begin{align*}
    \frac 34\kappa N\int_{\Phi_t\left(U_{1-\delta}^T\right)\cap U^0_{1/4}}R dm(\Theta)&=\frac 34\kappa N\int_{\Phi_t\left(U_{1-\delta}^T\right)\cap U^0_{1/4}}\int_0^{1/4} \mathbbm{1}_{R\ge s} dsdm(\Theta)\\
    &=\frac 34\kappa N\int_0^{1/4} m\left(\Phi_t\left(U_{1-\delta}^T\right)\cap U^0_{1/4}\setminus  U^0_s\right) ds\\
    &\ge \frac 34\kappa N\int_0^{1/4} \max\left\{0,m\left(\Phi_t\left(U_{1-\delta}^T\right)\cap U^0_{1/4}\right)-m\left( U^0_s\right)\right\} ds\\
    &\ge \frac 34\kappa N\int_0^{1/4} \max\left\{0,m\left(\Phi_t\left(U_{1-\delta}^T\right)\cap U^0_{1/4}\right)-\left(\frac{\sqrt{\pi e s}}{2}\right)^N\right\} ds\\
    &=\frac{3\kappa N}{\pi e}\frac{1}{1+2/N}m\left(\Phi_t\left(U_{1-\delta}^T\right)\cap U^0_{1/4}\right)^{1+\frac 2N},
\end{align*}
where in the last equality we used $m\left(\Phi_t\left(U_{1-\delta}^T\right)\cap U^0_{1/4}\right)\le m\left( U^0_{1/4}\right)\le \left(\frac{\sqrt{\pi e}}{4}\right)^N$.

Harvesting estimates, we obtain
\begin{align}\label{eq:diffeq-2}
\begin{aligned}
    \frac{d}{dt}m\left(\Phi_t\left(U_{1-\delta}^T\right)\right)&\ge \frac{3}{16}\kappa N m\left(\Phi_t\left(U_{1-\delta}^T\right)\right)-\frac{3}{16}\kappa N m\left( \Phi_t\left(U_{1-\delta}^T\right)\cap U^0_{1/4}\right)\\
    &\quad +\frac {3\kappa N}{\pi e}\frac{1}{1+2/N}m\left(\Phi_t\left(U_{1-\delta}^T\right)\cap U^0_{1/4}\right)^{1+\frac 2N}\\
    &\ge \frac{3}{16}\kappa N m\left(\Phi_t\left(U_{1-\delta}^T\right)\right)+\frac{3\kappa N}{16} \inf_{0\le y\le m\left(\Phi_t\left(U_{1-\delta}^T\right)\right)}\left(- y+\frac {16}{\pi e}\frac 1{1+2/N}y^{1+\frac 2N}\right)\\
    &=
    \begin{cases}
        \frac{3\kappa N}{\pi e}\frac 1{1+2/N}m\left(\Phi_t\left(U_{1-\delta}^T\right)\right)^{1+2/N}& \mathrm{if~}m\left(\Phi_t\left(U_{1-\delta}^T\right)\right)\le \left(\frac{\sqrt{\pi e }}{4}\right)^{N}\\
        \frac{3\kappa N}{16} \left(m\left(\Phi_t\left(U_{1-\delta}^T\right)\right)-\frac{1}{N/2+1}\left(\frac{\sqrt{\pi e }}{4}\right)^{N}\right)& \mathrm{if~}m\left(\Phi_t\left(U_{1-\delta}^T\right)\right)> \left(\frac{\sqrt{\pi e }}{4}\right)^{N}
    \end{cases}
    \\
    &\ge 
    \begin{cases}
        \frac{3\kappa N}{2\pi e}m\left(\Phi_t\left(U_{1-\delta}^T\right)\right)^{1+2/N}& \mathrm{if~}m\left(\Phi_t\left(U_{1-\delta}^T\right)\right)\le \left(\frac{\sqrt{\pi e }}{4}\right)^{N}\\
        \frac{3\kappa N}{16} \left(m\left(\Phi_t\left(U_{1-\delta}^T\right)\right)-\frac{1}{N/2+1}\left(\frac{\sqrt{\pi e}}{4}\right)^{N}\right)& \mathrm{if~}m\left(\Phi_t\left(U_{1-\delta}^T\right)\right)> \left(\frac{\sqrt{\pi e }}{4}\right)^{N}
    \end{cases}
\end{aligned}
\end{align}
where in the second inequality we put $y=m\left( \Phi_t\left(U_{1-\delta}^T\right)\cap U^0_{1/4}\right)$, and in the first equality we optimized to find $y=m\left(\Phi_t\left(U_{1-\delta}^T\right)\right)$ if $m\left(\Phi_t\left(U_{1-\delta}^T\right)\right)\le \left(\frac{\sqrt{\pi e }}{4}\right)^{N}$ and $y=\left(\frac{\sqrt{\pi e }}{4}\right)^N$ if $m\left(\Phi_t\left(U_{1-\delta}^T\right)\right)> \left(\frac{\sqrt{\pi e }}{4}\right)^{N}$, and in the third inequality we used $N\ge 2$.

Solving this differential inequality \eqref{eq:diffeq-2} along with the terminal condition
\[
m\left(\Phi_T\left(U_{1-\delta}^T\right)\right)\le m(U^0_{1-\delta})\stackrel{\mathrm{Lemma~}\ref{lem:order-param-conc}}{\le} \exp(-N\delta^2)
\]
gives the stated inequality \eqref{eq:meas-ineq} in the case $\frac 14\le \delta<\frac 12$.

\item 
Now, suppose $\frac 12\le \delta<1$. Evidently $\Phi_t\left(U_{1-\delta}^T\right)\subset U_{1-\delta}^0$, so 
\begin{align}
\begin{aligned}\label{eq:diffeq-3}
    \frac{d}{dt}m\left(\Phi_t\left(U_{1-\delta}^T\right)\right)
&\stackrel{\mathclap{\eqref{eq:div-ineq}}}{\ge}\kappa N \int_{\Phi_t\left(U_{1-\delta}^T\right)}R(1-R) dm(\Theta)\ge \kappa N\delta  \int_{\Phi_t\left(U_{1-\delta}^T\right)}R dm(\Theta)\\
&\ge \kappa N \delta \int_{\Phi_t\left(U_{1-\delta}^T\right)} \int_0^{1-\delta}\mathbbm{1}_{R\ge s}ds dm(\Theta)\\
&=\kappa N \delta \int_0^{1-\delta}m\left(\Phi_t\left(U_{1-\delta}^T\right)\setminus U_s^0\right)ds\\
&\ge \kappa N \delta \int_0^{1-\delta}\max\left\{0,m\left(\Phi_t\left(U_{1-\delta}^T\right)\right)-m\left( U_s^0\right)\right\}ds\\
&\ge \kappa N \delta \int_0^{1-\delta}\max\left\{0,m\left(\Phi_t\left(U_{1-\delta}^T\right)\right)-\left(\frac{\sqrt{\pi e s}}{2}\right)^N\right\}ds\\
    &=\frac{4\kappa N\delta }{\pi e}\frac{1}{1+2/N}m\left(\Phi_t\left(U_{1-\delta}^T\right)\right)^{1+\frac 2N}\\
    &\ge \frac{2\kappa N\delta }{\pi e}m\left(\Phi_t\left(U_{1-\delta}^T\right)\right)^{1+\frac 2N},
\end{aligned}
\end{align}
where in the last equality we used $m\left(\Phi_t\left(U_{1-\delta}^T\right)\right)\le m\left(U_{1-\delta}^0\right)\le \left(\frac{\sqrt{\pi e (1-\delta)}}{2}\right)^N$ and in the last inequality we used $N\ge 2$.

In the case $\frac 12\le \delta<\frac 34$, solving the differential inequality \eqref{eq:diffeq-3} along with the terminal condition $m\left(\Phi_T\left(U_{1-\delta}^T\right)\right)\le m(U^0_{1-\delta})\stackrel{\mathrm{Lemma~}\ref{lem:order-param-conc}}{\le}   \exp(-N\delta^2)$ gives the stated inequality \eqref{eq:meas-ineq}. In the case $\frac 34\le \delta<1$, solving the differential inequality \eqref{eq:diffeq-3} along with the terminal condition $m\left(\Phi_T\left(U_{1-\delta}^T\right)\right)\le m(U^0_{1-\delta})\stackrel{\mathrm{Lemma~}\ref{lem:order-param-conc}}{\le}  \left(\frac{\sqrt{\pi e (1-\delta)}}{2}\right)^N$ gives the stated inequality \eqref{eq:meas-ineq} .

\end{enumerate}

It is now evident that
\[
U_{1-\delta}^\infty =\mathop{\bigcap}_{T\in\bbn}U_{1-\delta}^T=\mathop{\bigcap}_{T>0}U_{1-\delta}^T
\]
is a Borel set with measure $m\left(U_{1-\delta}^\infty\right)\le \liminf_{T\to\infty}m\left(U_{1-\delta}^T\right)=0$.
\end{proof}

One more ingredient is needed before we prove Theorem \ref{themainthm}, which is a fact about the linear stability of equilibria of \eqref{Winfree}.
\begin{proposition}\label{prop:linear-stability}
Let $\kappa>0$ and $\Omega\in \mathbb{R}^N$. Then any equilibrium $\Theta^0\in \mathbb{R}^N$ of \eqref{Winfree} with $\cos \theta_i^0<0$ for some $i\in [N]$ is linearly unstable, excluding the exceptional case $\Omega=\vec{0}$ and $\theta_i^0=  \pi\mod 2\pi$ for all $i\in [N]$.
\end{proposition}
\begin{proof}
We compute the Jacobian matrix of the vector field $F=(F_1,\cdots,F_N)$ given in \eqref{vf-1}-\eqref{vf-2} to be
\[
\frac{\partial F_i}{\partial \theta_i}=-\kappa R(\Theta)\cos \theta_i + \frac{\kappa}{N}\sin^2 \theta_i,~i\in [N],\quad \frac{\partial F_i}{\partial \theta_j}=\frac{\kappa}{N}\sin\theta_i\sin\theta_j,~i,j\in [N],~i\neq j.
\]
This is real symmetric, has a positive diagonal entry at an $i\in [N]$ such that $\cos \theta_i^0<0$ (since excluding the exceptional case gives $R>0$), and hence a positive eigenvalue.
\end{proof}
\begin{remark}\label{rem:linear-stability}
As observed right after Theorem \ref{standardpreciseCOD}, if $\kappa>\|\Omega\|_\infty$ then the only equilibrium of \eqref{Winfree} that has $\cos \theta_i^0\ge 0$ for all $i\in [N]$ is that of Theorem \ref{standardpreciseCOD}. By Proposition \ref{prop:linear-stability}, if $\Omega\neq 0$ and $\kappa>\|\Omega\|_\infty$, every other equilibrium is linearly unstable. When $\Omega=0$, there is one additional exception: the all-$\pi$ equilibrium has $R=0$ and zero Jacobian.
\end{remark}

We are now ready to prove Theorem \ref{themainthm}.
\begin{proof}[Proof of Theorem \ref{themainthm}] 
When $N=1$, it is straightforward to analyze the equation of \eqref{Winfree} to see that there is one repulsive point and one attractive point which attracts every point other than the repulsive point. Assume henceforth $N\ge2$. Likewise, if $\Omega=0$, it is not hard to see that any initial data with $\theta_i^0\not\equiv \pi \mod 2\pi$ converges to $0 \mod 2\pi$. Henceforth assume $\Omega\neq 0$.
	
Choose a real number $\mu\in \left(1-\sqrt{\frac 12+\sqrt{\frac 14-\frac{\|\Omega\|_\infty^2}{\kappa^2}}},1-\frac{1}{\sqrt{2}}\right)$. As
\[
\kappa>\frac{\|\Omega\|_\infty}{(1-\mu)\sqrt{\mu(2-\mu)}},
\]
we may choose a real number $\delta\in (0,1)$ sufficiently close to $0$ such that
\[ \kappa >\frac{\|\Omega\|_\infty}{(1-\delta-\mu)\sqrt{\mu(2-\mu)}}.
\]
Since $\kappa>0$, Proposition \ref{nodispersed} tells us that the set
\[
U_{1-\delta}^\infty=\{ \Theta^0\in \bbt^N: R(\Phi_t(\Theta^0))<1-\delta \mbox{ for all } t\ge 0\}
\]
has measure zero. So, for almost all initial data $\Theta^0\in \bbt^N\backslash U_{1-\delta}^\infty$, there exists a finite time $t_0\ge 0$ such that
\[
R(\Phi_{t_0}(\Theta^0))\ge 1-\delta,
\]
and we apply Theorem \ref{generalbestsincosmainthm} at this time $t=t_0$ to conclude that there exists $\lim_{t\to\infty}\theta_i(t)$ and $\lim_{t\to\infty}\dot{\theta}_i(t)=0$. 

We have proven that for Lebesgue a.e.~initial data, the flow converges to an equilibrium point. By \cite{kelley1966stable}, linearly unstable equilibria have a Lebesgue-zero basin of attraction, and by Theorem \ref{GlobalFiniteness} (whose proof is independent of Theorem \ref{themainthm}), there are finitely many equilibria. By Proposition \ref{prop:linear-stability} and Remark \ref{rem:linear-stability}, the equilibrium guaranteed by Theorem \ref{standardpreciseCOD} is the only not linearly unstable equilibrium, so we may conclude its basin of attraction has full Lebesgue measure. The exponential convergence is established in Theorem \ref{standardpreciseCOD}; the specific rate is given with $\alpha^\infty=\pi/4$.
\end{proof}

We are also ready to prove Theorem \ref{sincosmaincor-time}.

\begin{proof}[Proof of Theorem \ref{sincosmaincor-time}]
Similarly to the proof of Theorem \ref{sincosmaincor}, we have from Corollary \ref{bestsincosmainthm} that
\begin{align*}
&m\left\{\{\theta_i^0\}_{i=1}^N\in [-\pi,\pi]^N:~\forall t\ge T~ R(t)\ge \frac{1}{\sqrt{2}} -\frac{3\varepsilon}{20},\quad \forall i\in [N]~\sup_{t\ge T} \theta_i(t)-\inf_{t\ge T}\theta_i(t)<2\pi\right\}\\
&\ge m\Big\{\{\theta_i^0\}_{i=1}^N\in [-\pi,\pi]^N:~R(t)\ge 1-\frac \varepsilon {5}\mathrm{~for~some~}t\in [0,T]\Big\}\\
&= 1-m\left(U_{1-\varepsilon/5}^T\right)
\end{align*}
at which point we may invoke Proposition \ref{nodispersed} with $\delta=\frac \varepsilon 5<\frac 14$ to see that
\begin{equation*}
m(U_{1-\varepsilon/5}^T)\le 
\begin{cases}
\frac 1{N/2+1} \left(\frac{\sqrt{\pi e \varepsilon}}{2\sqrt{5}}\right)^N\left(1-\exp\left(-\frac{\kappa N \varepsilon(5-\varepsilon)T}{25}\right)\right)+\exp\left(-\frac{N\varepsilon^2}{25}-\frac{\kappa N \varepsilon(5-\varepsilon)T}{25}\right),& \mathrm{if~}T\le T_0,\\
\left(\frac{20}{\pi e\varepsilon}+\frac{4\kappa(5-\varepsilon)}{5\pi e}(T-T_0)\right)^{-\frac{N}{2}},&\mathrm{if~}T>T_0,
\end{cases}
\end{equation*}
where $T_0 = \frac{25}{\kappa N\varepsilon(5-\varepsilon)}\log\left(\left(1+\frac 2N\right)\left(\frac{2\sqrt{5}}{\sqrt{\pi \varepsilon}e^{1/2+\varepsilon^2/25}}\right)^N-\frac 2N\right)$.

For the second assertion, suppose $\kappa>8\|\Omega\|_\infty$. If $\Omega=0$, Proposition \ref{verytrivial} gives a phase range below $2\pi$ for every initial datum, so the conclusion is immediate. Assume $\|\Omega\|_\infty>0$. If $2\le N\le5$, then
	\[
	\kappa>8\|\Omega\|_\infty>
	\frac{4N\|\Omega\|_\infty}{3\sqrt3},
	\]
	so Proposition \ref{verytrivial} gives a phase range below $2\pi$ for every initial datum, and Proposition \ref{Loja} gives convergence. The second assertion therefore holds in these dimensions. Assume henceforth $N\ge6$.
Whenever the trajectory reaches $R\ge (2\sqrt2\|\Omega\|_\infty/\kappa)^{2/3}$ by time $T$, Corollary \ref{bestsincosmainthm}(a)--(b) applies from that time onward and gives the required trapping and convergence. Thus
\begin{align*}
	&m\Big\{\{\theta_i^0\}_{i=1}^N\in [-\pi,\pi]^N: \exists \lim_{t\to\infty}\theta_i(t), ~\sup_{t\ge T} \theta_i(t)-\inf_{t\ge T}\theta_i(t)<2\pi,~ \lim_{t\to\infty}\dot{\theta}_i(t)=0 ~\mbox{for all } i\in [N]\Big\}\\
	&\ge m\left\{\{\theta_i^0\}_{i=1}^N\in [-\pi,\pi]^N:~R(t)\ge \left(\frac{2\sqrt{2}\|\Omega\|_\infty}{\kappa}\right)^{2/3}\mathrm{~for~some~}t\in [0,T]\right\}=1-m\left(U^T_{\left(\frac{2\sqrt{2}\|\Omega\|_\infty}{\kappa}\right)^{2/3}}\right)
\end{align*}
To retain the stated coefficient, put
	$\delta=1-(2\sqrt2\|\Omega\|_\infty/\kappa)^{2/3}>1/2$.
	In the proof of Proposition \ref{nodispersed}, retain the sharper term in \eqref{eq:diffeq-3}, before its final simplification. Integrating the resulting volume inequality gives the coefficient
	\[
	\frac{8\kappa\delta}{\pi e(1+2/N)}\ge\frac{3\kappa}{\pi e}
	\qquad(N\ge6).
	\]
	With the same terminal estimates used there, we obtain
\[
m\left(U^T_{\left(\frac{2\sqrt{2}\|\Omega\|_\infty}{\kappa}\right)^{2/3}}\right)\le 
\begin{cases}
\left(e^{1/2}+\frac{3\kappa T}{\pi e}\right)^{-\frac{N}{2}}&\mathrm{if~}8\|\Omega\|_\infty<\kappa\le 16\sqrt{2}\|\Omega\|_\infty\\
\left(\frac{4}{\pi e}\left(\frac{\kappa}{2\sqrt{2}\|\Omega\|_\infty}\right)^{2/3}+\frac{3\kappa T}{\pi e}\right)^{-\frac{N}{2}}&\mathrm{if~}\kappa>16\sqrt{2}\|\Omega\|_\infty
\end{cases}
\]
since $\left(\frac{2\sqrt{2}\|\Omega\|_\infty}{\kappa}\right)^{2/3}<\frac 12$ if $\kappa>8\|\Omega\|_\infty$ and $\left(\frac{2\sqrt{2}\|\Omega\|_\infty}{\kappa}\right)^{2/3}<\frac 14$ if $\kappa>16\sqrt{2}\|\Omega\|_\infty$.
\end{proof}

\subsubsection{General interaction functions $I$ and $S$}
In this subsection we prove Theorems \ref{themainthm-IS} and \ref{maincor-quant}. This time, note that the Winfree model \eqref{GenWinfree} gives the integral curves to the following vector field:
\begin{equation}\label{vf-1-IS}
    F^{I,S}: \mathbb{T}^N\rightarrow T\mathbb{T}^N,\quad F^{I,S}=(F^{I,S}_1,\cdots,F^{I,S}_N),
\end{equation}
where
\begin{equation}\label{vf-2-IS}
    F^{I,S}_i(\Theta)=\omega_i+\frac{\kappa}{N}\sum_{j=1}^{N} I(\theta_j)S(\theta_i).
\end{equation}
Denote the flow of this vector field by $\Phi^{I,S}_t$. Since $I,S$ are Lipschitz, $F^{I,S}$ is Lipschitz and its flow maps are bi-Lipschitz for every finite time. Derivatives below are interpreted almost everywhere. The almost-everywhere Jacobian satisfies Liouville's identity
\[
\det D\Phi_t^{I,S}(x)=
\exp\left(\int_0^t(\nabla\cdot F^{I,S})(\Phi_s^{I,S}(x))\,ds\right)
\quad\text{for a.e. }x.
\]
Together with the area formula, this implies that, for every measurable set $E$, the function $t\mapsto m(\Phi_t^{I,S}(E))$ is locally absolutely continuous and its derivative equals
$\int_{\Phi_t^{I,S}(E)}\nabla\cdot F^{I,S}\,dm$ for a.e. $t$. Thus the volume differential inequalities used below remain valid at Lipschitz regularity.
The divergence satisfies, under the assumptions \eqref{c_4}, \eqref{c_5}, and $I\ge 0$,
\begin{equation}\label{div-IS}
    \nabla \cdot F^{I,S}=\sum_{i=1}^{N} \frac{\partial F^{I,S}_i}{\partial \theta_i}=\frac{\kappa}{N}\sum_{i, j=1,\cdots, N} I(\theta_j)S'(\theta_i)+\frac{\kappa}{N}\sum_{i=1}^NI'(\theta_i)S(\theta_i)\ge c_4\kappa N R(I_*-R).
\end{equation}
Again, the right-hand side is independent of the choice of natural frequencies $\omega_i$, and depends only on $\kappa$, $N$ and $R$.

We have the following generalization of Proposition \ref{nodispersed}, whose proof is the same.
\begin{proposition}\label{nodispersed-IS-qual}
Consider the Winfree model \eqref{GenWinfree} with interaction functions $I$ and $S$ satisfying \eqref{c_4}, \eqref{c_5}, and \eqref{c_6}. Let $\kappa>0$, $N$, and $\Omega=(\omega_1,\cdots,\omega_N)$ be fixed, and consider the flow $\Phi_t^{I,S}$ generated by \eqref{vf-1-IS}-\eqref{vf-2-IS}. Then, for any $0<\delta<I_*$,  the Borel set
\[
U_{I_*-\delta}^{\infty,I,S}=\{ \Theta^0\in \bbt^N: R(\Phi_t^{I,S}(\Theta^0))<I_*-\delta\mbox{ for all } t\ge 0\}
\]
has measure zero.
\end{proposition}
\begin{proof}
    Defining for each $T\in[0,\infty)$ the open set 
\[
U_{I_*-\delta}^{T,I,S}\coloneqq \{\Theta^0\in\bbt^N: R(\Phi_t^{I,S}(\Theta^0))<I_*-\delta \mbox{ for all }0\le t \le T\},
\]
we have that
\[
U_{I_*-\delta}^{\infty,I,S}=\bigcap_{T\in \mathbb{N}}U_{I_*-\delta}^{T,I,S}
\]
is a Borel set.

Note that 
\[
U_{\varepsilon}^{0,I,S} =\{\Theta^0\in\bbt^N: R(\Theta^0)<\varepsilon\}, \quad \varepsilon>0,
\]
which by condition \eqref{c_6} satisfies
\[
\lim_{\varepsilon\to 0}m(U_\varepsilon^{0,I,S})=0.
\]
By definition, we have
\[
\Phi_t^{I,S}(U_{I_*-\delta}^{T,I,S})\subset U_{I_*-\delta}^{0,I,S},\quad \forall~0\le t\le T\le \infty,~t<\infty.
\]
Also, since for $0<\delta_1<\delta_2<I_*$ and $0\le T\le \infty$ we have $U^{T,I,S}_{I_*-\delta_2}\subset U^{T,I,S}_{I_*-\delta_1}$, we may possibly shrink $\delta$ to assume that $\delta<I_*/2$.

By the fact that the divergence of a flow determines the rate of change of volume, we have for $0\le t\le T<\infty$
\begin{equation*}
\frac{d}{dt}m\left(\Phi_t^{I,S}\left(U_{I_*-\delta}^{T,I,S}\right)\right)
=\int_{\Phi_t^{I,S}\left(U_{I_*-\delta}^{T,I,S}\right)}(\nabla\cdot F^{I,S})(\Theta)dm(\Theta)
\stackrel{\eqref{div-IS}}{\ge}c_4\kappa N \int_{\Phi_t^{I,S}\left(U_{I_*-\delta}^{T,I,S}\right)}R(I_*-R) dm(\Theta).
\end{equation*}

Let $0<\varepsilon<\delta$. We have
\begin{align*}
    \frac{d}{dt}m\left(\Phi_t^{I,S}\left(U_{I_*-\delta}^{T,I,S}\right)\right)
&\ge c_4\kappa N \int_{\Phi_t^{I,S}\left(U_{I_*-\delta}^{T,I,S}\right)\setminus U^{0,I,S}_\varepsilon}R(I_*-R) dm(\Theta)\ge c_4\kappa N \varepsilon(I_*-\varepsilon) m\left(\Phi_t^{I,S}\left(U_{I_*-\delta}^{T,I,S}\right)\setminus U^{0,I,S}_\varepsilon\right)\\
&\ge c_4\kappa N \varepsilon(I_*-\varepsilon) \max\left\{0,m\left(\Phi_t^{I,S}\left(U_{I_*-\delta}^{T,I,S}\right)\right)-m\left(U^{0,I,S}_\varepsilon\right)\right\}
\end{align*}
which we can solve as
\[
\max\left\{0,m\left(\Phi_T^{I,S}\left(U_{I_*-\delta}^{T,I,S}\right)\right)-m\left(U^{0,I,S}_\varepsilon\right)\right\}\ge \exp\left(c_4\kappa N\varepsilon(I_*-\varepsilon)T\right)\max\left\{0,m\left(U_{I_*-\delta}^{T,I,S}\right)-m\left(U^{0,I,S}_\varepsilon\right)\right\}.
\]
However, because of the bound $m\left(\Phi_T^{I,S}\left(U_{I_*-\delta}^{T,I,S}\right)\right)\le 1$, we have
\[
m\left(U_{I_*-\delta}^{\infty,I,S}\right)\le m\left(U_{I_*-\delta}^{T,I,S}\right)\le m\left(U^{0,I,S}_\varepsilon\right)+\exp\left(-c_4\kappa N\varepsilon(I_*-\varepsilon)T\right),\quad \forall \varepsilon\in (0,\delta),~\forall T\in (0,\infty).
\]
First taking the limit $T\to\infty$ and then taking the limit $\varepsilon\to 0$ gives $m\left(U_{I_*-\delta}^{\infty,I,S}\right)=0$.
\end{proof}

The proof of Theorem \ref{themainthm-IS} from Theorem \ref{mainthm} and Proposition \ref{nodispersed-IS-qual} is similar to the proof of Theorem \ref{themainthm} from Theorem \ref{generalbestsincosmainthm} and Proposition \ref{nodispersed}.

\begin{proof}[Proof of Theorem \ref{themainthm-IS}] 
Choose a sequence $\{\delta_n\}_{n=1}^\infty$ in $(0,I_*)$ such that

\[
\kappa>\max\left\{\frac{2(2c_2)^{p/q}}{c_1c_3}\cdot\frac{\|\Omega\|_\infty}{{(I_*-\delta_n)}^{1+\frac pq}},\frac{2}{c_1c_3(\pi-\alpha_0)^p}\cdot \frac{\|\Omega\|_\infty}{{I_*-\delta_n}}\right\} \quad \forall n\ge 1,\quad\mbox{and}\quad \lim_{n\rightarrow\infty}\delta_n=0.
\]
Since $\kappa>0$, Proposition \ref{nodispersed-IS-qual} tells us that for every $n\ge 1$ the set
\[
U_{I_*-\delta_n}^{\infty,I,S}=\{ \Theta^0\in \bbt^N: R(\Phi_t^{I,S}(\Theta^0))<I_*-\delta_n\mbox{ for all } t\ge 0\}
\]
has measure zero and thus the union
\[
U^{\infty,I,S}=\mathop{\bigcup}_{n=1}^\infty U_{I_*-\delta_n}^{\infty,I,S}
\]
has measure zero as well. So, for almost all initial data $\Theta^0\in \bbt^N\backslash U^{\infty,I,S}$, for each $n\ge 1$ there exists a finite time $t_n\ge 0$ such that
\[
R(\Phi^{I,S}_{t_n}(\Theta^0))\ge I_*-\delta_n,
\]
and we apply Theorem \ref{mainthm} at this time $t=t_n$ to conclude the following:
\begin{enumerate}[(a)]
\item By Theorem \ref{mainthm} (a), we have $\liminf_{t\to\infty}R(t)\ge \frac{c_3R(t_n)}{2}\ge \frac{c_3(I_*-\delta_n)}{2}$. By taking the limit $n\to \infty$ we obtain assertion (a) of Theorem \ref{themainthm-IS}.

\item By Theorem \ref{mainthm} (b), $\limsup_{t\to\infty}\theta_i(t)-\liminf_{t\to\infty}\theta_i(t)\le \sup_{t\ge t_1}\theta_i(t)-\inf_{t\ge t_1}\theta_i(t)<2\pi$, proving assertion (b) of Theorem \ref{themainthm-IS}.

\end{enumerate}
\end{proof}
Now we prepare for the proof of Theorem \ref{maincor-quant}. We begin with a generalization of Lemma \ref{lem:order-param-conc}.
\begin{lemma}\label{lem:order-param-conc-IS}
Assuming \eqref{c_7}, we have
    \begin{equation*}
m\Big\{\{\theta_i^0\}_{i=1}^N\in [-\pi,\pi]^N:~\frac 1N\sum_{i=1}^N I(\theta_i^0)\le t\Big\}\le
\left(\frac{\Gamma(1/r)}{\pi r^{1-1/r}}\left(\frac{e t}{c_5}\right)^{1/r}\right)^N,\quad t>0.
\end{equation*}

Assuming \eqref{c_4}, we have
\begin{equation*}
m\Big\{\{\theta_i^0\}_{i=1}^N\in [-\pi,\pi]^N:~\frac 1N\sum_{i=1}^N I(\theta_i^0)\le I_*-\delta\Big\}\le
\exp\left(-\frac{2N\delta^2}{\|I\|^2_{\sup}}\right),\quad \delta\in (0,I_*).
\end{equation*}
\end{lemma}
\begin{proof}
Assuming \eqref{c_7}, we can compute
\begin{align*}
    \mathbb{E}e^{-\lambda I(\theta_j^0)}=\frac{1}{2\pi}\int_{-\pi}^\pi e^{-\lambda I(\theta)}d\theta \stackrel{\eqref{c_7}}{\le}\frac{1}{2\pi}\int_{-\pi}^\pi e^{-\lambda c_5 (\pi-|\theta|)^r}d\theta\le \frac{1}{\pi}\int_0^\infty e^{-\lambda c_5\theta^r}d\theta = \frac{\Gamma(1/r)}{\pi r c_5^{1/r}}\lambda^{-\frac 1r}
    ,\quad \lambda>0,
\end{align*}
so that by Lemma \ref{lem:extreme-close} with $M=0$,
\[
\mathbb{P}\left[\frac 1N\sum_{i=1}^NI(\theta_i^0)\le t\right]=\mathbb{P}\left[\frac 1N\sum_{i=1}^N(-I(\theta_i^0))\ge -t\right]\le \left(\frac{\Gamma(1/r)}{\pi r^{1-1/r}}\left(\frac{e t}{c_5}\right)^{1/r}\right)^N.
\]

On the other hand, assume \eqref{c_4}. Then
\[
0=\frac{1}{2\pi}\int_{-\pi}^\pi S'(\theta)d\theta \stackrel{\eqref{c_4}}{\ge}\frac{1}{2\pi}\int_{-\pi}^\pi c_4(I_*-I(\theta))d\theta =c_4(I_*-R^*),
\]
so that $I_*\le R^*$, where $R^*$ is defined as in \eqref{mu-pos}:
\[\tag{\ref{mu-pos}}
R^*=\frac{1}{2\pi}\int_{-\pi}^\pi I(\theta)d\theta.
\]
The distribution $m$ means that we are independently and identically sampling $\theta_i^0$ according to the uniform Lebesgue distribution on $[-\pi,\pi]$, so $\frac 1N\sum_{j=1}^N I(\theta_j^0)$ is the sum of the independent random variables $\frac 1N I(\theta_j^0)$, $j=1,\cdots, N$. By the Hoeffding lemma, $\frac 1N I(\theta_j^0)$ is $\frac{\|I\|_{\sup}^2}{4N^2}$-subgaussian, so by independence, $\sum_{j=1}^N \frac 1N I(\theta_j^0)$ is $\frac{\|I\|_{\sup}^2}{4N}$-subgaussian. As $\mathbb{E}\sum_{j=1}^N \frac 1N I(\theta_j^0)=R^*$, we conclude
\begin{align*}
m\left\{\{\theta_i^0\}_{i=1}^N\in [-\pi,\pi]^N:~\frac 1N\sum_{i=1}^N I(\theta_i^0)\le I_*-\delta\right\}&\le
m\left\{\{\theta_i^0\}_{i=1}^N\in [-\pi,\pi]^N:~\frac 1N\sum_{i=1}^N I(\theta_i^0)\le R^*-\delta\right\}\\
&\le \exp\left(-\frac{\delta^2}{2\cdot \|I\|^2_{\sup}/4N}\right)
=\exp\left(-\frac{2N\delta^2}{\|I\|^2_{\sup}}\right).
\end{align*}
\end{proof}
Next we obtain an analogue of Proposition \ref{nodispersed}.
\begin{proposition}\label{prop:nodispersed-IS-quant}
    Consider the Winfree model \eqref{GenWinfree} with interaction functions $I$ and $S$ satisfying \eqref{c_4}, \eqref{c_5}, and \eqref{c_7}.
Let $\kappa>0$, $N$, and $\Omega=(\omega_1,\cdots,\omega_N)$ be fixed, and consider the flow $\Phi_t^{I,S}$ generated by \eqref{vf-1-IS}-\eqref{vf-2-IS}. Then, for any $0<\delta\le I_*/2$ and $T\ge 0$, the open set 
\[
U_{\delta}^{T,I,S}\coloneqq \{\Theta^0\in\bbt^N: R(\Phi_t^{I,S}(\Theta^0))<\delta \mbox{ for all }0\le t \le T\}
\]
has measure
\[
m(U_{\delta}^{T,I,S})\le\left(\max\left\{\frac{c_5\pi^r r^{r-1}}{\Gamma(1/r)^re\delta},\exp\left(\frac{2r(I_*-\delta)^2}{\|I\|^2_{\sup}}\right)\right\}+\frac{c_4c_5\pi^r r^{r}}{e\Gamma(1/r)^r(1+r/N)}\kappa (I_*-\delta)T\right)^{-N/r}.
\]
\end{proposition}
\begin{proof}
As in the proof of Proposition \ref{nodispersed}, we begin by noting that
\[
U_{\varepsilon}^{0,I,S} =\{\Theta^0\in\bbt^N: R(\Theta^0)<\varepsilon\}, \quad \varepsilon>0,
\]
and estimate its measure to be the following, via Lemma \ref{lem:order-param-conc-IS}:
\[
m(U_\varepsilon^{0,I,S})\le \left(\frac{\Gamma(1/r)}{\pi r^{1-1/r}}\left(\frac{e \varepsilon}{c_5}\right)^{1/r}\right)^N.
\]
By definition, we have
\[
\Phi_t^{I,S}(U_{\delta}^{T,I,S})\subset U_{\delta}^{0,I,S},\quad \forall~ 0\le t\le T.
\]
By the fact that the divergence of a flow determines the rate of change of volume, we have for $0\le t\le T$
\begin{equation*}
\frac{d}{dt}m\left(\Phi_t^{I,S}\left(U_{\delta}^{T,I,S}\right)\right)
=\int_{\Phi_t^{I,S}\left(U_{\delta}^{T,I,S}\right)}(\nabla\cdot F^{I,S})(\Theta)dm(\Theta)
\stackrel{\eqref{div-IS}}{\ge}c_4\kappa N \int_{\Phi_t^{I,S}\left(U_{\delta}^{T,I,S}\right)}R(I_*-R) dm(\Theta).
\end{equation*}

Since $\Phi_t^{I,S}\left(U_{\delta}^{T,I,S}\right)\subset U^{0,I,S}_\delta$,
\begin{align*}
    \frac{d}{dt}m\left(\Phi_t^{I,S}\left(U_{\delta}^{T,I,S}\right)\right)
&\ge c_4\kappa N(I_*-\delta)\int_{\Phi_t^{I,S}\left(U_{\delta}^{T,I,S}\right)}R dm(\Theta),\\
    &=c_4\kappa N(I_*-\delta)\int_{\Phi_t^{I,S}\left(U_{\delta}^{T,I,S}\right)}\int_0^\delta \mathbbm{1}_{R\ge s} dsdm(\Theta)\\
    &=c_4\kappa N(I_*-\delta)\int_0^\delta m\left(\Phi_t^{I,S}\left(U_{\delta}^{T,I,S}\right)\setminus  U^{0,I,S}_s\right) ds\\
    &\ge c_4\kappa N(I_*-\delta)\int_0^\delta \max\left\{0,m\left(\Phi_t^{I,S}\left(U_{\delta}^{T,I,S}\right)\right)-m\left( U^{0,I,S}_s\right)\right\} ds\\
    &\ge c_4\kappa N(I_*-\delta)\int_0^\delta \max\left\{0,m\left(\Phi_t^{I,S}\left(U_{\delta}^{T,I,S}\right)\right)-\left(\frac{\Gamma(1/r)}{\pi r^{1-1/r}}\left(\frac{e s}{c_5}\right)^{1/r}\right)^N\right\} ds\\
    &=c_4\kappa N(I_*-\delta)\frac{\pi^r r^{r-1}c_5}{\Gamma(1/r)^r e}\frac{1}{1+r/N}m\left(\Phi_t^{I,S}\left(U_{\delta}^{T,I,S}\right)\right)^{1+\frac rN},
\end{align*}
where in the last equality we used $m\left(\Phi_t^{I,S}\left(U_{\delta}^{T,I,S}\right)\right)\le m\left( U^{0,I,S}_\delta\right)\le \left(\frac{\Gamma(1/r)}{\pi r^{1-1/r}}\left(\frac{e \delta}{c_5}\right)^{1/r}\right)^N$.

Solving this differential inequality along with the terminal condition
\[
m\left(\Phi_T^{I,S}\left(U_{\delta}^{T,I,S}\right)\right)\le m(U^{0,I,S}_{\delta})\stackrel{\mathrm{Lemma~}\ref{lem:order-param-conc-IS}}{\le} \min\left\{\left(\frac{\Gamma(1/r)}{\pi r^{1-1/r}}\left(\frac{e\delta}{c_5}\right)^{1/r}\right)^N,\exp\left(-\frac{2N(I_*-\delta)^2}{\|I\|^2_{\sup}}\right)\right\}
\]
gives the stated inequality.
\end{proof}

We are now ready to prove Theorem \ref{maincor-quant}.
\begin{proof}[Proof of Theorem \ref{maincor-quant}]
Under the condition
\begin{equation*}\tag{\ref{eq:kappa-large-maincor-quant}}
\kappa>\max\left\{\frac{4(4c_2)^{p/q}}{c_1c_3}\cdot\frac{\|\Omega\|_\infty}{{I_*}^{1+\frac pq}},\frac{4}{c_1c_3(\pi-\alpha_0)^p}\cdot \frac{\|\Omega\|_\infty}{{I_*}}\right\}.
\end{equation*}
we have from Theorem \ref{mainthm} that
\begin{align*}
&m\left\{\{\theta_i^0\}_{i=1}^N\in [-\pi,\pi]^N:R(t)> \frac{c_3I_*}{4}~\forall t\ge T \mbox{ and }\sup_{t\ge T} \theta_i(t)-\inf_{t\ge T}\theta_i(t)<2\pi ~\forall i\in [N]\right\}\\
&\ge m\Big\{\{\theta_i^0\}_{i=1}^N\in [-\pi,\pi]^N:~R(t)\ge I_*/2\mathrm{~for~some~}t\in [0,T]\Big\}\\
&= 1-m\left(U_{I_*/2}^{T,I,S}\right)\\
&\ge 1-\left(\exp\left(\frac{rI_*^2}{2\|I\|^2_{\sup}}\right)+\frac{c_4c_5\pi^r r^{r}}{2e\Gamma(1/r)^r(1+r/N)}\kappa I_*T\right)^{-N/r}
\end{align*}
where in the last inequality we used Proposition \ref{prop:nodispersed-IS-quant} with $\delta = \frac{I_*}{2}$.

For the second assertion, assume both \eqref{eq:kappa-large-maincor-quant} and \eqref{eq:kappa-large-maincor-quant-2}. If $\|\Omega\|_\infty=0$, Theorem \ref{mainthm} applies to every initial datum with $R_0>0$, and \eqref{c_7} shows that these data have full measure. Thus the desired event has measure one. Assume henceforth $\|\Omega\|_\infty>0$, and set
	\[
	\begin{aligned}
		\varepsilon_0&=\max\left\{
		\left(\frac{2(2c_2)^{p/q}\|\Omega\|_\infty}{c_1c_3\kappa}\right)^{1/(1+p/q)},
		\frac{2\|\Omega\|_\infty}{c_1c_3(\pi-\alpha_0)^p\kappa}
		\right\},\\
		C_0&=\frac{c_5\pi^r r^{r-1}}{e\Gamma(1/r)^r}.
	\end{aligned}
	\]
	The two coupling assumptions give $\varepsilon_0<I_*/2$ and $\varepsilon_0<C_0/2$, respectively. Choose
	\[
	\varepsilon_0<\varepsilon<\min\{I_*/2,C_0/2\}.
	\]
	If the trajectory reaches $R\ge\varepsilon$ by time $T$, then the strict coupling hypothesis of Theorem \ref{mainthm} holds at that time, and the theorem gives a phase range less than $2\pi$ on $[T,\infty)$. Consequently, Proposition \ref{prop:nodispersed-IS-quant}, applied with $\delta=\varepsilon$, gives
	\[
	\begin{aligned}
		&m\left\{\Theta^0\in[-\pi,\pi]^N:
		\sup_{t\ge T}\theta_i(t)-\inf_{t\ge T}\theta_i(t)<2\pi\ \text{for all }i\in[N]\right\}\\
		&\qquad\ge1-m(U_{\varepsilon}^{T,I,S})\\
		&\qquad\ge1-C_0^{-N/r}
		\left(\varepsilon^{-1}+\frac{c_4\kappa r I_*}{2(1+r/N)}T\right)^{-N/r}.
	\end{aligned}
	\]
	Here we retained the first term in the maximum in that proposition and used $I_*-\varepsilon\ge I_*/2$. Letting $\varepsilon\downarrow\varepsilon_0$ and substituting the definitions of $\varepsilon_0$ and $C_0$ yields the second bound of Theorem \ref{maincor-quant}.
\end{proof}

\section{On the equilibria of the Winfree model}\label{sec:equilibria}
The main purpose of this section is to prove Theorem \ref{main_eq_thm}. We will develop a framework to analyze the equilibria of \eqref{Winfree}.

This is similar to a framework developed by Ha, Kim, and the author in \cite{ha2016finiteness} that is used to prove the finiteness of phase-locked states for the Kuramoto model; see Theorem \ref{GlobalFiniteness-Ku}.

\subsection{Description of equilibria of \eqref{Winfree}}\label{subsec:eq-descrip}
In the case $\omega_i=0$ for all $i\in [N]$, the ODE \eqref{Winfree_orderparam} becomes
\[
\dot{\theta}_i=-\kappa R(t)\sin\theta_i,
\]
and it is clear that the only equilibrium initial data are the ``bipolar states''
\[
\theta_i^0=0\mbox{ or }\pi\mod 2\pi,\quad i\in [N].
\]
In the general case where $\omega_i\neq 0$ for some $i$, we can expect the equilibria to be perturbations of the equilibria for the case $\omega_i=0$.

Indeed, fix $\kappa\neq 0$ and $\{\omega_i\}_{i=1}^N$, where $\omega_i\neq 0$ for some $i$. If $\{\theta_i^0\}_{i=1}^N$ is an equilibrium of \eqref{Winfree}, then, by \eqref{order_parameter} and \eqref{Winfree_orderparam}, we must necessarily have
\[
R\coloneqq \frac 1N \sum_{j=1}^N (1+\cos\theta_j^0)=1+\frac 1N \sum_{j=1}^N \cos\theta_j^0
\]
and
\[
0=\omega_i-\kappa R \sin\theta_i^0,\quad i\in [N].
\]
Clearly, we must have $R\ge \max_i \frac{|\omega_i|}{|\kappa|}$ (and \textit{a fortiori} $R> 0$). Now, for some choice $\sigma_i\in \{-1,1\}$ of signs,
\[
\sin \theta_i^0=\frac{\omega_i}{\kappa R},~ \cos \theta_i^0=\sigma_i\sqrt{1-\frac{\omega_i^2}{\kappa^2 R^2}},\quad i\in [N],
\]
or equivalently
\begin{equation}\label{solution}
\theta_i^0=
\begin{cases}
\sin^{-1}\frac{\omega_i}{\kappa R}\mod 2\pi&\mathrm{ if }~\sigma_i=1,\\
\pi-\sin^{-1}\frac{\omega_i}{\kappa R}\mod 2\pi&\mathrm{ if }~\sigma_i=-1,\\
\end{cases}
\quad i\in [N],
\end{equation}
 and
\begin{equation}\label{R-equation}
R=1+\frac 1N \sum_{j=1}^N \sigma_j\sqrt{1-\frac{\omega_j^2}{\kappa^2 R^2}}.
\end{equation}
Conversely, if \eqref{R-equation} holds with some $\sigma_i\in \{-1,1\}$, $i\in [N],$ and $R>0$, then it is easy to see that $\{\theta_i^0\}_{i=1}^N$ given as \eqref{solution} is an equilibrium of \eqref{Winfree}.

We can summarize the discussion so far into the following proposition.
\begin{proposition}\label{basic_pls}
Fix $\kappa\neq 0$ and $\{\omega_i\}_{i=1}^N$, where $\omega_i\neq 0$ for some $i$. Given any $R\in [\max_i \frac{|\omega_i|}{|\kappa|},\infty)$ and $\{\sigma_i\}_{i=1}^N\subseteq \{-1,1\}$, the initial data \eqref{solution} is an equilibrium for \eqref{Winfree} if and only if \eqref{R-equation} holds. There are no equilibria of \eqref{Winfree} other than the ones just described.
\end{proposition}
Based on the characterization of equilibria provided by Proposition \ref{basic_pls}, one can easily construct equilibria with a given structure (namely, the signs $\sigma_i$'s) using the intermediate value theorem.
\begin{proposition}\label{ivt}
Fix $\kappa\neq 0$ and $\{\omega_i\}_{i=1}^N$. Let $\{\sigma_i\}_{i=1}^N\subseteq \{-1,1\}$, with not every $\sigma_i$ being $-1$, and denote
\[
\rho_0\coloneqq 1+\frac 1N\sum_{j=1}^N \sigma_j>0.
\]
If
\begin{equation}\label{onehalfexpo}
\max_i \frac{|\omega_i|}{|\kappa|}<\frac{\rho_0^{1.5}}{4},
\end{equation}
then there exists $R\in [\frac 12\rho_0,\frac 32 \rho_0]$ such that the initial data \eqref{solution} is an equilibrium to \eqref{Winfree}.
\end{proposition}
\begin{proof}
If $\Omega=0$, the bipolar state corresponding to the prescribed signs has order parameter $R=\rho_0$, and the conclusion is immediate. Assume $\Omega\ne0$. By Proposition \ref{basic_pls}, it suffices to solve \eqref{R-equation} in $[\rho_0/2,3\rho_0/2]$. Define the function $f:[\max_i \frac{|\omega_i|}{|\kappa|},\infty)\to\bbr$ by
\[
f(r)\coloneqq r-1-\frac 1N \sum_{j=1}^N \sigma_j\sqrt{1-\frac{\omega_j^2}{\kappa^2 r^2}}, \quad r\in [\max_i \frac{|\omega_i|}{|\kappa|},\infty).
\]
It is enough to see that $f$ has a zero in $[\frac 12\rho_0,\frac 32\rho_0]$. Using the fact that $|1-\sqrt{1-x}|\le x$ for $0\le x\le 1$, we have
\begin{equation}\label{eq:f-approx}
|f(r)-r+\rho_0|\le \frac 1N \sum_{j=1}^N \left|1-\sqrt{1-\frac{\omega_j^2}{\kappa^2 r^2}}\right|\le \frac 1N \sum_{j=1}^N \frac{\omega_j^2}{\kappa^2 r^2}\le \frac{\max_i\omega_i^2}{\kappa^2r^2}\stackrel{\eqref{onehalfexpo}}{\le} \frac{\rho_0^3}{16r^2}.
\end{equation}
Observing that $[\frac 12\rho_0,\frac 32\rho_0]\subset [\max_i \frac{|\omega_i|}{|\kappa|},\infty)$,
\[
f(\frac 12 \rho_0)\stackrel{\mathclap{\eqref{eq:f-approx}}}{\le} \frac 12 \rho_0-\rho_0+\frac{\rho_0^3}{4 \rho_0^2}=-\frac 14 \rho_0<0,
\]
and
\[
f(\frac 32 \rho_0)\stackrel{\mathclap{\eqref{eq:f-approx}}}{\ge} \frac 32 \rho_0-\rho_0-\frac{\rho_0^3}{36 \rho_0^2}\ge \frac {17}{36} \rho_0>0,
\]
we conclude by the intermediate value theorem that $f$ does indeed have a zero in $[\frac 12\rho_0,\frac 32\rho_0]$.
\end{proof}
We are now ready to prove our main result on the existence of certain equilibria, namely Theorem \ref{main_eq_thm}.
\begin{proof}[Proof of Theorem \ref{main_eq_thm}]
As $N\ge \frac 2\rho$, we may find $1\le m\le N$ such that
\[
\frac{2m}{N}\le \rho <\frac{2m+2}{N}.
\]
Choose $\sigma_1=\cdots=\sigma_m=1$ and $\sigma_{m+1}=\cdots=\sigma_N=-1$. Then, under the notation of Proposition \ref{ivt}, we have $\rho_0=\frac{2m}N$, so $\rho_0\le \rho<2\rho_0$. Now
\[
\max_i \frac{|\omega_i|}{|\kappa|}<\frac{\rho^{1.5}}{16}<\frac{\rho_0^{1.5}}{4},
\]
so by Proposition \ref{ivt} there exists $R\in [\frac 12 \rho_0,\frac 32 \rho_0]\subset [\frac 14 \rho,\frac 32\rho]$ such that the initial data \eqref{solution} is an equilibrium to \eqref{Winfree}.

Moreover, it is easy to check that
\[
\frac \rho 4<\frac mN\le \frac \rho 2,
\]
while
\[
\frac{2|\omega_i|}{3|\kappa| \rho}\le\frac{|\omega_i|}{|\kappa| R}\le |\theta_i^0|\stackrel{\eqref{solution}}{=}\left|\sin^{-1}\frac{\omega_i}{\kappa R}\right|\le\frac \pi 2  \frac{|\omega_i|}{|\kappa| R}\le \frac{2\pi |\omega_i|}{|\kappa| \rho},\quad i =1,\cdots,m,
\]
and
\[
\frac{2|\omega_i|}{3|\kappa| \rho}\le\frac{|\omega_i|}{|\kappa| R}\le |\pi-\theta_i^0|\stackrel{\eqref{solution}}{=}\left|\sin^{-1}\frac{\omega_i}{\kappa R}\right|\le\frac \pi 2  \frac{|\omega_i|}{|\kappa| R}\le \frac{2\pi |\omega_i|}{|\kappa| \rho},\quad i =m+1,\cdots,N.
\]

\end{proof}

\begin{remark}\label{rem:all-negative-state}
	What happens if we take $\sigma_i\equiv -1$ in \eqref{R-equation}? If $\Omega=0$, the all-$\pi$ equilibrium has order parameter zero.
If $\|\Omega\|_\infty>0$ and $\kappa>N\|\Omega\|_\infty$, then all-negative-sign branch of \eqref{R-equation} is determined by
\[
g_\kappa(r)=1-\frac1N\sum_{j=1}^N
\sqrt{1-\frac{\omega_j^2}{\kappa^2r^2}}-r=0,
\qquad \|\Omega\|_\infty/\kappa\le r\le1.
\]
The function is strictly decreasing, with
$g_\kappa(\|\Omega\|_\infty/\kappa)\ge 1/N-\|\Omega\|_\infty/\kappa>0$ and $g_\kappa(1)<0$.
It therefore has a unique root $R_\kappa\in(\|\Omega\|_\infty/\kappa,1)$, which is the order parameter of an equilibrium.
Writing $A=N^{-1}\sum_j\omega_j^2>0$ and using
$x/2\le1-\sqrt{1-x}\le x$, we obtain
\[
\left(\frac{N^{-1}\sum_j\omega_j^2}{2\kappa^2}\right)^{1/3}
\le R_\kappa\le
\left(\frac{N^{-1}\sum_j\omega_j^2}{\kappa^2}\right)^{1/3}.
\]
In particular, $R_\kappa\to0$. Taylor expansion then gives the more precise asymptotic
\[
R_\kappa=\left(\frac{N^{-1}\sum_j\omega_j^2}{2\kappa^2}\right)^{1/3} + O_{\Omega,N}(\kappa^{-4/3}).
\]

\end{remark}


\subsection{The critical coupling strength}

Based on the characterization \eqref{solution}-\eqref{R-equation} of equilibria by Proposition \ref{basic_pls}, we can compute the critical coupling strength for the Winfree model \eqref{Winfree}. Note that the analogous computation for the Kuramoto model \eqref{Ku} has been performed in \cite{verwoerd2009computing}.

\begin{proposition}\label{prop:crit_comp}
Given $\Omega=\{\omega_i\}_{i=1}^N\neq 0$, we have
\[
\kappa_{\mathrm{c}}(\Omega)= \frac{Nu_*}{N+\sum_{j=1}^N\sqrt{1-\frac{\omega_j^2}{u_*^2}}},
\]
where $u_*\in [\|\Omega\|_\infty,\frac{2}{\sqrt{3}}\|\Omega\|_\infty]$ is the solution to
\[
N+2 \sum_{j=1}^N \sqrt{{1-\frac{\omega_j^2}{u_*^2}}}=\sum_{j=1}^N \frac{1}{\sqrt{1-\frac{\omega_j^2}{u_*^2}}}.
\]
More precisely, given $\Omega$ and $\kappa>0$, \eqref{Winfree} admits an equilibrium initial data $\Theta^0$ if and only if $\kappa\ge \kappa_{\mathrm{c}}(\Omega)$. Also,
\[
\frac{2N\|\Omega\|_\infty}{4N-1}\le\frac{N\|\Omega\|_\infty}{(2N-1+c_N)\sqrt{1-c_N^2}}\le\kappa_{\mathrm{c}}(\Omega)\le \frac{4}{3\sqrt{3}}\|\Omega\|_\infty,
\]
where $c_N=(\sqrt{4N^2-4N+9}-2N+1)/4$. The equality case for the second inequality is when $\omega_i=0$ for all but at most one $i$, and the equality case for the third inequality is when $|\omega_i|=\|\Omega\|_\infty$ for all $i$.
\end{proposition}
\begin{proof}
We first observe that \eqref{Winfree} has an equilibrium if and only if it has an equilibrium of the form \eqref{solution}-\eqref{R-equation} with $\sigma_i=1$ for all $i$. The if direction is obvious; to see the only if direction, assume \eqref{Winfree} has an equilibrium, which by Proposition \ref{basic_pls} must be of the form \eqref{solution}-\eqref{R-equation} for some choice of $R=R_1\le 2$ and $\sigma_i$'s. It follows that
\[
1+\frac 1N\sum_{j=1}^N\sqrt{1-\frac{\omega_j^2}{\kappa^2R_1^2}}-R_1\ge 1+\frac 1N\sum_{j=1}^N\sigma_j\sqrt{1-\frac{\omega_j^2}{\kappa^2R_1^2}}-R_1=0.
\]
As $1+\frac 1N\sum_{j=1}^N\sqrt{1-\frac{\omega_j^2}{\kappa^2\cdot 2^2}}-2\le 0$, it follows from the intermediate value theorem that there is $R_2\in [R_1,2]$ such that \eqref{R-equation} holds for $R=R_2$ and $\sigma_i=1$ for all $i$. This completes the proof of our observation.

By the transformation $u=\kappa R$, we see that \eqref{Winfree} has an equilibrium if and only if
\begin{equation*}
\frac 1\kappa=\frac{1+\frac 1N \sum_j\sqrt{1-\omega_j^2/u^2}}{u}\quad \mathrm{for~some~}u\in [\|\Omega\|_\infty,\infty).
\end{equation*}
The right-hand side converges to $0$ as $u\to\infty$ and is a unimodal function of $u$, increasing on $[\|\Omega\|_\infty,u_*]$ and decreasing on $[u_*,\infty)$, because it has derivative
\[
\frac{1}{u^2}\left(-1-\frac 2N\sum_{j=1}^N \sqrt{{1-\frac{\omega_j^2}{u^2}}}+ \frac 1N \sum_{j=1}^N \frac{1}{\sqrt{1-\frac{\omega_j^2}{u^2}}}\right),
\]
the term in parentheses being a decreasing function of $u$ and taking the value $\infty$ at $u=\|\Omega\|_\infty$ and taking a value $\le 0$ at $u=\frac{2}{\sqrt{3}}\|\Omega\|_\infty$. This completes the proof that \eqref{Winfree} has an equilibrium if and only if $\kappa\ge \kappa_{\mathrm{c}}(\Omega)$.

We have
\[
\frac 1{\kappa_{\mathrm{c}}(\Omega)}=\max_{u\in [\|\Omega\|_\infty,\infty)}\frac{1+\frac 1N \sum_j\sqrt{1-\omega_j^2/u^2}}{u}.
\]
Given $\|\Omega\|_\infty$, the function to be maximized is pointwise minimized when $|\omega_j|=\|\Omega\|_\infty$ for all $i$ and is maximized when $\omega_i=0$ for all but one $i$. This gives our equality cases. It is easy to compute that in these cases $u_*=\frac{2}{\sqrt{3}}\|\Omega\|_\infty$ and
\[
u_*=\frac{4\|\Omega\|_\infty}{\sqrt{5-2N+\sqrt{4N^2-4N+9}}\sqrt{3+2N-\sqrt{4N^2-4N+9}}}=\frac{\|\Omega\|_\infty}{\sqrt{1-c_N^2}},
\]
respectively, and that $\kappa_{\mathrm{c}}(\Omega)$ computes to the above values. The first inequality follows from
\begin{align*}
\frac 1{\kappa_{\mathrm{c}}(\Omega)}&=\max_{u\in [\|\Omega\|_\infty,\infty)}\frac{1+\frac 1N \sum_j\sqrt{1-\omega_j^2/u^2}}{u}\le \max_{u\in [\|\Omega\|_\infty,\infty)}\frac{1+\frac {N-1}N +\frac 1N\sqrt{1-\|\Omega\|_\infty^2/u^2}}{u}\\
&\le \max_{u\in [\|\Omega\|_\infty,\infty)}\frac {2N-1}{Nu}+\max_{u\in [\|\Omega\|_\infty,\infty)}\frac 1{N\|\Omega\|_\infty}\sqrt{\frac{\|\Omega\|_\infty^2}{u^2}\left(1-\frac{\|\Omega\|_\infty^2}{u^2}\right)}\\
&=\frac{4N-1}{2N\|\Omega\|_\infty}.
\end{align*}

\end{proof}

\subsection{A polynomial description for the equilibria} 

In this subsection, we prove Theorem \ref{GlobalFiniteness}. We provide a method of understanding the behavior of equilibria of the sinusoidal Winfree model \eqref{Winfree} in terms of a single univariate polynomial, which we will denote by $W(\cdot)$. We have seen in subsection \ref{subsec:eq-descrip} that the zero of an algebraic function serves as the order parameter of an equilibrium. If we multiply these functions over all the different signatures $\sigma_i$, this will happen to give us a rational function, with the property that its roots over $[\frac{\|\Omega\|_\infty}{|\kappa|},\infty)$ are precisely the order parameters of the possible equilibria; we will let $W(\cdot)$ be the numerator of this rational function. With some careful analysis, it will turn out that the number of possible different configurations corresponding to a root of $W(\cdot)$ is bounded by the order of the zero with respect to $W(\cdot)$, and by bounding the degree of $W(\cdot)$ we have the bound stated in Theorem \ref{GlobalFiniteness} on the total number of equilibria.

We remark that this polynomial description was first developed by Ha, Kim, and the author in \cite{ha2016finiteness}, where we proved the following theorem for the Kuramoto model:
\begin{theorem}[{\cite[Theorem 6.1]{ha2016finiteness}}]\label{GlobalFiniteness-Ku}
For fixed data $(N,\{\Omega_i\}_{i=1}^N)$ with $\kappa>0$, there are at most $2^N$ phase-locked states (i.e., relative equilibria) of the  Kuramoto model \eqref{Ku} with positive order parameter, up to modulo $2\pi$ shifts and $\mathrm{U}(1)$ symmetry.
\end{theorem}
The proof method of Theorem \ref{GlobalFiniteness-Ku} presented in \cite{ha2016finiteness} is similar to the proof method of Theorem \ref{GlobalFiniteness} that we present here. We also remark that this proof works only when the Winfree model \eqref{Winfree} has the given specific sinusoidal interaction functions. It may also work (with worse results) when the interaction functions are trigonometric polynomials, or more generally when $I$ and $S$ are algebraically dependent analytic functions. However, it is likely that Theorem \ref{GlobalFiniteness} will fail when we require less than analytic regularity on the interaction functions, say if we only require them to belong in the class $C^2$.

It is trivial that Theorem \ref{GlobalFiniteness} holds when $\Omega=0$ (recall that in the beginning of subsection \ref{subsec:eq-descrip} we fully classified the equilibria into the bipolar states), so we will assume that $\omega_i\neq 0$ for some $i$.

To prove Theorem \ref{GlobalFiniteness}, let us begin by observing the following algebraic result, whose proof we omit.

\begin{lemma}\label{L6.1}
Let $x_1,\ldots,x_{N+1}$ be real variables. Then, the polynomial $S$ defined by
\[
S(x_1,\ldots,x_{N+1})\coloneqq \prod_{(\sigma_1,\ldots,\sigma_N)\in\{-1,1\}^N}\left(\sum_{j=1}^{N}\sigma_{j}x_j+x_{N+1}\right)
\]
is a homogeneous polynomial of degree $2^{N-1}$ in $x_1^2,\ldots,x_{N+1}^2$.
\end{lemma}

For $\Sigma=(\sigma_1,\ldots,\sigma_N)\in\{-1,1\}^N$, we introduce the auxiliary function
\[
P(r,\Sigma) \coloneqq 1+\frac 1N\sum_{j=1}^{N} \sigma_{j}\sqrt{1-\left(\frac{\omega_j}{\kappa r}\right)^2}-r,\quad r\ge\frac{\|\Omega\|_\infty}{|\kappa|}>0.
\]
This is the function from \eqref{R-equation} whose zeros were the order parameters of equilibria. We denote their normalized product
\[
W(r) \coloneqq r^{2^N}\prod_{\Sigma \in\{-1,1\}^N}P(r,\Sigma).
\]
Then, using Lemma \ref{L6.1} with
\[
x_j=\frac 1N \sqrt{1-\left(\frac{\omega_j}{\kappa r}\right)^2},~j\in [N],\quad x_{N+1}=1-r,
\]
we see that $\prod_{\Sigma \in\{-1,1\}^N}P(r,\Sigma)$ is a polynomial of homogeneous degree $2^{N-1}$ in $\frac 1{N^2} \left(1-\left(\frac{\omega_j}{\kappa r}\right)^2\right)$, $j=1,\cdots, N$, and $(1-r)^2$, so $W(r)$ is a polynomial of degree at most $2^{N+1}$ in $r$. It is not hard to see that $W(\cdot)$ has degree exactly $2^{N+1}$.

The expression for $W$ defines a polynomial on $\mathbb R$ by algebraic continuation. The equilibrium correspondence is asserted only on the physically admissible interval: for $r\in[\|\Omega\|_\infty/|\kappa|,2]$,
\begin{align*}
\begin{aligned}
& \text{$r$~~is the order parameter of some equilibrium $\Theta$} \\
& \hspace{1cm} \Longleftrightarrow \quad  P(r, \Sigma)=0~\text{for some}~\Sigma \in\{-1,1\}^N \quad \Longleftrightarrow \quad  W(r)=0.
\end{aligned}
\end{align*}
In fact, this correspondence can be made more precise, i.e., the number of distinct equilibria with order parameter $r$ is bounded above by the multiplicity of $r$ as a zero of $W$.

\begin{proposition}\label{correspondence}
Suppose that $R \in \Big [ \frac{ \|\Omega\|_\infty}{|\kappa|},2 \Big]$ is the order parameter of $m$ distinct (modulo $2\pi$) equilibria of the system \eqref{Winfree}. Then $R$ is a zero of $W$ of multiplicity at least $m$.
\end{proposition}
\begin{proof}
\noindent $\bullet$~Case A: $\displaystyle R> \frac{\|\Omega\|_\infty}{|\kappa|} $.
In this case
\[
\frac{\|\Omega\|_\infty}{|\kappa| R}<1,
\]
so for each of the $m$ equilibria there exists a unique $\Sigma \in\{-1,1\}^N$ such that the equilibrium is described as in \eqref{solution}. In other words, there are $m$ distinct $\Sigma_1,\ldots,\Sigma_m\in\{-1,1\}^N$ satisfying $P(R,\Sigma_i)=0$, and since $P(r,\Sigma_i)$ is differentiable at $r=R$, we easily see that $R$ is a zero of $W(r)$ of multiplicity at least $m$.

\noindent $\bullet$~Case B: $\displaystyle R= \frac{\|\Omega\|_\infty}{|\kappa|} $.

By rearranging the indices we may assume, for some $1\le a\le N$,
\[
|\omega_i|=\|\Omega\|_\infty \mathrm{~for~} i=1,\cdots,a,\quad |\omega_i|<\|\Omega\|_\infty \mathrm{~for~} i=a+1,\cdots,N.
\]
For $\Sigma=(\sigma_1,\ldots,\sigma_N)$, define the class
\[
[\Sigma]=\{(\tau_1,\ldots,\tau_a,\sigma_{a+1},\ldots,\sigma_N):\tau_1,\ldots,\tau_a\in\{-1,1\}\}\subseteq \{-1,1\}^N.
\]
It is clear that these classes partition $\{-1,1\}^N$. As in Case A, for each equilibrium $\Theta^*$ there exists a unique class $[\Sigma]$, all of whose elements along with $R$ give the equilibrium $\Theta^*$ via \eqref{solution}. Since there are $m$ distinct equilibria with order parameter $R$, there are $m$ distinct classes $[\Sigma_1],\ldots,[\Sigma_m]$, with each element $\Sigma$ of each class $[\Sigma_i]$ satisfying 
 \[ P(R,\Sigma)=0. \]
For $\Sigma=(\sigma_1,\ldots,\sigma_N)\in\{-1,1\}^N$, define
\[
Q(r,[\Sigma]) \coloneqq \prod_{\tau_1,\ldots,\tau_a\in\{-1,1\}}P(r,(\tau_1,\ldots,\tau_a,\sigma_{a+1},\ldots,\sigma_N))=\prod_{T\in[\Sigma]}P(r,T).
\]
This is clearly well-defined, and
\[
Q(R,[\Sigma_1])=\cdots=Q(R,[\Sigma_m])=0.
\]
while again by Lemma \ref{L6.1} we see that $Q(r,[\Sigma_i])$ is a polynomial of 
\[
\frac 1N\left(1-\left(\frac{\omega_1}{\kappa r}\right)^2\right),\cdots,\frac 1N\left(1-\left(\frac{\omega_a}{\kappa r}\right)^2\right),\mbox{ and }\left(1+\frac 1N\sum_{j=a+1}^N\sigma_{j}\sqrt{1-\left(\frac{\omega_j}{\kappa r}\right)^2}-r\right)^2.
\]
Hence, each $Q(r,[\Sigma_i])$ is differentiable at $\|\Omega\|_\infty/|\kappa|$, and we again see that $W(r)$ has a zero of multiplicity at least $m$ at $R$.
\end{proof}

Now, based on the fact that $W$ is a polynomial of degree $2^{N+1}$ in $r$, Theorem \ref{GlobalFiniteness} now follows from Proposition \ref{correspondence}.

Given that the polynomial $W(r)$ encodes the equilibria of \eqref{Winfree}, it seems likely that $W$ carries some physical significance.

\begin{question}
What is the physical significance of the polynomial $W(r)$ to the model \eqref{Winfree}, besides the fact that its zeros encode the order parameters of the equilibria?
\end{question}



\appendix

\section{Proof of inequality \eqref{technical2}}\label{app:ineq}

We need to verify
\[
K_c\ge\frac{1}{\rho \sqrt{\mu(2-\mu)}},\quad 0<R_0\le 2.
\]

For $1\le R_0\le 2$, this is just the inequality
\[
2(2-R_0)+\frac{4}{3\sqrt{3}}(R_0-1)\ge \frac{16}{\left(3R_0-3+\sqrt{R_0^2-2R_0+9}\right)\sqrt{3+R_0-\sqrt{R_0^2-2R_0+9}}\sqrt{5-R_0+\sqrt{R_0^2-2R_0+9}}},
\]
and since the right-hand side is $2$ when $R_0=1$ and $\frac{4}{3\sqrt{3}}$ when $R_0=2$, it is enough to see that the right-hand side is convex. Indeed, its second derivative with respect to $R_0$ is
\[
\frac{1024 \left((R_0-1)(-3R_0^4+12R_0^3-32R_0^2+40R_0+99)+(3R_0^4-12R_0^3+20R_0^2-16R_0+33)\sqrt{R_0^2-2 R_0+9}\right)}{\sqrt{R_0^2-2 R_0+9} \left(\left(-\sqrt{R_0^2-2 R_0+9}+R_0+3\right)
\left(\sqrt{R_0^2-2 R_0+9}-R_0+5\right)\right)^{5/2} \left(\sqrt{R_0^2-2 R_0+9}+3 R_0-3\right)^3}
\]
which is easily seen to be positive when $1\le R_0\le 2$, as each term in brackets in the above expression is positive.

For $0<R_0\le 1$, the inequality is equivalent to
\[
1\ge \frac{\left(3-3R_0+\sqrt{R_0^2-2R_0+9}\right)\sqrt{3+R_0+\sqrt{R_0^2-2R_0+9}}}{2\sqrt{2}(2-R_0)\sqrt{5-R_0+\sqrt{R_0^2-2R_0+9}}},
\]
and since the right-hand side is 1 when $R_0=1$, it is enough to see that the right-hand side is increasing in $R_0$. Indeed, its derivative with respect to $R_0$ is
\[
\frac{\sqrt{2} \left(3-3R_0+\sqrt{R_0^2-2 R_0+9}\right)}{(2-R_0) \left(\sqrt{R_0^2-2
R_0+9}-R_0+5\right)^{3/2} \sqrt{\sqrt{R_0^2-2 R_0+9}+R_0+3}}>0,\quad 0<R_0<1.
\]

\section*{Acknowledgments}
Partially supported by the Korea Foundation for Advanced Studies. Part of this work was conducted under the auspices of the Hyperbolic and Kinetic Equations research group at Seoul National University in the academic year 2020-2021. The author thanks Professor Myeongju Kang and Bora Moon for helpful discussions.\\
	\indent AI models consulted were OpenAI's ChatGPT-5.6 Sol Pro, ChatGPT-6 Astra, and ChatGPT-6.2 Astra, as well as Anthropic’s Claude Opus 5 and Claude Fable 5.1. They were used to the following effect:
	\begin{itemize}
		\item Improvement of the statement of Theorem \ref{themainthm} to identify the limiting equilibrium; previously, the second arXiv version \url{https://arxiv.org/abs/2601.01203v2} only asserted the existence of a limit.
		\item Creation of Remark \ref{rem:all-negative-state}, leading to an improvement in Remark \ref{rem:exchange-order}.
		\item Simplification and improvement of the statement and proof of Proposition \ref{bootstrap}.
		\item Correction of Theorem \ref{thm:pc2} from its previous overstatement in the second arXiv version.
		\item Creation of Remark \ref{rem:pc-sup}, which disproves Conjecture 10 and 11(b) of the second arXiv version.
		\item Proofreading and presentation improvements.
	\end{itemize}
	The author takes full responsibility for the content of this paper.
	
	\bibliographystyle{myalpha}
	\bibliography{bib}
\end{document}